\pdfoutput = 1
\documentclass[leqno]{jnsao}

\usepackage[utf8]{inputenc}

\usepackage[nameinlink, capitalise]{cleveref}

\usepackage{tkz-euclide}
\usepackage{multicol}

\usepackage{enumitem}
\newlist{assumption}{enumerate}{1}
\setlist[assumption]{label=(\textsc{a}\arabic*)}
\crefname{assumptioni}{Assumption}{Assumptions}

\usepackage{booktabs}
\usepackage[exponent-product=\cdot]{siunitx}

\newcommand{\R}{\mathbb{R}}
\newcommand{\N}{\mathbb{N}}
\newcommand{\1}{\mathbb{1}}
\newcommand{\Linop}{\mathbb{L}}

\DeclareMathOperator{\dive}{\mathrm{div}}

\DeclareMathOperator{\meas}{\mathrm{meas}}
\DeclareMathOperator{\Proj}{\mathrm{Proj}}

\newcommand{\dx}{\,\mathrm{d}x}

\newcommand{\dH}{\,\mathrm{d}\mathcal{H}}
\newcommand{\Ha}{\mathcal{H}}
\newcommand{\Cu}{\mathcal{C}}

\newcommand{\norm}[1]{\|#1\|}

\renewcommand{\epsilon}{\varepsilon}

\providecommand{\orcid}[1]{\mbox{\scshape\sffamily orcid:}\,\href{https://orcid.org/#1}{\detokenize{#1}}}

\title{Numerical analysis of a nonsmooth quasilinear elliptic control problem: II. Finite element discretization and error estimates}
\author{Christian Clason\thanks{Faculty of Mathematics, Universit\"at Duisburg-Essen, 45117 Essen, Germany; \emph{current address:} Institute of Mathematics and Scientific Computing, University of Graz, Heinrichstrasse 36, 8010 Graz, Austria (\email{c.clason@uni-graz.at}, \orcid{0000-0002-9948-8426})}
    \and Vu Huu Nhu\thanks{Faculty of Fundamental Sciences, PHENIKAA University, Yen Nghia, Ha Dong, Hanoi 12116, Vietnam (\email{nhu.vuhuu@phenikaa-uni.edu.vn}, \orcid{0000-0003-4279-3937})}
    \and Arnd Rösch\thanks{Faculty of Mathematics, University of Duisburg-Essen, Thea-Leymann-Strasse 9, 45127 Essen, Germany
        (\email{arnd.roesch@uni-due.de})
    }
}
\shortauthor{Clason, Nhu, Rösch}
\manuscripteprinttype{arxiv}
\manuscripteprint{2303.03060v2}
\manuscriptlicense{}
\manuscriptcopyright{}
\date{2024-02-22}
\begin{document}
\maketitle

\begin{abstract}
    In this paper, we carry out the numerical analysis of a nonsmooth quasilinear elliptic optimal control problem, where the coefficient in the divergence term of the corresponding state equation is not differentiable with respect to the state variable. Despite the lack of differentiability of the nonlinearity in the quasilinear elliptic equation, the corresponding control-to-state operator is of class $C^1$ but not of class $C^2$. Analogously, the discrete control-to-state operators associated with the approximated control problems are proven to be of class $C^1$ only. By using an explicit second-order sufficient optimality condition, we prove a priori error estimates for a variational approximation, a piecewise constant approximation, and a continuous piecewise linear approximation of the continuous optimal control problem. The numerical tests confirm these error estimates.

    \medskip

    \noindent\textcolor{structure}{Key words}\quad Optimal control, nonsmooth optimization, quasilinear elliptic equation, piecewise differentiable function, sufficient optimality condition, error estimate, finite element approximation
\end{abstract}

\section{Introduction}
We investigate the nonsmooth quasilinear elliptic optimal control problem
\begin{equation}
    \label{eq:P}
    \tag{P}
    \left\{
        \begin{aligned}
            \min_{u\in L^\infty(\Omega)} & j(u) := \int_\Omega L(x,y_u(x)) \dx + \frac{\nu}{2} \norm{u}_{L^2(\Omega)}^2 \\
            \text{s.t.} \quad &-\dive [(b + a(y_u)) \nabla y_u] = u \quad \text{in } \Omega, \quad y_u = 0 \, \text{on } \partial\Omega, \\
            & \alpha \leq u(x) \leq \beta  \quad \text{a.e. } x \in \Omega.
        \end{aligned}
    \right.
\end{equation}
Here $\Omega$ is a two-dimensional bounded, convex and polygonal domain; $L: \Omega \times \R \to\R$ is a Carath\'{e}odory function that is of class $C^2$ with respect to (w.r.t.) the second variable; $b: \overline \Omega \to \R$ is a Lipschitz continuous function;  $a: \R \to \R$ is Lipschitz continuous but not differentiable; and constants $\alpha, \beta, \nu \in \R$ satisfying $\beta > \alpha$ and $\nu > 0$.  We refer to \cref{sec:assumption} for the precise assumptions on the data of \eqref{eq:P}.

The state equation in the optimal control problem \eqref{eq:P} occurs, for instance, in models of heat conduction in which the coefficient in the divergence term acts as the heat conductivity and is a function of two variables: the temperature $y$ variable and the spatial coordinate $x$ variable; see, e.g. \cite{Bejan2013,ZeldovichRaizer1966}.
When the data belong to class $C^2$, the numerical analysis of the discrete approximation of  optimal control problems governed by such state equations were studied by Casas et al. in \cite{CasasTroltzsch2011,CasasTroltzsch2012} for distributed control  and in \cite{CasasDhamo2012} for Neumann control.

Let us briefly comment on other works concerning the error analysis of optimal control problems governed by partial differential equations (PDEs), in particular by elliptic PDEs. For control-constrained elliptic problems, we refer to the early papers \cite{Falk1973,Geveci1979} for linear elliptic control problems; to \cite{AradaCasasTroltzsch2002,CasasMateosTroltzsch2005} for semilinear elliptic problems. For state-constrained control problems, we mention only the recent contributions \cite{DeckelnickHinze2007,NeitzelPfeffererRosch2015} and refer to the survey paper \cite{HinzeTroltzsch2010} for further references. Although the error analysis for smooth PDE-constrained problems has been intensively investigated, there are very few contributions  on this topic for nonsmooth
PDE-constrained optimal control. Here we want to mention the work \cite{MaritaRosch2021} concerning error analysis for optimal control of a coupled PDE--ODE system, where the nonsmooth nonlinearity acts on a semilinear ODE.
For works related to optimal control of obstacle problems, we refer to \cite{MeyerThoma2013,ChristofMeyer2018} and the references therein.
Based on a quadratic growth condition, a priori error estimates were established in \cite{MaritaRosch2021,MeyerThoma2013}.
To the authors’ best knowledge, this is the first work that exploits a second-order sufficient optimality condition to show a priori error estimates for the discretization of optimal control problems governed by nonsmooth PDEs.

In this paper, our main aim is to derive the convergence analysis and error estimates of the discretization of \eqref{eq:P} under the explicit second-order sufficient optimality conditions established in \cite{ClasonNhuRosch_levelset} (first derived in general form in \cite{ClasonNhuRosch_os2nd}).
As the coefficient $a$ in the state equation is Lipschitz continuous but not differentiable, there are two major difficulties in deriving the error analysis. The first issue arises in studying error estimates of the discretization of the adjoint state equation; the other is the lack of the second-order differentiability of the  cost functional.
We therefore cannot apply an abstract theorem on error estimates shown in \cite{CasasTroltzsch2012} and in \cite{CasasMateosRosch2021}. To deal with the first issue, we introduce the function $Z_{y, \hat y}$ defined in \eqref{eq:Z-func} that measures pointwise  the difference between the gradients of the superposition mappings of $a$ associated with two distinct states $y$ and $\hat y$.
This allows us to derive thereafter both $L^2$- and an $H^1_0$- error estimates for the approximation of the adjoint state equation.
Let us emphasize that the solutions to the linearized state equation can only be shown to be in a $W^{1,p}$-space (see \cref{rem:regularity-derivative-S}) and thus we cannot directly employ the standard duality argument based on the Aubin--Nitsche trick; see, e.g. \cite{CasasMateos2002CAM,ReyesDhamo2016,BrennerScott2008}. In order to derive $L^2$-error estimates, we will instead introduce an \emph{adjusted} linearized state equation in \eqref{eq:adjusted-deri} below. Under the assumption that the jump functional $\Sigma$ defined in \eqref{eq:E-functional} to be finite at the optimal state $\bar y$, we show that the solutions to this adjusted equation belong to a fractional Sobolev space of order greater than $1$. We can then apply standard interpolation error estimates to obtain the desired result.
For handling the second issue, we will exploit an structural assumption on the optimal state and employ an explicit formula of a second-order generalized  derivative  of the objective functional, formulated in \cite{ClasonNhuRosch_levelset}. Based on the second-order sufficient optimality conditions for \eqref{eq:P} from \cite{ClasonNhuRosch_os2nd,ClasonNhuRosch_levelset}, we then prove general error estimates for variational, piecewise constant, and  continuous piecewise linear approximations of the optimal control, which generalize those of Theorem 2.14 in \cite{CasasTroltzsch2012} and of of Lemma 5.2 in \cite{CasasMateosRosch2021}, and can be applied for the case where the cost functional $j$ is of class $C^1$ but not necessarily $C^2$; see \cref{thm:error-esti-general}.

The plan of the paper is as follows. This section ends with our notation. In the next section, we make the assumptions for \eqref{eq:P} and provide some preliminary results from \cite{ClasonNhuRosch_os2nd}  and \cite{ClasonNhuRosch_levelset}.
\cref{sec:state-equ-error} is devoted to the numerical approximation of the state equation by finite elements and the local well-posedness and differentiability of the discrete counterpart of the control-to-state operator. In \cref{sec:adjoint-state-equ-error}, the error analysis of the adjoint state equation is investigated. Finally, the main results of the paper are presented in \cref{sec:P-descretized}. There, the convergence and error estimates of local minima of discrete optimal control problems are, respectively, stated in \cref{sec:convergence} and \cref{sec:error-estimate}.
The numerical tests illustrating the obtained results are given in \cref{sec:numerical-confirmation}.
Finally, the verification of a structural assumption and a computation of the jump functional for a specific situation are carried out in \cref{sec:structural-jump-func}, while the regularity of solutions in a fractional Sobolev space to the adjusted linearized state equation is shown in \cref{sec:linearized-regularity}.

\medskip

\paragraph*{Notation.}
We denote by $B_X(u,\rho)$ and $\overline B_X(u,\rho)$ the open and closed balls in a Banach space $X$ of radius $\rho >0$ centered at $u \in X$, respectively.
For Banach spaces $X$ and $Y$, the notation $X \hookrightarrow (\Subset) Y$ is understood that $X$ is continuously (compactly) embedded in $Y$. Let $X$ be a Banach space with its dual $X^*$, the symbol $\langle\cdot, \cdot \rangle_{X^*, X}$ stands for the dual product of $X$ and $X^*$.
For a given function $g:\overline\Omega\to \R$ and a subset $A \subset \R$, $\{ g \in A \}$ denotes the set of all points $x \in \overline \Omega$ for which $g(x) \in A$. For functions $g_1,g_2$ and subsets $A_1, A_2 \subset \R$, we set $\{ g_1 \in A_1, g_2 \in A_2 \} := \{ g_1 \in A_1 \} \cap \{ g_2 \in A_2\}$.
For any set $\omega \subset \overline \Omega$, we denote by $\1_{\omega}$
the characteristic function
of $\omega$, i.e., $\1_\omega(x) = 1$ if $x \in \omega$ and $\1_\omega(x) =0$ otherwise.
We write the symbol $C$ for a generic positive constant, which may be different at different places of occurrence and the notation, e.g. $C_\xi$ for a constant depending only on the parameter $\xi$.
For a measurable two-dimensional subset $M$, by $\meas_{\R^2}(M)$, we denote the two-dimensional Lebesgue measure of $M$.
Finally, the symbol $\Ha^1$ denotes the one-dimensional Hausdorff measure on $\R^2$ that is scaled as in  \cite{Evans1992}, Def.~2.1.

\section{Main assumptions and preliminary results}\label{sec:assumption}
In this section we first  present assumptions, which will be used in the whole paper, and then state some preliminary results on the state equation, the adjoint state equation, as well as the first-order and explicit second-order sufficient optimality conditions from \cite{ClasonNhuRosch_os2nd} and \cite{ClasonNhuRosch_levelset}.

\medskip

We first address the salient point, which is the structure of the nondifferentiable nonlinearity $a$. In this work, we assume that $a$ is defined by
\begin{equation}
    \label{eq:PC1-rep}
    a(t) := \1_{(-\infty, \bar t]}(t) a_0(t) + \1_{(\bar t, \infty)} a_1(t)  \quad \text{for all } t \in \R,
\end{equation}
for a given number $\bar t \in \R$ and given functions $a_0 \in C^2((-\infty, \bar t])$ and $a_1 \in C^2([\bar t, \infty))$ with $a_{0}(\bar t) = a_{1}(\bar t)$.
Obviously, $a$ is Lipschitz continuous on $\R$ and twice continuously differentiable on  $(-\infty, \bar t) \cup (\bar t, \infty)$, but not even of class $C^1$ in general.
However, $a$ is directionally differentiable and its directional derivative for all $t\in \R$ in direction $s\in \R$ is given by
\begin{equation} \label{eq:directional-der-a1}
    a'(t; s) = \1_{(-\infty, \bar t)}(t)a_0'(t)s + \1_{( \bar t, \infty)}(t)a_1'(t)s
    + \1_{\{\bar t\}}(t) [\1_{(0, \infty)}(s) a_{1}'(\bar t)s + \1_{(-\infty, 0)}(s) a_{0}'(\bar t)s].
\end{equation}
\begin{remark}
    Let us emphasize that the results and the underlying analysis in this paper can be applied to the situation in which the function $a$ is continuous and is twice continuously differentiable on finitely many intervals (i.e., a finitely $PC^2$ function; see \cite{ClasonNhuRosch_os2nd} for a precise definition).
    However, in order to keep the presentation concise and to be able to focus on the main arguments, we restrict the presentation to the simplest such situation given by \eqref{eq:PC1-rep}.
\end{remark}

The following assumptions shall hold throughout the following.
\begin{assumption}
    \item \label{ass:domain}
        $\Omega \subset \R^2$ is an open bounded convex polygonal.
    \item \label{ass:b_func}
        The Lipschitz continuous function $b: \overline\Omega \to \R$ satisfies  $b(x) \geq \underline{b} > 0$ for all $x \in \overline\Omega$.
    \item \label{ass:PC1-func}
        $a: \R \to \R$ is nonnegative and given by \eqref{eq:PC1-rep}.
    \item \label{ass:cost_func}
        $L: \Omega \times \R \to \R$ is a Carath\'{e}odory function that is of class $C^2$ w.r.t. the second variable with $L(\cdot,0) \in L^1(\Omega)$. Besides, for any $M>0$, there exist $C_M >0$ and $\psi_M \in L^{\bar p}(\Omega)$ ($\bar p > 2$) such that
        $|\frac{\partial L}{\partial y}(x,y)|  \leq \psi_M(x) \, \text{and} \, |\frac{\partial^2 L}{\partial y^2}(x,y)| \leq C_M$
        for all $y \in \R$ with $|y| \leq M$, and a.e. $x \in \Omega$.
\end{assumption}
In the remainder of this subsection, we state some known results for the state equation, the adjoint state equation, and the optimality conditions for  \eqref{eq:P}; see, e.g. \cite{ClasonNhuRosch_os2nd} and \cite{ClasonNhuRosch_levelset}.
Let us first study the state equation
\begin{equation} \label{eq:state}
    -\dive [(b + a(y) )\nabla y ] = u \, \text{in } \Omega, \quad y =0 \, \text{on } \partial\Omega.
\end{equation}
\begin{theorem}[{\cite{ClasonNhuRosch_levelset}, Thm.~3.2 and cf. \cite{ClasonNhuRosch_os2nd}, Thms.~3.1 and 3.5}]
    \label{thm:control2state-oper}
    Assume that \crefrange{ass:domain}{ass:PC1-func} are verified. Then, the control-to-state mapping $S: W^{-1,p}(\Omega) \ni u \mapsto y_u \in W^{1,p}_0(\Omega)$ is of class $C^1$, where $y_u$ is the unique solution to \eqref{eq:state}. Moreover, for any $u,v \in W^{-1,p}(\Omega)$ with $p > 2$ and $y_u:= S(u)$, $z_v:= S'(u)v$ uniquely satisfies
    \begin{equation} \label{eq:deri}
        -\dive [(b + a(y_u)) \nabla z_v + \1_{\{ y_u \neq \bar t \}}a'(y_u) z_v\nabla y_u ]  = v \, \text{in } \Omega,  \quad z_v =0 \, \text{on } \partial\Omega.
    \end{equation}
    Moreover, a number $p_* >2$ exists and satisfies that for any $p \in [2,p_*)$ and for any bounded set $U \subset L^p(\Omega)$, there hold $S(u) \in W^{1,p}_0(\Omega) \cap W^{2,p}(\Omega)$ and
    $\norm{S(u)}_{W^{2,p}(\Omega)} \leq C_U$.
\end{theorem}

We now investigate the adjoint state equation
\begin{equation} \label{eq:adjoint-state}
    -\dive [(b + a(y_u)) \nabla \varphi]  + \1_{\{ y_u \neq \bar t \}}a'(y_u) \nabla y_u \cdot \nabla \varphi  = v \, \text{in } \Omega,  \quad \varphi =0 \, \text{on } \partial\Omega
\end{equation}
for $u \in W^{-1,p}(\Omega)$, $p >2$, $v\in H^{-1}(\Omega)$, and $y_u:= S(u)$.
\begin{theorem}[{\cite{ClasonNhuRosch_levelset}, Thm.~3.3 and cf. \cite{ClasonNhuRosch_os2nd}, Lem.~4.1}]
    \label{thm:adjoint-equation}
    Assume that \crefrange{ass:domain}{ass:PC1-func} are satisfied. Let $p, q> 2$ be arbitrary. Then, for any $u \in W^{-1,p}(\Omega), v \in H^{-1}(\Omega)$, there exists a unique $\varphi \in H^1_0(\Omega)$ which solves \eqref{eq:adjoint-state}.
    Moreover, if $U$ is a bounded subset in $L^p(\Omega)$, then for any $u \in U$ and any $v \in L^q(\Omega)$, the solution $\varphi$ to \eqref{eq:adjoint-state} is an element of $H^2(\Omega) \cap W^{1,\infty}(\Omega)$ and there holds
    $\norm{\varphi}_{H^{2}(\Omega)} + \norm{\varphi}_{W^{1,\infty}(\Omega)} \leq C_{U} \norm{v}_{L^{q}(\Omega)}$.
    Furthermore, if $u \in L^p(\Omega)$ and $v \in L^r(\Omega)$ with $r \in (2, p_*)$, then $\varphi \in W^{2,r}(\Omega)$, where $p_*$ is given as in \cref{thm:control2state-oper}.
\end{theorem}

\begin{remark}
    \label{rem:regularity-derivative-S}
    In spite of the $W^{2,p}$-regularity of the state and adjoint state, the function $z_v := S'(u)v$ determined in \eqref{eq:deri} belongs to  $W^{1,p}(\Omega)$ only. This fact is due to the nondifferentiability of the function $a$.
\end{remark}

The optimal control problem \eqref{eq:P} can be transferred in the following form
\begin{equation}\label{eq:P2}
    \tag{P}
    \min_{u\in \mathcal{U}_{ad}}  j(u)  = \int_\Omega L(x,S(u)(x))\dx + \frac\nu2\norm{u}_{L^2(\Omega)}^2,
\end{equation}
where the admissible set is defined as
\[
    \mathcal{U}_{ad} := \{u \in L^\infty(\Omega) \mid  \alpha \leq u(x) \leq \beta \qquad \text{for a.e. } x \in \Omega  \}.
\]
Thanks to \crefrange{ass:domain}{ass:cost_func}, the cost functional $j: L^2(\Omega) \to \R$ is first-order continuously differentiable and satisfies
\begin{equation} \label{eq:obje-deri}
    j'(u)v = \int_{\Omega} (\varphi_u + \nu u )v \dx \quad \text{for } u, v \in L^2(\Omega)
\end{equation}
with $\varphi_u \in H^{1}_0(\Omega)$ solving \eqref{eq:adjoint-state} corresponding to the right-hand side term $v$ substituted by $\frac{\partial L}{\partial y}(\cdot, S(u))$; see \cite{ClasonNhuRosch_os2nd}, Thm.~4.2.
We have the following first-order necessary optimality conditions from Theorem 4.3 in \cite{ClasonNhuRosch_os2nd}.
\begin{theorem}[{\cite{ClasonNhuRosch_os2nd}, Thm.~4.3}]
    \label{thm:1st-OC}
    Let \crefrange{ass:domain}{ass:cost_func} hold.
    Then \eqref{eq:P2} admits at least one local minimizer $\bar u$. Furthermore, an adjoint state $\bar\varphi \in H^1_0(\Omega)$ exists and fulfills the following first-order optimality conditions
    \begin{subequations}
        \label{eq:1st-OS}
        \begin{align}
            &-\dive [(b + a(\bar y)) \nabla \bar y] = \bar u \quad \text{in } \Omega, \quad \bar y = 0 \, \text{on } \partial\Omega, \label{eq:state_OS} \\
            &-\dive  [(b +a(\bar y) ) \nabla \bar\varphi ] + \1_{\{ \bar y \neq \bar t \} } a'(\bar y) \nabla \bar y \cdot \nabla \bar \varphi = \frac{\partial L}{\partial y}(x, \bar y) \, \text{in } \Omega, \quad \bar \varphi =0 \, \text{on } \partial\Omega,
            \label{eq:adjoint_OS} \\
            &\int_\Omega (\bar \varphi + \nu \bar u )( u - \bar u ) \dx \geq 0 \quad \text{for all } u \in \mathcal{U}_{ad},  \label{eq:normal_OS}
        \end{align}
    \end{subequations}
    with  $\bar y:=S(\bar u)$.
    Moreover, $\bar y \in W^{2,p}(\Omega)$ and $\bar \varphi \in W^{2,r}(\Omega)$ for any $p, r \in (2, p_*)$ and $r \leq \bar p$ with $\bar p$ and $p_*$, respectively, defined in \cref{ass:cost_func} and \cref{thm:control2state-oper}. Consequently, $\bar y, \bar \varphi \in C^1(\overline\Omega)$ and $\bar u \in C^{0,1}(\overline\Omega)$.
\end{theorem}
Assume now that $\bar\varphi \in H^1_0(\Omega)$ fulfills \eqref{eq:1st-OS}.
The \emph{critical cone} of the problem \eqref{eq:P2} at $\bar u$ is defined by
\begin{equation} \label{eq:critical-cone}
    \mathcal{C}({\mathcal{U}_{ad};\bar u}) := \{ v \in L^2(\Omega)  \mid  v \geq 0 \, \text{if } \bar u = \alpha,  v \leq 0 \, \text{if } \bar u = \beta,  v  = 0 \, \text{if } \bar \varphi + \nu \bar u \neq 0 \ \text{a.e. in } \Omega
    \}.
\end{equation}

In the rest of this subsection, we shall provide second-order
sufficient
optimality conditions for \eqref{eq:P}. On that account, we need introduce the curvature functional of $j$, which can be separated into three contributions. For any $(u,y,\varphi)\in L^2(\Omega)\times H^1(\Omega)\times W^{1,\infty}(\Omega)$, the smooth part and the first-order nonsmooth part of the curvature in direction $(v_1,v_2)\in L^2(\Omega)^2$ are defined by
\begin{align*}
    & Q_s(u,y,\varphi; v_1,v_2) := \frac{1}{2} \int_\Omega \frac{\partial^2 L}{\partial y^2}(\cdot,y)z_{v_1} z_{v_2} \dx + \frac{\nu}{2} \int_\Omega v_1v_2 \dx
    - \frac{1}{2} \int_\Omega\1_{ \{ y \neq \bar t \}} a''(y)z_{v_1}z_{v_2} \nabla y \cdot \nabla \varphi \dx, \\
    &Q_1(u,y,\varphi; v_1,v_2) :=
    - \frac{1}{2} \int_\Omega [a'(y; z_{v_1}) \nabla z_{v_2} +a'(y; z_{v_2})\nabla z_{v_1} ]\cdot  \nabla\varphi  \dx
\end{align*}
with $z_{v_i} := S'(u)v_i$, $i=1,2$.
The critical part of the curvature is of course the second-order nonsmooth part involving some additional notation.
Let $\delta >0$ be arbitrary but fixed and set
\begin{equation}
    \label{eq:Omega-123-sets}
    \left\{
        \begin{aligned}
            \Omega_{y, \hat y }^{2} & :=  \{\hat y \in (\bar t, \bar t + \delta), y \in (\bar t-\delta, \bar t] \},\\
            \Omega_{y, \hat y }^{3} & := \{ \hat y \in (\bar t- \delta, \bar t), y \in [\bar t,\bar t + \delta) \}
        \end{aligned}
    \right.
\end{equation}
for  given functions $y, \hat y \in C(\overline\Omega)$.
For any $s \in \R$, $u, v \in L^2(\Omega)$, $y \in C(\overline\Omega) \cap H^1(\Omega)$, and $\varphi \in W^{1,\infty}(\Omega)$, we define
\begin{equation}  \label{eq:zeta-func}
    \left\{
        \begin{aligned}
            \zeta_0(u,y;s,v) &:= - \{a'\}_{\bar t+0}^{\bar t-0} (\bar t - S(u+sv))    \1_{\Omega_{S(u+sv), y }^{3}}, \\
            \zeta_1(u,y;s,v) &:= \{a'\}_{\bar t+0}^{\bar t-0} (\bar t - S(u+sv))    \1_{\Omega_{S(u+sv), y }^{2}}, \\
            \zeta(u,y;s,v)& := \zeta_0(u,y;s,v) + \zeta_1(u,y;s,v) = \{a'\}_{\bar t+0}^{\bar t-0}  (\bar t - S(u+sv))\left[ \1_{\Omega_{S(u+sv), y }^{2}} -  \1_{\Omega_{S(u+sv), y }^{3}} \right].
        \end{aligned}
    \right.
\end{equation}
Here $\{a'\}_{\bar t+0}^{\bar t-0}$ denotes  the difference between the one-sided derivatives of $a$ at $\bar t$ from left and right, i.e.,
\[
    \{a'\}_{\bar t+0}^{\bar t-0} : = \lim\limits_{t \to \bar t^{-}} a'(t) - \lim\limits_{t \to \bar t^{+}}  a'(t) =  a'_{0}(\bar t) - a'_{1}(\bar t).
\]
We then determine for any $\{s_n\} \in c_0^+:= \{ \{s_n\}\subset (0,\infty) \mid s_n \to 0 \}$ and $v \in L^2(\Omega)$ the term
\begin{multline}
    \label{eq:key-term-sn}
    \tilde{Q}(u,y,\varphi;\{s_n\}, v) := \liminf\limits_{n \to \infty} \frac{1}{s_n^2} \int_\Omega  \sum_{i =0 }^1 \zeta_i(u,y;s_n,v)  \nabla y \cdot \nabla \varphi \dx \\
    = \{a'\}_{\bar t+0}^{\bar t-0} \liminf\limits_{n \to \infty} \frac{1}{s_n^2}   \int_\Omega (\bar t - S(u+s_nv))     \left[\1_{\Omega_{S(u+s_n v), y }^{2}}  -\1_{\Omega_{S(u+s_nv), y }^{3}}\right] \nabla y \cdot \nabla \varphi \dx.
\end{multline}
The second-order nonsmooth part of the curvature in direction $v \in L^2(\Omega)$ is thus defined as
\begin{equation*}
    Q_2(u,y,\varphi; v) := \inf \{ \tilde{Q}(u,y,\varphi; \{s_n\}, v) \mid \{s_n\} \in c_0^+ \}.
\end{equation*}
We finally identify the total curvature in direction $v$ as
\begin{equation}
    \label{eq:curvature}
    Q(u,y,\varphi;v) := Q_s(u,y, \varphi; v,v) + Q_1(u,y, \varphi; v,v) +Q_2(u,y, \varphi; v).
\end{equation}
\begin{remark}
    \label{rem:curvature-term-simplification}
    The definitions of the sets $\Omega_{y, \hat y }^{2}$ and $\Omega_{y, \hat y }^{3}$ in \eqref{eq:Omega-123-sets} are derived from the ones for $\Omega_{y, \hat y }^{1,2}$ and $\Omega_{y, \hat y }^{0,3}$, respectively, in Lemma 3.3 in \cite{ClasonNhuRosch_os2nd} for the situation in which $K:=1$, $t_0 := -\infty$, $t_1 := \bar t$, and $t_2 := \infty$. Analogously, the definitions of $Q_s$, $Q_1$, and  $Q_2$ in this section can be obtained from those in \S\,5.1 in \cite{ClasonNhuRosch_os2nd}.
\end{remark}

Thanks to Proposition 5.6 and Lemma 5.7 in \cite{ClasonNhuRosch_os2nd}, we have the weak lower semicontinuity of  $Q_2$ in the last variable and there holds
\begin{equation*}
    |Q_2(u,S(u),\varphi; v) | \leq \Sigma(S(u)) \norm{\nabla \varphi}_{L^\infty(\Omega)} \norm{S'(u)v}_{L^\infty(\Omega)}^2
    \quad\text{for all }u, v \in L^2(\Omega)\text{ and } \varphi \in W^{1,\infty}(\Omega).
\end{equation*}
Here $\Sigma: C(\overline\Omega) \cap W^{1,1}(\Omega) \to \R \cup \{\infty \}$ is  the \emph{jump functional} and defined by
\begin{equation}
    \label{eq:E-functional}
    \Sigma(y) := \sigma_0 \limsup\limits_{r \to 0^+} \frac{1}{r} \sum_{m= 1}^{2}  \int_\Omega [ \1_{\{ 0 < |y - \bar t | \leq r \}} | \partial_{x_m} y | ]\dx, \quad y \in W^{1,1}(\Omega) \cap C(\overline\Omega)
\end{equation}
with
\begin{equation}
    \label{eq:sigma-i}
    \sigma_0 := |\{a'\}_{\bar t+0}^{\bar t-0}| = |a'_{0}(\bar t) - a'_{1}(\bar t)|.
\end{equation}
Moreover, if follows from Corollary 5.5 in \cite{ClasonNhuRosch_os2nd} that, for any $u \in L^2(\Omega)$,  $\{s_n\} \in c_0^+$ and $v_n \rightharpoonup v$ in $L^2(\Omega)$, there holds
\begin{equation} \label{eq:sigma-tilde}
    \liminf\limits_{n \to \infty} \frac{1}{s_n^2} \int_\Omega \sum_{i =0}^1 \zeta_i(u,S(u);s_n,v_n) \nabla S(u) \cdot \nabla \varphi \dx = \tilde{Q}(u,S(u),\varphi;\{s_n\}, v)  \geq Q_2(u,S(u),\varphi;v),
\end{equation}
provided that $\Sigma(S(u)) < \infty$.

\begin{theorem}[explicit second-order sufficient optimality conditions, {\cite{ClasonNhuRosch_levelset}, Thm.~3.23}]
    \label{thm:2nd-OS-suf}
    Assume that \crefrange{ass:domain}{ass:cost_func} are fulfilled. Let $\bar u$ be  an admissible control of \eqref{eq:P2} such that $\{ \bar y = \bar t\}$ decomposes into finitely many connected components and that on each such connected component $\Cu$, either
    \begin{equation}
        \label{eq:main-hypothesis-nonvanishing-gradient}
        \nabla \bar y(x) \neq 0\quad \text{for all} \quad x \in \Cu
    \end{equation}
    or
    \begin{equation} \label{eq:main-hypothesis-y-vanish-structure}
        \left\{
            \begin{aligned}
                & \nabla \bar y = 0 \, \text{on } \Cu,\\
                & \meas_{\R^2}( \{0 < | \bar y - \bar t| < r \} \cap \Cu^{\epsilon_0}) \leq c_s r, \quad \text{for all } r \in (0, r_0), \, \text{for some constants } {\epsilon_0}, r_0 >0
            \end{aligned}
        \right.
    \end{equation}
    holds
    with $\bar y:= S(\bar u)$
    and
    \[
        \Cu^{\epsilon_0} := \{ x \in \overline\Omega \mid \mathrm{dist}(x, \Cu) < \epsilon_0 \}.
    \]
    Assume further that  there exists a $\bar\varphi \in W^{1,\bar p}_0(\Omega) \cap W^{1,\infty}(\Omega)$, with $\bar p$ defined in \cref{ass:cost_func}, that together with $\bar u, \bar y$ fulfills \eqref{eq:1st-OS}
    and
    \begin{equation} \label{eq:2nd-OC-suff-explicit}
        \begin{aligned}[t]
            Q(\bar u, \bar y, \bar \varphi; v) &= \frac{1}{2} \int_\Omega \frac{\partial^2 L}{\partial y^2}(\cdot, \bar y)z_{v}^2 \dx + \frac{\nu}{2} \int_\Omega v^2 \dx
            - \frac{1}{2} \int_\Omega\1_{ \{\bar y \neq \bar t \}} a''(\bar y)z_{v}^2 \nabla \bar y \cdot \nabla \bar \varphi \dx \\
            \MoveEqLeft[-1] -  \int_\Omega a'(\bar y; z_{v}) \nabla z_{v} \cdot  \nabla \bar \varphi  \dx \\
            \MoveEqLeft[-1] + \frac{1}{2}  [a'_{0}(\bar t) - a'_{1}(\bar t)]   \int_{\{\bar y=\bar t\}} \1_{\{ |\nabla \bar y| > 0 \}} z_v^2  \frac{\nabla \bar y \cdot \nabla\bar  \varphi}{|\nabla \bar y|} \dH^{1}(x) \\
            &> 0 \qquad \text{for all } v\in \mathcal{C}({\mathcal{U}_{ad};\bar u}) \setminus \{0\}
        \end{aligned}
    \end{equation}
    with $z_v := S'(\bar u)v$.
    Then  constants $c_0, \rho_0 >0$ exist and fulfill
    \begin{equation*}
        j(\bar u) + c_0 \norm{u - \bar u}_{L^2(\Omega)}^2 \leq j(u)
        \qquad\text{for all } u \in \mathcal{U}_{ad} \cap \overline B_{L^2(\Omega)}(\bar u, \rho_0).
    \end{equation*}
\end{theorem}

\medskip
The following will be used later to show the error estimates for the approximation of \eqref{eq:P}.
\begin{proposition}[{\cite{ClasonNhuRosch_levelset}, Thm.~3.19 }]
    \label{prop:explicit-formular-Q2}
    Let $\bar u \in L^2(\Omega)$ be arbitrary and let $\bar y:= S(\bar u)$. Assume that
    $\{ \bar y = \bar t\}$ decomposes into finitely many connected components and that on each such connected component $\Cu$, either \eqref{eq:main-hypothesis-nonvanishing-gradient} or \eqref{eq:main-hypothesis-y-vanish-structure} is fulfilled.
    Let $v \in L^2(\Omega)$ and
    $\bar \varphi \in C^1(\overline\Omega) \cap W^{2,1}(\Omega)$.
    Then, for any $\{s_n\} \in c_0^+$ and $\{v_n\} \subset L^2(\Omega)$ such that $v_n \rightharpoonup v$ in $L^2(\Omega)$, $y_n:= S(\bar u + s_n v_n)  \to \bar y$ in $C^1(\overline\Omega)$,
    and $(y_n - \bar y)/s_n \to w$ in $W^{1,p}_0(\Omega)$ for some $p>2$, there hold
    \begin{equation}
        \label{eq:explicit-formular}
        \tilde Q(\bar u,\bar y,\bar \varphi;\{s_n\}, v) = \frac{1}{2}   [a'_{0}(\bar t) - a'_{1}(\bar t)]\int_{\{\bar y=\bar t\}}   \1_{\{ |\nabla \bar y| > 0 \}} w^2 \frac{\nabla \bar y \cdot \nabla\bar  \varphi}{|\nabla \bar y|} \dH^{1}(x)
    \end{equation}
    and
    \begin{equation}
        \label{eq:Pn-auxi}
        \lim\limits_{n \to \infty} \frac{1}{s_n^2} \int_\Omega (2 \bar t - \bar y - y_n) [\1_{ \Omega_{y_n, \bar y }^{2} }  -\1_{ \Omega_{y_n, \bar y }^{3}}] \nabla \bar y \cdot \nabla \bar \varphi \dx = 0.
    \end{equation}
\end{proposition}

Finally, the wellposedness of $\Sigma$, defined in \eqref{eq:E-functional}, is guaranteed by either \eqref{eq:main-hypothesis-nonvanishing-gradient} or \eqref{eq:main-hypothesis-y-vanish-structure} as stated below.
\begin{proposition}[{\cite{ClasonNhuRosch_levelset}, Prop.~3.20}]
    \label{prop:E-func-finite}
    Let  $\bar y \in C^1(\overline{\Omega})$ be such that
    $\{ \bar y = \bar t\}$ decomposes into finitely many connected components and that on each such connected component $\Cu$, either \eqref{eq:main-hypothesis-nonvanishing-gradient} or \eqref{eq:main-hypothesis-y-vanish-structure} is fulfilled.
    Then $\Sigma(\bar y) < \infty$.
\end{proposition}

\section{Analysis of the discrete state equation} \label{sec:state-equ-error}

In this section, we study the discrete version of the state equation \eqref{eq:state} and show error estimates of solutions to the discrete state equation \eqref{eq:state-discrete}, local uniqueness of these solutions, and local differentiability of the solution operators of \eqref{eq:state-discrete}. To this end, we introduce a family of regular triangulations $\{ \mathcal{T}_h \}_{h >0}: \overline\Omega = \bigcup_{T \in \mathcal{T}_h} T$ for all $h>0$.  For each element $T \in \mathcal{T}_h$, we denote by $\varrho(T)$ and $\delta(T)$ the diameter of $T$ and the diameter of the largest ball contained in $T$, respectively. The mesh size of $\mathcal{T}_h$ will be denoted by $h:= \max_{T \in \mathcal{T}_h} \varrho(T)$.  This triangulation is assumed to be \emph{regular} in the sense that there exist $\bar \varrho, \bar \delta >0$ such that $\frac{\varrho(T)}{\delta(T)} \leq \bar \delta$ and $\frac{h}{\varrho(T)} \leq \bar \varrho$ for all $T \in \mathcal{T}_h$ and $h > 0$; see, e.g. \cite{Ciarlet2002}.

We will employ the standard continuous piecewise linear finite elements for the state $y$ and set
\begin{equation*}
    V_h := \left\{ v_h \in C(\overline\Omega)  \mid   v_{h|_T} \in \mathcal{P}_1 \, \text{for all } T \in \mathcal{T}_h,v_h = 0 \, \text{on } \partial \Omega \right\},
\end{equation*}
where $\mathcal{P}_1$ stands for the space of polynomials of degree equal at most $1$.
The discrete approximation of the state equation \eqref{eq:state} for $y_h\in V_h$ is then
\begin{equation}
    \label{eq:state-discrete}
    \int_{\Omega} (b + a(y_h) ) \nabla y_h \cdot \nabla v_h \dx = \int_{\Omega} u v_h \dx\qquad\text{for all }v_h\in V_h.
\end{equation}
While the existence of solutions to \eqref{eq:state-discrete} follows from Theorem 3.1 in \cite{CasasDhamo2011}, the uniqueness of solutions is still an open problem. However, if $a: \R \to R$ is assumed to be bounded, then we have uniqueness provided that $h$ is small enough; see Theorem 4.1 in \cite{CasasDhamo2011}.
Below, we provide some error estimates for solutions to \eqref{eq:state-discrete} that are sufficiently close to the solutions of \eqref{eq:state}.

In what follows, we fix $\bar u \in L^2(\Omega)$ and set $\bar y := S(\bar u)$.
From \cref{thm:control2state-oper} and the continuous embedding $L^2(\Omega) \hookrightarrow W^{-1,p}(\Omega)$ for any $p > 1$, we then have $\bar y \in W^{1,p}_0(\Omega) \cap H^2(\Omega)$.
\begin{theorem}[{\cite{CasasTroltzsch2011}, Thm.~3.1}]
    \label{thm:error-state}
    Let  $\rho_0 >0$ be arbitrary but fixed,   $U := \overline B_{L^2(\Omega)}(\bar u, \rho_0)$,  and let $p \geq 2$. Assume that \crefrange{ass:domain}{ass:PC1-func} are fulfilled. Then there exists a constant $h_0 >0$ such that for any $u \in U$ and $h < h_0$, there exists at least one solution $y_h(u)$ to \eqref{eq:state-discrete} satisfying for $y_u:=S(u)$
    \begin{align}
        & \norm{y_u - y_h(u)}_{L^2(\Omega)} + h \norm{y_u - y_h(u)}_{H^1_0(\Omega)} +  h\norm{y_u - y_h(u)}_{L^\infty(\Omega)} \leq C_{U} h^2, \label{eq:error-state-l2} \\
        & \norm{y_u - y_h(u)}_{W^{1,p}_0(\Omega)}   \leq C_{U,p} h^{2/p}. \label{eq:error-state-w1p}
    \end{align}
\end{theorem}
\begin{proof}
    The estimates for the norms in $L^2$, $H^1_0$, and $W^{1,p}$ are shown in Theorem 3.1 in \cite{CasasTroltzsch2011}, while the estimate for the $L^\infty$ norm can be obtained similar to estimate (3.11) in \cite{CasasTroltzsch2011}.
\end{proof}

The following theorem guarantees the local uniqueness of solutions to \eqref{eq:state-discrete}. Its proof is similar to that of Theorem 4.2 in \cite{CasasDhamo2012} with slight modifications and
is thus omitted here.
\begin{theorem}
    \label{thm:uniqueness-state}
    Let $p >2$ be arbitrary and let $h_0$ be defined in
    \cref{thm:error-state}. Under \crefrange{ass:domain}{ass:PC1-func}, there exist $h_1 \in (0, h_0)$, $\rho >0$, and $\kappa_\rho >0$ such that for any $h < h_1$ and any $u \in \overline B_{L^2(\Omega)}(\bar u, \rho)$,  \eqref{eq:state-discrete} admits a unique solution in $\overline B_{W^{1,p}_0(\Omega)}(\bar y, \kappa_\rho) \cap V_h$.
\end{theorem}

From now on, let us fix $\tilde{p} \geq 4$ and let $h_0, h_1, \rho$, and $\kappa_\rho$ be the constants defined in
\cref{thm:error-state,thm:uniqueness-state} for $p = \tilde{p}$.
In the rest of this section, we shall investigate the differentiability of the discrete solution operator
\begin{equation}
    \label{eq:discrete-operator}
    S_h: B_{L^2(\Omega)}(\bar u, \rho) \ni u \mapsto y_h(u) \in \overline B_{W^{1,\tilde{p}}_0(\Omega)}(\bar y, \kappa_\rho) \cap V_h,
\end{equation}
where $y_h(u)$ is the unique solution to \eqref{eq:state-discrete} in $\overline B_{W^{1,\tilde{p}}_0(\Omega)}(\bar y, \kappa_\rho)$ from \cref{thm:uniqueness-state}.

For any $y, \hat y \in C(\overline\Omega) \cap W^{1,1}(\Omega)$, we define functions $T_{y, \hat y}$ and $Z_{y, \hat y}$ on $\Omega$ via
\begin{equation}
    \label{eq:T-func}
    T_{y, \hat y} := \1_{\{\hat y \neq \bar t\} }[ a(y) - a(\hat y) - a'(\hat y)(y -\hat y) ]
\end{equation}
and
\begin{equation}
    \label{eq:Z-func}
    Z_{y, \hat y} := \1_{\{y \neq \bar t\}} a'(y) \nabla y - \1_{\{\hat y \neq \bar t\}} a'(\hat y) \nabla \hat y.
\end{equation}
In order to prove the differentiability of $S_h$, we need the following lemmas.
\begin{lemma}[{\cite{ClasonNhuRosch_os2nd}, Lem.~3.3}] \label{lem:T-decomposition}
    Let \cref{ass:PC1-func} be fulfilled. Assume that $y_n \to y$ in $W^{1,p}_0(\Omega)$ as $n \to \infty$ with $p >2$. Then
    \begin{equation*}
        \frac{1}{\norm{y_n - y}_{W^{1,p}_0(\Omega)}} \norm{ T_{y_n, y} \nabla y}_{L^p(\Omega)} \to 0 \quad \text{as } n \to \infty.
    \end{equation*}
\end{lemma}
\begin{lemma}
    \label{lem:Z-decomposition}
    Let \cref{ass:PC1-func} be fulfilled and let $y, \hat y \in C(\overline\Omega) \cap W^{1,1}(\Omega)$ and $M>0$ be arbitrary such that $\norm{y - \hat y}_{C(\overline\Omega)} < \delta$ with $\delta$ defined in \eqref{eq:Omega-123-sets} and $\norm{y}_{C(\overline\Omega)}, \norm{\hat y}_{C(\overline\Omega)} \leq M$. Then
    \begin{equation}
        \label{eq:Z-decomposition}
        Z_{y,\hat y} = Z_{y,\hat y}^{(1)} + Z_{y,\hat y}^{(2)}+ Z_{y,\hat y}^{(3)} +Z_{y,\hat y}^{(4)},
    \end{equation}
    for
    \begin{align*}
        Z_{y,\hat y}^{(1)} &:= \1_{\{ \hat y \in (-\infty, \bar t), y \in (-\infty, \bar t) \}} [ a_0'(y) \nabla y - a_0'(\hat y) \nabla \hat y ] + \1_{\{ \hat y \in (\bar t, \infty), y \in (\bar t, \infty) \}} [ a_1'(y) \nabla y - a_1'(\hat y) \nabla \hat y ],\\
        Z_{y,\hat y}^{(2)} &:= \1_{\{ \hat y = \bar t \}}\1_{\{ y \neq \bar t \}} a'(y) \nabla y, \qquad   Z_{y,\hat y}^{(3)}  :=  [a_0'(\bar t) -a_1'(\bar t)][ \1_{\Omega_{y,\hat y}^{2}}    -  \1_{\Omega_{y,\hat y}^{3}} ]  \nabla \hat y,
    \end{align*}
    and
    \begin{multline*}
        Z_{y, \hat y}^{(4)} :=  \1_{ \Omega_{y,\hat y}^{2} } [ a_{0}'(y) \nabla (y - \hat y) + (a_{0}'(y) - a_{0}'(\bar t)) \nabla \hat y + (a_{1}'(\bar t) - a_1'(\hat y)) \nabla \hat y  ] \\
        + \1_{ \Omega_{y,\hat y}^{3} } [ a_{1}'(y) \nabla (y - \hat y) + (a_{1}'(y) - a_{1}'(\bar t)) \nabla \hat y + (a_{0}'(\bar t) - a_0'(\hat y)) \nabla \hat y  ],
    \end{multline*}
    with  the sets $\Omega_{y, \hat y }^{2}$  and $\Omega_{y, \hat y }^{3}$ defined as in \eqref{eq:Omega-123-sets}. Moreover, there exists a constant $C_M>0$ such that a.e. in $\Omega$,
    \begin{equation}
        \label{eq:Z-esti}
        \left\{
            \begin{aligned}
                | Z_{y, \hat y}^{(1)} | + | Z_{y, \hat y}^{(2)} |&\leq C_M [ |y - \hat y|| \nabla \hat y| + |\nabla(y - \hat y)| ],\\
                | Z_{y, \hat y}^{(4)} | &\leq C_M [ |y - \hat y|| \nabla \hat y| + |\nabla(y - \hat y)| ]  (  \1_{\Omega_{y,\hat y}^{2}} + \1_{\Omega_{y,\hat y}^{3}} ).
            \end{aligned}
        \right.
    \end{equation}
    Consequently, $Z_{y,\hat y} \to 0$ in  $L^p(\Omega)$ as $y \to \hat y$ in $W^{1,p}(\Omega) \cap C(\overline\Omega)$ for any $p \geq 1$.
\end{lemma}
\begin{proof}
    Clearly, we deduce from the fact $\nabla \hat y = 0$ a.e. in $\{\hat y = \bar t\}$  (see \cite{Chipot2009}) that
    \begin{multline} \label{eq:Z-sum}
        Z_{y,\hat y} =  \1_{ \{ \hat y \in (-\infty, \bar t) \}} [ \1_{\{y \neq \bar t \} }  a'(y) \nabla y - a_0'(\hat y) \nabla \hat y]  + \1_{ \{ \hat y \in ( \bar t, \infty) \}} [ \1_{\{y \neq \bar t \} }  a'(y) \nabla y - a_1'(\hat y) \nabla \hat y]\\
        + \1_{ \{ \hat y = \bar t  \} } \1_{\{y \neq \bar t \} }  a'(y) \nabla y  = Z_0^1 + Z_1^1 + Z_{y,\hat y}^{(2)}
    \end{multline}
    with
    \begin{equation*}
        Z_0^1  :=   \1_{ \{ \hat y \in (-\infty, \bar t) \}} [ \1_{\{y \neq \bar t \} }  a'(y) \nabla y - a_0'(\hat y) \nabla \hat y] \quad \text{and} \quad
        Z_1^1 :=   \1_{ \{ \hat y \in ( \bar t, \infty) \}} [ \1_{\{y \neq \bar t \} }  a'(y) \nabla y - a_1'(\hat y) \nabla \hat y].
    \end{equation*}
    Since $\norm{y - \hat y}_{C(\overline\Omega)} < \delta$ and $\nabla y = 0$ a.e. in $\{y=\bar t \}$, we can write
    \begin{equation*}
        \begin{aligned}
            Z_0^1 &=  \1_{ \{ \hat y \in (-\infty, \bar t), y \in (-\infty, \bar t) \} } [ a_{0}'(y) \nabla y - a_0'(\hat y) \nabla \hat y   ]
            + \1_{ \{ \hat y \in (\bar t -\delta, \bar t),  y \in [\bar t , \bar t+ \delta) \} } [  a_{1}'(y) \nabla y - a_0'(\hat y) \nabla \hat y ] \\
            &=: Z_0^{1,2}+Z_0^{1,1}.
        \end{aligned}
    \end{equation*}
    and
    \begin{equation*}
        \begin{aligned}
            Z_1^1 &= \1_{ \{ \hat y \in (\bar t, \bar t + \delta),  y \in (\bar t- \delta, \bar t] \} } [ a_{0}'(y) \nabla y - a_1'(\hat y) \nabla \hat y   ]
            + \1_{ \{ \hat y \in (\bar t, \infty), y \in (\bar t,\infty ) \} } [ a_{1}'(y) \nabla y - a_1'(\hat y) \nabla \hat y   ] \\
            &=: Z_1^{1,1}+Z_1^{1,2}.
        \end{aligned}
    \end{equation*}
    Thus, we have  from the definition of $Z_{y, \hat y}^{(1)}$ that
    \begin{equation}
        \label{eq:Z-1}
        \begin{aligned}[t]
            Z_0^{1,2} + Z_1^{1,2} &= \1_{ \{ \hat y \in (-\infty, \bar t), y \in (-\infty, \bar t) \} } [ a_{0}'(y) \nabla y - a_0'(\hat y) \nabla \hat y   ] + \1_{ \{ \hat y \in (\bar t, \infty), y \in (\bar t,\infty ) \} } [ a_{1}'(y) \nabla y - a_1'(\hat y) \nabla \hat y   ] \\
            &= Z_{y, \hat y}^{(1)}.
        \end{aligned}
    \end{equation}
    By using the definition of  $\Omega_{y, \hat y }^{3}$ in \eqref{eq:Omega-123-sets}, we now write
    \begin{equation*}
        \begin{aligned}
            Z_0^{1,1} & = \1_{ \{ \hat y \in (\bar t -\delta, \bar t),  y \in [\bar t , \bar t+ \delta) \} } [  a_{1}'(y) \nabla y - a_0'(\hat y) \nabla \hat y ] =\1_{ \Omega_{y, \hat y }^{3}}[  a_{1}'(y) \nabla y - a_0'(\hat y) \nabla \hat y ] \\
            & = \1_{\Omega_{y, \hat y }^{3}} [ a_{1}'(\bar t)  - a_0'(\bar t)    ] \nabla  \hat y + \tilde{Z}_0^{1,1}
        \end{aligned}
    \end{equation*}
    with
    \[
        \tilde{Z}_0^{1,1} :=  \1_{ \Omega_{y,\hat y}^{3} } [ a_{1}'(y) \nabla (y - \hat y) + (a_{1}'(y) - a_{1}'(\bar t)) \nabla \hat y + (a_{0}'(\bar t) - a_0'(\hat y)) \nabla \hat y  ].
    \]
    Similarly, the definition of $\Omega_{y, \hat y }^{2}$ in \eqref{eq:Omega-123-sets} implies that
    \begin{equation*}
        \begin{aligned}
            Z_1^{1,1} & = \1_{ \{ \hat y \in (\bar t, \bar t + \delta),  y \in (\bar t- \delta, \bar t] \} } [ a_{0}'(y) \nabla y - a_1'(\hat y) \nabla \hat y   ] \\
            & = \1_{ \Omega_{y,\hat y}^{2} } [ a_{0}'(\bar t)  - a_1'(\bar t)    ] \nabla  \hat y  + \tilde{Z}_1^{1,1}
        \end{aligned}
    \end{equation*}
    with
    \[
        \tilde{Z}_1^{1,1}  := \1_{ \Omega_{y,\hat y}^{2} } [ a_{0}'(y) \nabla (y - \hat y) + (a_{0}'(y) - a_{0}'(\bar t)) \nabla \hat y + (a_{1}'(\bar t) - a_1'(\hat y)) \nabla \hat y  ].
    \]
    Obviously, we have
    \begin{equation*}
        \tilde{Z}_0^{1,1} +\tilde{Z}_1^{1,1} =Z_{y, \hat y}^{(4)}
    \end{equation*}
    and there then holds
    \[
        Z_{y, \hat y}^{(3)} + Z_{y, \hat y}^{(4)} = Z_0^{1,1} + Z_1^{1,1}.
    \]
    From this and \eqref{eq:Z-sum}--\eqref{eq:Z-1}, we derive \eqref{eq:Z-decomposition}. Moreover, \eqref{eq:Z-esti} is derived by combining the definition of $Z_{y, \hat y}^{(k)}$, \cref{ass:PC1-func}, the estimates
    \[
        |y(x) - \bar t|, |\hat y(x) - \bar t| \leq |y(x) - \hat y(x)| \quad \text{for a.e. } x \in \Omega_{y,\hat y}^{2} \cup \Omega_{y,\hat y}^{3}
    \]
    due to the definition of $\Omega_{y,\hat y}^{2}$ and $\Omega_{y,\hat y}^{3}$, and the fact that $\nabla \hat y = 0$ a.e. in $\{\hat y = t_i \}$ (see; e.g. Remark 2.6 in \cite{Chipot2009}). Finally, the claimed convergence follows from \eqref{eq:Z-decomposition}, \eqref{eq:Z-esti}, the fact that
    $
    \1_{\Omega_{y, \hat y }^{2}}, \1_{\Omega_{y, \hat y }^{3}} \to 0$ a.e. in $\Omega$ as $y \to \hat y$ in $C(\overline\Omega)$,
    and Lebesgue's dominated convergence theorem.
\end{proof}

For any $h \in (0, h_1)$ and $y_h \in V_h$, we now define the operator $D_{h, y_h}: V_h \to V_h^*$ via
\begin{equation}
    \label{eq:Dh-oper}
    \langle D_{h, y_h} w_h, z_h \rangle := \int_\Omega [ (b + a(y_h) ) \nabla w_h +  \1_{\{y_h \neq \bar t \}} a'(y_h) \nabla y_h w_h ] \cdot \nabla z_h \dx, \, w_h,z_h \in V_h.
\end{equation}
\begin{lemma}
    \label{lem:Dh-continuity}
    Let all assumptions of \cref{thm:error-state} hold. Then for any $h \in (0,h_1)$ and any $\{y_h^{k}\} \subset V_h$ converging to $y_h \in V_h$ in $H^1_0(\Omega)$ as $k\to \infty$, there holds
    $\norm{D_{h, y_h^k} - D_{h, y_h} }_{\Linop(V_h,V_h^*)} \to 0$.
\end{lemma}
\begin{proof}
    Let $w_h, v_h \in V_h$ be arbitrary such that $\norm{w_h}_{H^1_0(\Omega)}, \norm{v_h}_{H^1_0(\Omega)} \leq 1$ and $h \in (0,h_1)$ be arbitrary but fixed. Assume that $\{y_h^{k}\} \subset V_h$ converges to $y_h \in V_h$ in $H^1_0(\Omega)$ as $k \to \infty$. By virtue of the inverse inequality \cite{Ciarlet2002}, Thm.~3.2.6,
    we deduce that $y_h^k \to y_h$ in $W^{1,\tilde{p}}_0(\Omega)$ and hence in $C(\overline\Omega)$ as $k \to \infty$. We can therefore assume that $\norm{y_h^k - y_h}_{C(\overline\Omega)} < \delta$ for all $k \in \N$ large enough.
    On the other hand, we have
    \begin{equation*}
        \langle (D_{h, y_h^k} - D_{h, y_h} )w_h, v_h   \rangle =   \int_\Omega [  (a(y_h^k) - a(y_h) ) \nabla w_h + Z_{y_h^k,y_h} w_h ]\cdot \nabla v_h \dx.
    \end{equation*}
    Together with the Hölder inequality, this yields that
    \begin{equation*}
        \begin{aligned}
            \norm{D_{h, y_h^k} - D_{h, y_h} }_{\Linop(V_h,V_h^*)}   & \leq \norm{a(y_n^k) - a(y_h)}_{L^\infty(\Omega)} + \norm{ Z_{y_h^k, y_h}}_{L^{\tilde{p}}(\Omega)} \norm{w_h}_{L^{2\tilde{p}/(\tilde{p}-2)}(\Omega)} \\
        & \leq \norm{a(y_h^k) - a(y_h)}_{L^\infty(\Omega)} + C \norm{ Z_{y_h^k, y_h}}_{L^{\tilde{p}}(\Omega)},\end{aligned}
    \end{equation*}
    where we have employed the continuous embedding $H^1_0(\Omega) \hookrightarrow L^{2\tilde{p}/(\tilde{p}-2)}(\Omega)$ and the fact that $\norm{w_h}_{H^1_0(\Omega)} \leq 1$ to obtain the last inequality.
    The first term on the right-hand side of the last estimate tends to zero as $k\to \infty$ since $y_h^k \to y_h$ in $C(\overline\Omega)$ as $k \to \infty$. Moreover, the second term tends to zero as a result of \cref{lem:Z-decomposition}.
\end{proof}

\begin{lemma}
    \label{lem:Dh-isomorphism}
    Let all assumptions of \cref{thm:error-state} hold. Then there exists a constant $h_2 \in (0, h_1)$ such that for any $h \in (0, h_2)$ and any $y_h \in \overline B_{W^{1,\tilde{p}}_0(\Omega)}(\bar y, \kappa_\rho) \cap V_h$, the operator $D_{h,y_h}: V_h \to V_h^*$ is an isomorphism.
\end{lemma}
\begin{proof}
    Since $V_h$ is finite-dimensional and $D_{h, y_h}$ is linear, it suffices to prove that there exists an $h_2 \in (0, h_1)$ such that for any $h \in (0, h_2)$ and $y_h \in \overline B_{W^{1,\tilde{p}}_0(\Omega)}(\bar y, \kappa_\rho) \cap V_h$, the equation
    \begin{equation}
        \label{eq:Dh-injective}
        D_{h,y_h}w_h = 0
    \end{equation}
    admits the unique solution $w_h = 0$. We argue by contradiction. Assume for any $k \geq 1$ that there exist $h_k \in (0, h_1)$, $y_{h_k} \in \overline B_{W^{1,\tilde{p}}_0(\Omega)}(\bar y, \kappa_\rho) \cap V_{h_k}$, and $ w_{h_k} \in V_h \backslash \{0\}$ such that $h_k \to 0^+$ and $w_{h_k}$ solves \eqref{eq:Dh-injective} for $h = h_k$ and $y_h = y_{h_k}$. By setting $\hat w_{h_k} := \frac{w_{h_k}}{\norm{w_{h_k}}_{L^{2\tilde{p}/(\tilde{p}-2)}(\Omega)}}$, we deduce that
    \begin{equation}
        \label{eq:Dh-iso-auxi}
        \norm{\hat w_{h_k}}_{L^{2\tilde{p}/(\tilde{p}-2)}(\Omega)} = 1 \quad\text{and} \quad D_{h_k, y_{h_k}}\hat w_{h_k} = 0.
    \end{equation}
    Furthermore, as a result of the embedding $W^{1,\tilde{p}}_0(\Omega) \Subset C(\overline\Omega)$, there hold that $\norm{y_{h_k}}_{C(\overline\Omega)} \leq M$ for all $k \geq 1$ and some constant $M>0$ independent of $k$ and that
    \begin{equation}
        \label{eq:Dh-iso-yh-limit}
        y_{h_k} \to y \, \text{in } C(\overline\Omega) \quad \text{for some } y \in W^{1,\tilde{p}}_0(\Omega).
    \end{equation}
    Testing the second equation in \eqref{eq:Dh-iso-auxi} by $\hat w_{h_k}$, Hölder's inequality thus gives
    \begin{equation*}
        \underline{b} \norm{\nabla \hat w_{h_k}}_{L^2(\Omega)} \leq \norm{\nabla y_{h_k}}_{L^{\tilde{p}}(\Omega)}  \norm{\1_{\{y_{h_k} \neq \bar t \}} a'(y_{h_k})}_{L^\infty(\Omega)} \norm{\hat w_{h_k}}_{L^{2\tilde{p}/(\tilde{p}-2)}(\Omega)} \leq C_{M,\rho}
    \end{equation*}
    for some constant $C_{M,\rho}>0$. From this and the compact embedding $H^1_0(\Omega) \Subset L^{2\tilde{p}/(\tilde{p}-2)}(\Omega)$, a subsequence argument shows that we can assume that
    \begin{equation} \label{eq:Dh-iso-wh-limit}
        \hat w_{h_k} \rightharpoonup \hat w \, \text{in } H^{1}_0(\Omega) \quad \text{and} \quad   \hat w_{h_k} \to \hat w \, \text{in } L^{2\tilde{p}/(\tilde{p}-2)}(\Omega)
    \end{equation}
    for some $\hat w \in H^1_0(\Omega)$. Moreover, there exist an element $\mathbb{b} \in L^{\tilde  p}(\Omega)^2$ and a subsequence of $\{\mathbb b_k\}$ with $\mathbb b_k := \1_{\{y_{h_k} \neq \bar t \}} a'(y_{h_k}) \nabla y_{h_k}$, denoted in the same way, such that $\mathbb b_k \rightharpoonup \mathbb b$ weakly in $L^{\tilde{p}}(\Omega)^2$.
    By fixing any $v \in H^2(\Omega) \cap H^1_0(\Omega)$ and testing the last equation in \eqref{eq:Dh-iso-auxi} with $v_{h_k} := \Pi_{h_k} v \in V_{h_k}$, where $\Pi_{h_k}$ is the interpolation operator, we have
    \begin{equation*}
        \int_\Omega [ (b + a(y_{h_k}) ) \nabla \hat w_{h_k}  + \hat w_{h_k} \mathbb b_k ] \cdot \nabla v_{h_k} \dx = 0 \quad \text{for all } k \geq 1.
    \end{equation*}
    Letting $k \to \infty$ and exploiting the limits \eqref{eq:Dh-iso-yh-limit}, \eqref{eq:Dh-iso-wh-limit}, $\mathbb b_k \rightharpoonup \mathbb b$ in  $L^{\tilde{p}}(\Omega)^N$, and $v_{h_k} \to v$ in $H^1_0(\Omega)$, we can conclude that
    $\int_\Omega [ (b + a(y) ) \nabla \hat w  + \hat  w \mathbb b ] \cdot \nabla v \dx = 0$.
    From this, the density of $H^2(\Omega) \cap H^1_0(\Omega)$ in $H^1_0(\Omega)$, and Theorem 2.6 in \cite{CasasDhamo2011}, we conclude that $\hat w = 0$, contradicting the fact that $\norm{\hat w}_{L^{2\tilde{p}/(\tilde{p}-2)}(\Omega)} = \lim\limits_{k \to \infty} \norm{\hat w_{h_k}}_{L^{2\tilde{p}/(\tilde{p}-2)}(\Omega)} = 1$.
\end{proof}

As a consequence of \cref{lem:Dh-continuity,lem:Dh-isomorphism} and the implicit function theorem, we obtain the differentiability of $S_h$.
\begin{theorem}
    \label{thm:diff-discrete-state}
    Let all assumptions of \cref{thm:error-state} hold. Then, for any $h \in (0, h_2)$, the operator $S_h$ defined in \eqref{eq:discrete-operator} is of class $C^1$. Moreover, for any $u \in B_{L^2(\Omega)}(\bar u, \rho)$, let $y_h(u) := S_h(u)$. Then for any $v \in L^2(\Omega)$, the Fréchet derivative $S_h'(u)v=:z_h$ is the unique solution to
    \begin{equation}
        \label{eq:diff-state-discrete}
        \int_{\Omega} [ (b + a(y_h(u)) ) \nabla z_h + \1_{\{y_h(u) \neq \bar t \}} a'(y_h(u)) z_h \nabla y_h(u) ] \cdot \nabla w_h \dx = \int_{\Omega} v w_h \dx\quad\text{for all }w_h\in V_h.
    \end{equation}
\end{theorem}
\begin{proof}
    We first consider for any $h \in (0, h_2)$ the mapping $F_h: B_{L^2(\Omega)}(\bar u, \rho) \times V_h \to V_h^*$ defined via
    \begin{equation} \label{eq:Fh-mapping}
        \left\langle F_h(u,y_h), v_h  \right\rangle = \int_\Omega ( b + a(y_h) ) \nabla y_h \cdot \nabla v_h - uv_h \dx, \quad u \in B_{L^2(\Omega)}(\bar u, \rho), y_h, v_h \in V_h.
    \end{equation}
    Clearly, $F_h(u,y_h(u)) =0$ and $F_h$ is continuously partially differentiable in $u$. We now prove that $F_h$ is partially differentiable in $y_h$ with $\frac{\partial F_h}{\partial y_h}(u,y_h) = D_{h, y_h}$, where $D_{h,y_h}$ is defined in \eqref{eq:Dh-oper}. We thus  derive the differentiability of $S_h$ according to \cref{lem:Dh-continuity,lem:Dh-isomorphism} as well as a simple computation. To this end, by taking any $v_h \in V_h$ and $\{w_h^k \} \subset V_h $ with $\norm{w_h^k}_{H^1_0(\Omega)} \to 0$ as $k \to \infty$ and $\norm{v_h}_{H^1_0(\Omega)} \leq 1$, we deduce from a straightforward computation that
    \[
        \langle F_h(u, y_h  + w_h^k) - F_h(u, y_h) - D_{h,y_h} w_h^k, v_h  \rangle = \int_\Omega [ T_{y_h^k, y_h} \nabla y_h  + ( a(y_h^k ) - a(y_h) ) \nabla w_h^k ] \cdot \nabla v_h \dx,
    \]
    where $y_h^k := y_h  + w_h^k$ and $T_{y, \hat y}$ is defined in \eqref{eq:T-func}.
    This gives
    \begin{equation*}
        \norm{ F_h(u, y_h  + w_h^k) - F_h(u, y_h) - D_{h,y_h} w_h^k}_{\Linop(V_h, V_h^*)}
        \leq \norm{T_{y_h^k, y_h} \nabla y_h}_{L^2(\Omega)} + \norm{ a(y_h^k) - a(y_h)}_{L^\infty(\Omega)} \norm{w_h^k}_{H^1_0(\Omega)}.
    \end{equation*}
    Moreover, in view of inverse estimates \cite{Ciarlet2002}, Thm.~3.2.6, we have
    $y_h^k \to y_h$ in $W^{1,\tilde{p}}_0(\Omega)$ and hence in $C(\overline\Omega)$ as $k \to \infty$. Then \cref{lem:T-decomposition} and the embedding $W^{1,\tilde{p}}_0(\Omega) \hookrightarrow H^1_0(\Omega)$ imply that
    \begin{equation*}
        \frac{1}{\norm{w_h^k}_{H^1_0(\Omega)}}\norm{ F_h(u, y_h  + w_h^k) - F_h(u, y_h) - D_{h,y_h} w_h^k}_{\Linop(V_h, V_h^*)} \to 0,
    \end{equation*}
    which gives that $\frac{\partial F_h}{\partial y_h}(u,y_h) = D_{h, y_h}$.
    We have shown that $F_h(u,S_h(u)) =0$ and $\frac{\partial F_h}{\partial y_h}(u,y_h) = D_{h, y_h}$. We then deduce from the Implicit Function Theorem and \cref{lem:Dh-continuity,lem:Dh-isomorphism} that $S_h$ is of class $C^1$. Finally, \eqref{eq:diff-state-discrete} follows from \eqref{eq:Dh-oper} and \eqref{eq:Fh-mapping}.
\end{proof}

\section{Numerical analysis of the adjoint state equation} \label{sec:adjoint-state-equ-error}
In this section, we will carry out the numerical analysis of the adjoint equation \eqref{eq:adjoint-state}.
For any $h \in (0, h_2)$, $u \in B_{L^2(\Omega)}(\bar u, \rho)$, $v \in L^2(\Omega)$ and $y_h := S_h(u)$,  we approximate \eqref{eq:adjoint-state} using the triangulation $\mathcal{T}_h$ by
\begin{equation}
    \label{eq:adjoint-state-discrete}
    \int_{\Omega}  (b + a(y_h) ) \nabla \varphi_h \cdot \nabla w_h + \1_{\{y_h \neq \bar t \}}a'(y_h) w_h \nabla y_h  \cdot \nabla \varphi_h \dx = \int_\Omega v w_h \dx \quad\text{for all }w_h\in V_h.
\end{equation}

From the bijectivity of $D_{h,y_h}$ shown in \cref{lem:Dh-isomorphism}, we deduce the existence and uniqueness of solutions to \eqref{eq:adjoint-state-discrete}.
\begin{theorem}
    \label{thm:exist-sol-adjoint-discrete}
    Let all assumptions of \cref{thm:diff-discrete-state} hold. Then for all $h \in (0, h_2)$, $u \in B_{L^2(\Omega)}(\bar u, \rho)$, and $v \in L^2(\Omega)$, there exists a unique solution $\varphi_h\in V_h$ to \eqref{eq:adjoint-state-discrete}.
\end{theorem}

In order to derive error estimates for the full approximation \eqref{eq:adjoint-state-discrete} of \eqref{eq:adjoint-state}, we first consider the continuous problem \eqref{eq:adjoint-state} with $y_h(u)$ in place of $y_u$.
\begin{lemma}
    \label{lem:adjoint-disrete-half}
    Let all assumptions of \cref{thm:diff-discrete-state} hold. Then  for any $h \in (0, h_2)$, $u \in B_{L^2(\Omega)}(\bar u, \rho)$, $y_h :=S_h(u)$, and $v \in L^2(\Omega)$, the equation
    \begin{equation} \label{eq:adjoint-discrete-half}
        -\dive [(b + a(y_h)) \nabla \tilde  \varphi]  + \1_{\{ y_h \neq \bar t \}}a'(y_h) \nabla y_h \cdot \nabla \tilde \varphi  = v \quad\text{ in } \Omega, \quad \tilde \varphi =0 \quad \text{on } \partial\Omega,
    \end{equation}
    has a unique solution $\tilde{\varphi}$ in $H^2(\Omega) \cap H^1_0(\Omega)$. Moreover,
    \begin{equation}
        \label{eq:adjoint-discrete-half-H1-L2error}
        \norm{\varphi - \tilde{\varphi}}_{H^2(\Omega)} \leq C_{\rho} h \norm{v}_{L^2(\Omega)} \quad \text{and} \quad \norm{\varphi - \tilde{\varphi}}_{L^2(\Omega)}  \leq C_{\rho} h^2 \norm{v}_{L^2(\Omega)}
    \end{equation}
    for some constant $C_{\rho}$ independent of $u, v$, and $h$, where $\varphi$ is the unique solution to \eqref{eq:adjoint-state}.
\end{lemma}
\begin{proof}
    From \cref{thm:control2state-oper}, the continuous embedding $W^{1,\tilde{p}}_0(\Omega) \hookrightarrow C(\overline\Omega)$, and \eqref{eq:error-state-w1p} for $p:= \tilde{p}\geq 4$, there holds
    \begin{equation*}
        \norm{\1_{\{ y_h \neq \bar t \}}a'(y_h) \nabla y_h}_{L^{\tilde{p}}(\Omega)} + \norm{b + a(y_h(u))}_{W^{1,\tilde{p}}(\Omega)}  \leq C_{\rho} \quad \text{for all } h \in (0, h_2), u \in \overline B_{L^2(\Omega)}(\bar u, \rho).
    \end{equation*}
    A standard argument then proves the existence of solutions $\tilde{\varphi}$ to \eqref{eq:adjoint-discrete-half} in $H^2(\Omega) \cap H^1_0(\Omega)$; see. e.g. Theorem 2.6 in \cite{CasasDhamo2011} and the proof of Lemma 4.1 in \cite{ClasonNhuRosch_os2nd}.
    Moreover, we have
    \begin{equation}
        \label{eq:adjoint-state-half-estiH2}
        \norm{\tilde\varphi}_{H^2(\Omega)} \leq C_{\rho} \norm{v}_{L^2(\Omega)}.
    \end{equation}
    Setting $\psi := \varphi - \tilde{\varphi}$ and subtracting the equations corresponding to $\varphi$ and $\tilde{\varphi}$ yields
    \begin{equation} \label{eq:adjoint-discrete-subtract}
        -\dive [(b + a(y_u)) \nabla \psi]  + \1_{\{ y_u \neq \bar t \}}a'(y_u) \nabla y_u \cdot \nabla \psi  = g_{u,h} \quad \text{in } \Omega, \quad \psi =0 \quad \text{on } \partial\Omega,
    \end{equation}
    with
    \[
        g_{u,h} := \dive [(a(y_u) - a(y_h) ) \nabla \tilde{\varphi} ]+ Z_{y_h, y_u}  \cdot \nabla \tilde{\varphi}.
    \]
    By the chain rule \cite{Gilbarg_Trudinger}, Thm.~7.8 and the fact that $y_u, y_h \in W^{1,\tilde{p}}_0(\Omega)$ and that $\tilde{\varphi} \in H^2(\Omega)$, we can write
    \begin{equation} \label{eq:g-uh}
        g_{u,h} = (a(y_u) - a(y_h)) \Delta \tilde{\varphi}.
    \end{equation}
    Similar to \eqref{eq:adjoint-state-half-estiH2}, there holds
    \[
        \norm{\psi}_{H^2(\Omega)} \leq C_\rho \norm{g_{u,h}}_{L^2(\Omega)} \leq C_\rho \norm{a(y_u) - a(y_h)}_{L^\infty(\Omega)} \norm{\Delta \tilde{\varphi}}_{L^2(\Omega)}.
    \]
    Combining this with the Lipschitz continuity of $a$ on bounded sets, the $L^\infty$-estimate in \eqref{eq:error-state-l2}, and \eqref{eq:adjoint-state-half-estiH2} yields the first estimate in \eqref{eq:adjoint-discrete-half-H1-L2error}.
    To show the second estimate, set $z_{u, \psi} := S'(u)\psi$ and note that $\psi = S'(u)^*g_{u,h}$. We then deduce from  \eqref{eq:g-uh} that
    \begin{equation}
        \label{eq:l2-esti-auxi}
        \norm{\psi}_{L^2(\Omega)}^2  = \int_\Omega g_{u,h} z_{u,\psi} \dx    \leq C_\rho \norm{z_{u, \psi}}_{L^\infty(\Omega)} \norm{\Delta \tilde{\varphi}}_{L^2(\Omega)} \norm{y_u - y_h}_{L^2(\Omega)}.
    \end{equation}
    By \cref{thm:control2state-oper} and the compact embedding $L^2(\Omega) \Subset W^{-1,\tilde{p}}(\Omega)$, we have
    \begin{equation} \label{eq:S-der-bound}
        \sup\{ \norm{S'(u)}_{\Linop(W^{-1,\tilde{p}}(\Omega), W^{1,\tilde{p}}_0(\Omega) ) }\mid u \in \overline B_{L^2(\Omega)}(\bar u, \rho) \}  \leq C_{\rho}.
    \end{equation}
    The continuous embeddings $W^{1,\tilde{p}}_0(\Omega) \hookrightarrow L^\infty(\Omega)$ and $L^2(\Omega) \hookrightarrow W^{-1,\tilde{p}}(\Omega)$ therefore yield
    \begin{equation*}
        \norm{z_{u, \psi}}_{L^\infty(\Omega)}  \leq C\norm{z_{u, \psi}}_{W^{1,\tilde{p}}_0(\Omega)} \leq C_{\rho}  \norm{\psi}_{W^{-1,\tilde{p}}(\Omega)} \leq C_{\rho} \norm{\psi}_{L^2(\Omega)}.
    \end{equation*}
    The inequality
    \eqref{eq:l2-esti-auxi} thus yields
    \[
        \norm{\psi}_{L^2(\Omega)}  \leq C_{\rho}\norm{\Delta \tilde{\varphi}}_{L^2(\Omega)} \norm{y_u - y_h}_{L^2(\Omega)}.
    \]
    This, \eqref{eq:error-state-l2} and \eqref{eq:adjoint-state-half-estiH2} yield the last estimate in \eqref{eq:adjoint-discrete-half-H1-L2error}.
\end{proof}

Below, we shall estimate the term $Z_{y_h,y}$ defined in \eqref{eq:Z-func}. We first observe from the $L^\infty$-error estimate in \eqref{eq:error-state-l2} that
\begin{equation} \label{eq:state-error-infty}
    \norm{S(u) - S_h(u)}_{L^\infty(\Omega)} \leq C_{\infty} h \quad \text{for all } u \in \overline B_{L^2(\Omega)}(\bar u, \rho) \cap \mathcal{U}_{ad}, h \in (0,h_2)
\end{equation}
for some positive constant  $C_{\infty}$.
For any $y \in W^{1,1}(\Omega) \cap C(\overline\Omega)$ and $r >0$, let
\begin{equation} \label{eq:KE-func}
    V(y,r) := \sigma_0 \sum_{m= 1}^{2}   \1_{\{ 0 < |y - \bar t | \leq  r \}} | \partial_{x_m} y | \quad \text{and} \quad \Sigma_r(y) :=  \frac{1}{r}V(y,r)
\end{equation}
with $\sigma_0$ determined as in \eqref{eq:sigma-i}.
\begin{proposition}
    \label{prop:K-E-esti}
    Let $r>0$, $y  \in W^{1,1}(\Omega) \cap C(\overline\Omega)$, and $\hat y \in W^{1,\infty}(\Omega)$  be arbitrary and let $\kappa :=r+ \norm{y - \hat y}_{C(\overline\Omega)}$.  Then
    \begin{enumerate}[label=(\roman*)]
        \item $V(y,r) \leq V(\hat y, \kappa ) + \sigma_0 \sum_{m= 1}^{2}  |\partial_{x_m} y - \partial_{x_m} \hat y|$ for a.e. in $\Omega$;
        \item $\norm{V(\hat y,r)}_{L^2(\Omega)}^2 \leq 2r \sigma_0 \norm{\nabla \hat y}_{L^\infty(\Omega)} \norm{\Sigma_r(\hat y)}_{L^1(\Omega)} $.
    \end{enumerate}
\end{proposition}
\begin{proof}
    The proof of the second claim is straightforward.
    It remains to prove the first assertion. To this end, we now observe that  $\{ 0 < |y - \bar t | \leq r\} \subset \{  |\hat y - \bar t | \leq \kappa\}$  and $| \partial_{x_m} y | \leq | \partial_{x_m} \hat y |  + |\partial_{x_m} y - \partial_{x_m} \hat y|$. There
    thus holds
    \begin{align*}
        V(y,r) & \leq   \sum_{m= 1}^{2} \sigma_0 \1_{\{  |\hat y - \bar t | \leq \kappa\}}[ | \partial_{x_m} \hat y |  + |\partial_{x_m} y - \partial_{x_m} \hat y|] \\
        & =   \sum_{m= 1}^{2} \sigma_0 \1_{\{  |\hat y - \bar t | \leq \kappa\}} | \partial_{x_m} \hat y |  +    \sum_{m= 1}^{2}  \sigma_0 \1_{\{  |\hat y - \bar t | \leq \kappa\}} |\partial_{x_m} y - \partial_{x_m} \hat y|\\
        & = V(\hat y, \kappa ) + \sigma_0 \sum_{m= 1}^{2}   \1_{\{  |\hat y - \bar t | \leq \kappa\}} |\partial_{x_m} y - \partial_{x_m} \hat y|,
    \end{align*}
    where  we have employed the fact that $\nabla \hat y$ vanishes a.e. in $\{ \hat y = \bar t \}$ in order to obtain the last identity.  This yields the first claim.
\end{proof}
\begin{lemma}
    \label{lem:delta}
    There exist an $h_3 \in (0, h_2]$ and  a constant $L_{\rho} >0$ such that for all
    $u \in \overline B_{L^2(\Omega)}(\bar u, \rho) \cap \mathcal{U}_{ad}$
    and $h \in (0, h_3)$, there hold
    \begin{align} \label{eq:nonsmooth}
        \norm{ Z_{y_h,y}}_{L^2(\Omega)} &\leq L_{\rho} h +   \norm{V(y, \norm{y_h - y}_{L^\infty(\Omega)})}_{L^2(\Omega)}
        \shortintertext{and}
        \norm{ Z_{y,\bar y}}_{L^2(\Omega)} &\leq L_{\rho}\norm{u - \bar u}_{L^2(\Omega)} + \norm{V(\bar y, \norm{y - \bar y}_{L^\infty(\Omega)})}_{L^2(\Omega)}
        \label{eq:nonsmooth-bary}
    \end{align}
    with $y :=S(u)$ and $y_h := S_h(u)$.

\end{lemma}
\begin{proof}
    By \cref{thm:control2state-oper} (also, see, Theorems 3.1 and 3.5 in \cite{ClasonNhuRosch_os2nd}), there exists a constant $M_{1,\rho}$ such that
    \begin{equation} \label{eq:C1-bound}
        \norm{S(u)}_{W^{1,\infty}(\Omega)} \leq M_{1,\rho} \quad \text{for all }u \in \overline B_{L^2(\Omega)}(\bar u, \rho) \cap \mathcal{U}_{ad}.
    \end{equation}
    Setting $h_3 := \min\{h_2,\delta 2^{-1} C_{\infty}^{-1}  \}$ and exploiting \eqref{eq:state-error-infty} shows that $\norm{y - y_h}_{C(\overline{\Omega})} < \delta$  for any $u \in \overline B_{L^2(\Omega)}(\bar u, \rho) \cap \mathcal{U}_{ad}$ and $h \in (0, h_3)$.
    From the definition of $Z_{y_h,y}$ in \eqref{eq:Z-func} and \cref{lem:Z-decomposition}, we arrive at
    \begin{equation} \label{eq:nonsmooth-decom}
        Z_{y_h,y} = Z_{y_h,y}^{(1)} + Z_{y_h,y}^{(2)} + Z_{y_h,y}^{(3)} +Z_{y_h,y}^{(4)}.
    \end{equation}
    By \eqref{eq:Z-esti}, \eqref{eq:C1-bound}, and \cref{thm:error-state}, we have
    \begin{equation}
        \label{eq:error-nonsmooth1}
        \norm{Z_{y_h,y}^{(k)}}_{L^2(\Omega)} \leq L_{\rho} h \quad \text{for } k=1,2,4.
    \end{equation}
    On the other hand, we have
    \begin{equation}
        \label{eq:Omega-23-esti}
        \left\{
            \begin{aligned}
                \Omega_{y_h,y}^{2} &= \{ y \in (\bar t, \bar t + \delta), y_h \in (\bar t -\delta, \bar t ] \} \subset  \{ 0 < y - \bar t  \leq \norm{y_h - y}_{L^\infty(\Omega)} \},\\
                \Omega_{y_h,y}^{3} &= \{ y \in (\bar t -\delta, \bar t ), y_h \in [\bar t ,\bar t +\delta)\} \subset  \{ 0 < \bar t  - y \leq \norm{y_h - y}_{L^\infty(\Omega)} \},
            \end{aligned}
        \right.
    \end{equation}
    which together with the definitions of $Z_{y_h,y}^{(3)}$ in \cref{lem:Z-decomposition} and of $V$ in \eqref{eq:KE-func}, and \eqref{eq:sigma-i}, show that
    \begin{equation}
        \label{eq:Z3-esti}
        |Z_{y_h,y}^{(3)}| \leq V(y, \norm{y_h - y}_{L^\infty(\Omega)})
    \end{equation}
    a.e. in $\Omega$.
    Combining this with \eqref{eq:nonsmooth-decom} and \eqref{eq:error-nonsmooth1}, we obtain \eqref{eq:nonsmooth}.

    For \eqref{eq:nonsmooth-bary}, we first see that
    \[
        \norm{Z_{y,\bar y}^{(k)}}_{L^2(\Omega)} \leq C \norm{y - \bar y}_{H^1_0(\Omega)} \leq C\norm{u - \bar u}_{L^2(\Omega)} \quad \text{for } k=1,2,4,
    \]
    which is similar to \eqref{eq:error-nonsmooth1} and is derived by using \eqref{eq:Z-esti}, \eqref{eq:C1-bound},  \cref{thm:control2state-oper}, and the continuity of the mapping $S: L^2(\Omega) \hookrightarrow W^{-1,\tilde{p}}(\Omega) \to W^{1,\tilde{p}}_0(\Omega) \hookrightarrow H^1_0(\Omega)$. Finally, similar to \eqref{eq:Z3-esti}, one has
    \begin{equation*}
        |Z_{y,\bar y}^{(3)}| \leq V(y, \norm{y - \bar y}_{L^\infty(\Omega)})
    \end{equation*}
    a.e. in $\Omega$. We thus obtain \eqref{eq:nonsmooth-bary}.
\end{proof}

\medskip

From now on, let us fix constants $\gamma, p_0$, and $p_1$ satisfying
\begin{equation}
    \label{eq:gamma-exponent}
    1 - \frac{2}{p_0} \leq \gamma < \min\left\{ \frac{1}{p_0}, 1 - \frac{2}{p_1} \right\} < \frac{1}{2}, \quad p_0 > 2 = N, \quad \text{and} \quad p_1 \in (2,p_*),
\end{equation}
where $p_*$ is given in \cref{thm:control2state-oper}.

In order to derive an $L^2$-error estimate of $\tilde{\varphi} - \varphi_h$, we cannot directly employ a duality argument based on the Aubin--Nitsche trick since the linearized state equation \eqref{eq:deri} admits solutions belonging to $W^{1,p}(\Omega)$ only due to the nondifferentiability of the function $a$; see \cref{rem:regularity-derivative-S}. To overcome this difficulty, we now consider the following \emph{adjusted} linearized state equation:
\begin{equation} \label{eq:adjusted-deri}
    -\dive [(b + a(y_u)) \nabla \tilde z_{u,v} + \1_{\{ \bar y \neq \bar t \}}a'(\bar y) \tilde z_{u,v} \nabla \bar y ]  = v \, \text{in } \Omega,  \quad  \tilde z_{u,v} =0 \, \text{ on } \partial\Omega
\end{equation}
for all $u \in  \mathcal{U}_{ad}$ and $v$ belonging to $W^{-1+\gamma,2}(\Omega)$, the dual space  of a Sobolev space of fractional order.
Compared to \eqref{eq:deri}, we have here partly linearized the state equation \eqref{eq:state} at the optimal state $\bar y$ by replacing the vector-valued function $\1_{\{ y_u \neq \bar t \}}a'(y_u) \nabla y_u$ in \eqref{eq:deri} by $\1_{\{ \bar y \neq \bar t \}}a'(\bar y) \nabla \bar y$.
We can therefore exploit assumptions imposed on $\bar y$ to derive the necessary regularity of solutions $\tilde{z}_{u,v}$ to \eqref{eq:adjusted-deri}.
Indeed, as we will see later in \cref{prop:fu-regularity-origin} in \cref{sec:linearized-regularity}, the finiteness of $\Sigma(\bar y)$ implies that
\[
    \1_{\{ \bar y \neq \bar t \}}a'(\bar y) \nabla \bar y \in (W^{\gamma,p_0}(\Omega))^2 = (W^{\gamma,p_0}_0(\Omega))^2,
\]
where the constants $\gamma$ and $p_0$ are fixed and satisfy \eqref{eq:gamma-exponent}.
From this and the $W^{1 + \gamma,2}$-regularity of solutions to \eqref{eq:adjusted-deri}, we can then show that $\tilde z_{u,v} \in W^{1 + \gamma,2}(\Omega)$ whenever $v \in W^{-1+\gamma,2}(\Omega)$. This regularity of solutions to  \eqref{eq:adjusted-deri} will play an important role in establishing a priori error estimates for the discretization of the adjoint state equation \eqref{eq:adjoint-state}.

\begin{theorem}
    \label{thm:adjoint-discrete}
    If $\Sigma(\bar y) < \infty$, then there exist constants  $\bar h := \bar h(\bar u) \in (0,h_3)$,
    $\bar \rho := \bar\rho(\bar u) \leq \min\{ \rho, \hat{\rho}, \rho_{\tilde{p}} \}$
    (with  $\hat\rho$ and $\rho_{\tilde p}$ being the constants in \cref{prop:adjusted-deri-regularity} associated with $p:= \tilde{p}$),
    and $C_{\bar u}>0$ such that
    \begin{equation*}
        \norm{\varphi - \varphi_h}_{L^2(\Omega)} \leq C_{\bar u} \epsilon_h^{u} \norm{v}_{L^2(\Omega)} \quad \text{and} \quad \norm{\varphi - \varphi_h}_{H^1_0(\Omega)} \leq C_{\bar u}h\norm{v}_{L^2(\Omega)}
    \end{equation*}
    for all $h \in (0, \bar h)$,
    $u \in  \overline B_{L^2(\Omega)}(\bar u, \bar\rho) \cap \mathcal{U}_{ad}$,
    and $v \in L^2(\Omega)$, where
    \begin{equation}
        \label{eq:epsilon-h}
        \epsilon_h^{u} :=
        h^{1+\gamma} + h \norm{Z_{S(u), \bar y}}_{L^2(\Omega)}
        + h \norm{V(S(u), \norm{S_h(u) - S(u)}_{L^\infty(\Omega)})}_{L^2(\Omega)},
    \end{equation}
    and $\varphi$ and $\varphi_h$ are the unique solutions to \eqref{eq:adjoint-state} and \eqref{eq:adjoint-state-discrete}, respectively.
\end{theorem}
\begin{proof}
    Let $\tilde \varphi$ be the solution of \eqref{eq:adjoint-discrete-half}. To simplify the notation, set $y_u := S(u)$ and $y_h := S_h(u)$ for any $h \in (0, h_3)$ and
    $u \in  \overline B_{L^2(\Omega)}(\bar u, \min\{ \rho, \hat{\rho}, \rho_{\tilde{p}}\}) \cap \mathcal{U}_{ad}$.
    We divide the proof into three steps.

    \medskip

    \noindent\emph{%
        Step 1: Existence of a constant $C_{1, h_3, \rho}$ such that
        \begin{equation}
            \label{eq:step1}
            \norm{\tilde \varphi - \varphi_h }_{L^2(\Omega)} \leq C_{1, h_3, \rho} (
            h^{\gamma} + \norm{Z_{y_u, \bar y}}_{L^2(\Omega)}
            + \norm{V(y_u, \norm{y_h - y_u}_{L^\infty(\Omega)})}_{L^2(\Omega)} ) \norm{\tilde \varphi - \varphi_h }_{H^1_0(\Omega)}
        \end{equation}
        for all $h \in (0, h_3)$,
        $u \in  \overline B_{L^2(\Omega)}(\bar u, \min\{ \rho, \hat{\rho},\rho_{\tilde{p}}\}) \cap \mathcal{U}_{ad}$,
    and $v \in L^2(\Omega)$.}%

    To prove \eqref{eq:step1}, first
    let $\hat z$ be the unique solution to the adjusted linearized state equation \eqref{eq:adjusted-deri} corresponding to $v:= \tilde \varphi -\varphi_h$.
    Since $\tilde \varphi -\varphi_h \in H^1_0(\Omega)$, one has $\tilde \varphi -\varphi_h \in W^{-1+\gamma,2}(\Omega)$ due to the embeddings  $H^1_0(\Omega) \hookrightarrow L^2(\Omega) \hookrightarrow W^{-1+\gamma,2}(\Omega)$.
    From \cref{prop:adjusted-deri-regularity} and the finiteness of $\Sigma(\bar y)$, we have that
    $\hat z \in W^{1+\gamma,2}(\Omega)$ and that
    \begin{equation}
        \label{eq:Wgamma-hatz-esti}
        \norm{\hat z}_{W^{1+\gamma,2}(\Omega)} \leq C\norm{\tilde \varphi -\varphi_h}_{W^{-1+\gamma,2}(\Omega)} \leq C\norm{\tilde \varphi -\varphi_h}_{L^2(\Omega)}.
    \end{equation}
    Moreover, we also have that
    \begin{equation}
        \label{eq:hatz-deri-form}
        \hat z = S'(u)(\tilde \varphi -\varphi_h - \dive( Z_{y_u, \bar y}\hat z));
    \end{equation}
    see the identity \eqref{eq:zuv-deri-form} in the proof of \cref{prop:adjusted-deri-regularity} for $v := \tilde \varphi - \varphi_h$.
    Testing  \eqref{eq:deri} for $v:= \tilde \varphi -\varphi_h - \dive( Z_{y_u, \bar y}\hat z)$ by $\tilde{\varphi}- \varphi_h$ and using \eqref{eq:hatz-deri-form} thus yields
    \begin{multline} \label{eq:hatz-varphi}
        \norm{\tilde \varphi -\varphi_h}_{L^2(\Omega)}^2 + \int_\Omega Z_{y_u, \bar y} \cdot \nabla (\tilde \varphi -\varphi_h ) \hat z \dx \\
        = \int_\Omega (b+a(y_u)) \nabla \hat z \cdot \nabla (\tilde \varphi -\varphi_h ) + \1_{\{ y_u \neq \bar t\}} a'(y_u) \hat z \nabla y_u \cdot \nabla (\tilde \varphi -\varphi_h ) \dx.
    \end{multline}
    Furthermore, applying assertion \ref{item:zuv-W1p-esti} in \cref{prop:adjusted-deri-regularity} for $p := \tilde{p}$ and using the embedding $L^2(\Omega) \hookrightarrow W^{-1,\tilde{p}}(\Omega)$ shows that
    \begin{equation} \label{eq:z-apriori}
        \norm{\hat z}_{W^{1,\tilde{p}}_0(\Omega)} \leq C_{\rho} \norm{\tilde \varphi - \varphi_h}_{L^2(\Omega)} \quad \text{for all } h \in (0, h_3),
        u \in B_{L^2(\Omega)}(\bar u, \rho_{\tilde{p}})
        \cap \mathcal{{U}}_{ad}, v \in L^2(\Omega).
    \end{equation}
    Consider for any
    $u \in  \overline B_{L^2(\Omega)}(\bar u, \min\{ \rho, \hat{\rho},\rho_{\tilde{p}}\}) \cap \mathcal{U}_{ad}$
    and $h \in (0,h_3)$ the bilinear operators $B_u: H^1_0(\Omega) \times H^1_0(\Omega) \to \R$ and $B_{u,h}: H^1_0(\Omega) \times H^1_0(\Omega) \to \R$ defined via
    \begin{align*}
        B_u(z,w) &:= \int_\Omega [ ( b + a(y_u) ) \nabla z + \1_{\{y_u \neq \bar t \}}a'(y_u) z \nabla y_u ] \cdot \nabla w \dx, \\
        B_{u,h}(z,w) &:= \int_\Omega [ ( b + a(y_h) ) \nabla z + \1_{\{y_h \neq \bar t \}}a'(y_h) z \nabla y_h ] \cdot \nabla w \dx.
    \end{align*}
    From this
    and \eqref{eq:hatz-varphi},
    we obtain
    for any $w_h \in V_h$ that
    \begin{equation*}
        \begin{aligned}
            \norm{\tilde \varphi - \varphi_h}_{L^2(\Omega)}^2
            +
            \int_\Omega Z_{y_u, \bar y} \cdot \nabla (\tilde \varphi -\varphi_h ) \hat z \dx
            &= B_u(\hat z, \tilde \varphi - \varphi_h) = B_u(\hat z - w_h, \tilde \varphi - \varphi_h) + B_u(w_h, \tilde \varphi - \varphi_h) \\
            & = B_u(\hat z - w_h, \tilde \varphi - \varphi_h) + [ B_u(w_h, \tilde \varphi - \varphi_h) - B_{u,h}(w_h, \tilde \varphi - \varphi_h) ],
        \end{aligned}
    \end{equation*}
    where we have used the fact that $B_{u,h}(w_h, \tilde \varphi - \varphi_h) = 0$ which follows from combining \eqref{eq:adjoint-state-discrete} with \eqref{eq:adjoint-discrete-half}.
    We now estimate the first integral in the left-hand side and two  summands in the right-hand side of the above identity.
    From Hölder's inequality, there holds
    \begin{multline} \label{eq:varphi-auxi}
        \norm{\tilde \varphi - \varphi_h}_{L^2(\Omega)}^2
        -
        \norm{Z_{y_u, \bar y}}_{L^2(\Omega)} \norm{\hat z}_{L^\infty(\Omega)}   \norm{\tilde \varphi -  \varphi_h}_{H^1_0(\Omega)}
        \\
        \begin{aligned}[t]
            &\leq \norm{b+ a(y_u)}_{L^\infty(\Omega)} \norm{\hat z - w_h}_{H^1_0(\Omega)} \norm{\tilde \varphi -  \varphi_h}_{H^1_0(\Omega)} \\
            \MoveEqLeft[-1]  +  \norm{\1_{\{y_u \neq \bar t \}}a'(y_u)\nabla y_u}_{L^\infty(\Omega)}  \norm{\hat z - w_h}_{L^2(\Omega)} \norm{\tilde \varphi -  \varphi_h}_{H^1_0(\Omega)} \\
            \MoveEqLeft[-1]  + \norm{a(y_u) - a(y_h)}_{L^4(\Omega)} \norm{w_h}_{W^{1,4}_0(\Omega)}   \norm{\tilde \varphi -  \varphi_h}_{H^1_0(\Omega)}  \\
            & \quad  + \norm{Z_{y_h,y_u}}_{L^2(\Omega)} \norm{w_h}_{L^\infty(\Omega)}   \norm{\tilde \varphi -  \varphi_h}_{H^1_0(\Omega)}
        \end{aligned}
    \end{multline}
    for all $w_h \in V_h$.
    Moreover, we have from \cref{ass:PC1-func}, the continuous embedding $H^{N/4}(\Omega) \hookrightarrow L^4(\Omega)$ with $N=2$ (see, e.g. \cite{Grisvard1985}, Thm.~1.4.4.1), interpolation theory \cite{BrennerScott2008}, Thm.~14.2.7, and \eqref{eq:error-state-l2} that
    \begin{equation} \label{eq:interpolation-theory}
        \begin{aligned}[t]
            \norm{a(y_u) - a(y_h)}_{L^4(\Omega)}
            & \leq C\norm{y_u -y_h}_{L^4(\Omega)} \leq  C \norm{y_u -y_h}_{H^{N/4}(\Omega)}  \\
            & \leq C \norm{y_u -y_h}_{L^2(\Omega)}^{1 - \frac{2}{4}} \norm{y_u -y_h}_{H^1_0(\Omega)}^{\frac{2}{4}}  \leq C h^{\frac{3}{2}}.
        \end{aligned}
    \end{equation}
    We then deduce from this, \eqref{eq:varphi-auxi}, \eqref{eq:C1-bound}, the assumptions on $b$ and $a$, and \eqref{eq:nonsmooth} that
    \begin{multline} \label{eq:step1-esti1}
        \norm{\tilde \varphi - \varphi_h}_{L^2(\Omega)}^2 \leq C_{\rho}\norm{\tilde \varphi -  \varphi_h}_{H^1_0(\Omega)} \Big[ \norm{\hat z - w_h}_{H^1_0(\Omega)} + h^{\frac{3}{2}} \norm{w_h}_{W^{1,4}_0(\Omega)}\\
            + (h + \norm{V(y_u, \norm{y_h - y_u}_{L^\infty(\Omega)})}_{L^2(\Omega)})\norm{w_h}_{L^\infty(\Omega)} +
            \norm{Z_{y_u, \bar y}}_{L^2(\Omega)} \norm{\hat z}_{L^\infty(\Omega)}
        \Big]
    \end{multline}
    for all $w_h \in V_h$.
    Moreover, from standard interpolation error estimates (see, e.g. Theorem 4.4.20 in \cite{BrennerScott2008}) and applying the estimates (4.4.21) for $(p,s,m) := (4,1,1)$ and (4.4.22) for $(p,s,m,n) := (\tilde{p}, 0,1,2)$, we have that
    \begin{equation} \label{eq:interpolation-error}
        \left\{
            \begin{aligned}
                & \norm{\hat z - \Pi_h \hat z}_{W^{1,4}_0(\Omega)} \leq C_1  \norm{\hat z}_{W^{1,4}_0(\Omega)} \leq C_1  \norm{\hat z}_{W^{1,\tilde{p}}_0(\Omega)},
                \\
                & \norm{\hat z - \Pi_h \hat z}_{L^\infty(\Omega)} \leq C_1 h^{1 -\frac{2}{\tilde{p}}} \norm{\hat z}_{W^{1,\tilde{p}}_0(\Omega)},
            \end{aligned}
        \right.
    \end{equation}
    where we have used the fact that
    $\tilde{p} \geq 4$
    to derive the first line in \eqref{eq:interpolation-error}.
    On the other hand, from Theorem 6.1 in \cite{DupontScott1980} and arguments similar to the ones in Example~3 in \cite{DupontScott1980}, we also obtain that
    \begin{equation} \label{eq:interpolation-error-gamma}
        \norm{\hat z - \Pi_h \hat z}_{H^1_0(\Omega)} \leq C_2 h^{\gamma} \norm{\hat z}_{W^{1+\gamma,2}(\Omega)};
    \end{equation}
    see also the error estimates for the interpolation operator $\Pi_h: W^{1+\gamma,2}(\Omega) \to V_h$ in the proof of Lemma 3.1 in \cite{ReyesDhamo2016}.
    Combining \eqref{eq:interpolation-error} with \eqref{eq:interpolation-error-gamma}, the triangle inequality, and  the embedding $W^{1,\tilde{p}}_0(\Omega) \hookrightarrow L^\infty(\Omega) \cap W^{1,4}_0(\Omega)$ (due to $\tilde{p} \geq 4$) as well as estimates \eqref{eq:Wgamma-hatz-esti} and \eqref{eq:z-apriori}  gives
    \begin{equation*}
        \norm{\hat z - \Pi_h \hat z}_{H^1_0(\Omega)} \leq C_3 h^{\gamma}\norm{\tilde \varphi - \varphi_h}_{L^2(\Omega)}, \quad \norm{\hat z}_{L^\infty(\Omega)} + \norm{\Pi_h \hat z}_{L^\infty(\Omega)} + \norm{\Pi_h \hat z}_{W^{1,4}_0(\Omega)} \leq C_4 \norm{\tilde \varphi - \varphi_h}_{L^2(\Omega)}.
    \end{equation*}
    By choosing $w_h := \Pi_h \hat z$ in    \eqref{eq:step1-esti1} and using these above estimates and the fact that $\gamma < \frac{1}{2}$, we obtain
    \eqref{eq:step1}.

    \medskip

    \noindent\emph{Step 2: Existence of a constant $C_{2, h_3, \rho}$ such that
        \begin{equation}
            \label{eq:step2}
            \norm{\tilde \varphi - \varphi_h }_{H^1_0(\Omega)} \leq C_{2, h_3, \rho} ( h \norm{v}_{L^2(\Omega)} +  \norm{\tilde \varphi - \varphi_h }_{L^2(\Omega)} )
        \end{equation}
    for all $h \in (0, h_3)$, $u \in  \overline B_{L^2(\Omega)}(\bar u, \rho) \cap \mathcal{U}_{ad}$, and $v \in L^2(\Omega)$}.

    To show this, we first consider for any $u \in  \overline B_{L^2(\Omega)}(\bar u, \rho) \cap \mathcal{U}_{ad}$ and $h \in (0, h_3)$, the bilinear mapping $S_{u, h}: H^1_0(\Omega) \times H^1_0(\Omega) \to \R$ defined by
    \begin{equation*}
        S_{u,h}(\omega, \psi) := \int_\Omega ( b + a(y_h) ) \nabla \omega \cdot \nabla \psi \dx.
    \end{equation*}
    From \cref{ass:b_func,ass:PC1-func}, we obtain
    \begin{equation} \label{eq:step2-auxi1}
        \underline{b} \norm{\tilde \varphi - \varphi_h}_{H^1_0(\Omega)}^2  \leq S_{u,h}(\tilde \varphi - \varphi_h,\tilde \varphi - \varphi_h) = S_{u,h}(\tilde \varphi - \varphi_h,\tilde \varphi - \Pi_h \tilde \varphi) + S_{u,h}(\tilde \varphi - \varphi_h,\Pi_h \tilde \varphi - \varphi_h).
    \end{equation}
    Moreover, the Cauchy--Schwarz inequality, the uniform boundedness of $\{y_h\}$ on $C(\overline\Omega)$, and \cref{ass:b_func,ass:PC1-func} yield that
    \begin{equation} \label{eq:step2-auxi2}
        \begin{aligned}[t]
            S_{u,h}(\tilde \varphi - \varphi_h,\tilde \varphi - \Pi_h \tilde \varphi) & \leq C_{\rho} \norm{\tilde \varphi - \varphi_h }_{H^1_0(\Omega)}\norm{\tilde \varphi - \Pi_h \tilde \varphi }_{H^1_0(\Omega)} \\
            & \leq C_{\rho} h \norm{\tilde \varphi}_{H^2(\Omega)} \norm{\tilde \varphi - \varphi_h }_{H^1_0(\Omega)} \leq C_{\rho} h \norm{v}_{L^2(\Omega)}\norm{\tilde \varphi - \varphi_h }_{H^1_0(\Omega)},
        \end{aligned}
    \end{equation}
    where we have exploited the interpolation error \cite{Ciarlet2002} and \eqref{eq:adjoint-state-half-estiH2} in order to obtain the last two estimates. Now using \eqref{eq:adjoint-state-discrete} and \eqref{eq:adjoint-discrete-half}, we deduce from Hölder's inequality  that
    \begin{equation*}
        \begin{aligned}
            S_{u,h}(\tilde \varphi - \varphi_h,\Pi_h \tilde \varphi - \varphi_h) & = - \int_\Omega \1_{\{y_h \neq \bar t\}} a'(y_h) \nabla y_h \cdot \nabla (\tilde{\varphi} - \varphi_h) (\Pi_h \tilde \varphi - \varphi_h )\dx \\
            &\leq \norm{\1_{\{y_h \neq \bar t \}} a'(y_h)}_{L^\infty(\Omega)} \norm{ \nabla y_h }_{L^4(\Omega)} \norm{\tilde{\varphi} - \varphi_h}_{H^{1}_0(\Omega)} \norm{\Pi_h \tilde \varphi - \varphi_h}_{L^4(\Omega)}.
        \end{aligned}
    \end{equation*}
    Combing this with the uniform boundedness in $C(\overline\Omega)$ of $\{y_h\}$ and the embedding $W^{1,\tilde{p}}_0(\Omega) \hookrightarrow W^{1,4}_0(\Omega)$, we obtain that
    \begin{equation*}
        S_{u,h}(\tilde \varphi - \varphi_h,\Pi_h \tilde \varphi - \varphi_h) \leq C_{\rho}( \norm{\bar y}_{W^{1,\tilde{p}}_0(\Omega)} + \kappa_\rho ) \norm{\tilde{\varphi} - \varphi_h}_{H^{1}_0(\Omega)} \norm{\Pi_h \tilde \varphi - \varphi_h}_{L^4(\Omega)}.
    \end{equation*}
    The combination of a triangle inequality and the embedding $H^1_0(\Omega) \hookrightarrow L^4(\Omega)$ with Theorem 3.1.6 in \cite{Ciarlet2002} further implies that
    \begin{equation*}
        \begin{aligned}
            \norm{\Pi_h \tilde \varphi - \varphi_h}_{L^4(\Omega)}  \leq \norm{\Pi_h \tilde \varphi -\tilde \varphi}_{L^4(\Omega)} + \norm{ \tilde \varphi - \varphi_h}_{L^4(\Omega)}
            &\leq
            C\norm{\Pi_h \tilde \varphi -\tilde \varphi}_{H^1_0(\Omega)} + \norm{ \tilde \varphi - \varphi_h}_{L^4(\Omega)}
            \\
            &\leq C h \norm{\tilde{\varphi}}_{H^2(\Omega)} +  \norm{ \tilde \varphi - \varphi_h}_{L^4(\Omega)}\\
            &\leq C h \norm{v}_{L^2(\Omega)}  + \norm{ \tilde \varphi - \varphi_h}_{L^4(\Omega)},
        \end{aligned}
    \end{equation*}
    where we have used  \eqref{eq:adjoint-state-half-estiH2} to obtain the last inequality.
    Similar to \eqref{eq:interpolation-theory}, we find that
    \begin{equation*}
        \norm{ \tilde \varphi - \varphi_h}_{L^4(\Omega)} \leq C \norm{ \tilde \varphi - \varphi_h}_{L^2(\Omega)}^{1- \frac{2}{4}}\norm{ \tilde \varphi - \varphi_h}_{H^1_0(\Omega)}^{\frac{2}{4}}.
    \end{equation*}
    We then have
    \begin{equation*}
        S_{u,h}(\tilde \varphi - \varphi_h,\Pi_h \tilde \varphi - \varphi_h) \leq C_{\rho} \norm{\tilde{\varphi} - \varphi_h}_{H^{1}_0(\Omega)} \left( h  \norm{v}_{L^2(\Omega)} + \norm{ \tilde \varphi - \varphi_h}_{L^2(\Omega)}^{\frac{1}{2}}\norm{ \tilde \varphi - \varphi_h}_{H^1_0(\Omega)}^{\frac{1}{2}} \right),
    \end{equation*}
    which, together with \eqref{eq:step2-auxi1} and \eqref{eq:step2-auxi2}, yields
    \begin{equation*}
        \norm{\tilde \varphi - \varphi_h}_{H^1_0(\Omega)} \leq C_{\rho} \left( h \norm{v}_{L^2(\Omega)} +  \norm{ \tilde \varphi - \varphi_h}_{L^2(\Omega)}^{\frac{1}{2}}\norm{ \tilde \varphi - \varphi_h}_{H^1_0(\Omega)}^{\frac{1}{2}}\right).
    \end{equation*}
    Applying the Cauchy--Schwarz inequality then gives \eqref{eq:step2}.

    \medskip

    \noindent\emph{Step 3: Existence of constants $\bar h \in (0,h_3)$ and
        $\bar \rho := \bar\rho(\bar u) \leq \min\{ \rho, \hat{\rho}, \rho_{\tilde{p}} \}$%
    }.

    To show this, we first obtain from the definition of $\Sigma(\bar y)$ in \eqref{eq:E-functional} and of $\Sigma_r$ in \eqref{eq:KE-func} the existence of a $r_* >0$ such that
    $\norm{\Sigma_r(\bar y)}_{L^1(\Omega)} \leq \Sigma(\bar y) +1$ for all $r \in (0, r_*)$.
    This together with (ii) in \cref{prop:K-E-esti} yields
    \begin{equation} \label{eq:K-E-bary}
        \norm{V(\bar y,r)}_{L^2(\Omega)} \leq C_{\bar y}r^{1/2}(\Sigma(\bar y) +1)^{1/2} \quad \text{for all } r \in (0 ,r_*).
    \end{equation}
    Moreover, thanks to \cref{thm:control2state-oper} and the embeddings $L^2(\Omega) \hookrightarrow W^{-1,p}(\Omega)$ and $W^{1,p}_0(\Omega) \hookrightarrow C(\overline\Omega) \cap H^1_0(\Omega)$ for some $p>2$, one has
    \[
        \norm{y_u - \bar y}_{L^\infty(\Omega)} + \norm{y_u - \bar y}_{H^1_0(\Omega)} \leq C_\rho \norm{u - \bar u}_{L^2(\Omega)}
    \]
    for all $u \in  \overline B_{L^2(\Omega)}(\bar u, \rho) \cap \mathcal{U}_{ad}$ and some constant $C_{\rho}$.
    Now \cref{prop:K-E-esti}\,(i), \eqref{eq:state-error-infty},  and the monotonic growth of $V(y,\cdot)$ imply that
    \begin{equation} \label{eq:esti1}
        \begin{aligned}[t]
            \norm{V(y_u, \norm{y_h - y_u}_{L^\infty(\Omega)})}_{L^2(\Omega)}
            & \leq \norm{V(\bar y, \norm{y_h - y_u}_{L^\infty(\Omega)} + \norm{y_u - \bar y}_{L^\infty(\Omega)} )}_{L^2(\Omega)} + C_{\bar y, \rho}\norm{y_u - \bar y}_{H^1_0(\Omega)} \\
            & \leq \norm{V(\bar y, C_\infty h + C_{\rho} \norm{u - \bar u}_{L^2(\Omega)})}_{L^2(\Omega)} + C_{\bar y, \rho}C_{\rho} \norm{u - \bar u}_{L^2(\Omega)}
        \end{aligned}
    \end{equation}
    for some constant $ C_{\bar y, \rho}$ and for all $h \in (0, h_3)$ and $u \in  \overline B_{L^2(\Omega)}(\bar u, \rho) \cap \mathcal{U}_{ad}$.
    Besides, from \eqref{eq:nonsmooth-bary} and the monotonic growth of $V(\bar y,\cdot)$, there holds
    \begin{align*}
        \norm{Z_{y_u, \bar y}}_{L^2(\Omega)} & \leq L_\rho \norm{u - \bar u}_{L^2(\Omega)}+ \norm{V(\bar y,  \norm{y_u - \bar y}_{L^\infty(\Omega)} )}_{L^2(\Omega)} \\
        &  \leq L_\rho \norm{u - \bar u}_{L^2(\Omega)}+ \norm{V(\bar y,  C_\rho \norm{u - \bar u}_{L^2(\Omega)})}_{L^2(\Omega)}
    \end{align*}
    for all $u \in  \overline B_{L^2(\Omega)}(\bar u, \rho) \cap \mathcal{U}_{ad}$. We then have
    \begin{equation*}
        \label{eq:Z-V-esti}
        \norm{Z_{y_u, \bar y}}_{L^2(\Omega)} + \norm{V(y_u, \norm{y_h - y_u}_{L^\infty(\Omega)})}_{L^2(\Omega)} \leq C_{1, \rho} \norm{u - \bar u}_{L^2(\Omega)} + 2\norm{V(\bar y, C_\infty h + C_{\rho} \norm{u - \bar u}_{L^2(\Omega)})}_{L^2(\Omega)}
    \end{equation*}
    for all $h \in (0, h_3)$ and $u \in  \overline B_{L^2(\Omega)}(\bar u, \rho) \cap \mathcal{U}_{ad}$ with $C_{1,\rho} := (C_{\bar y, \rho}C_{\rho} + L_\rho)$.

    Next, fix $\bar h \in (0, h_3)$ and
    $\bar \rho \leq \min\{ \rho, \hat{\rho}, \rho_{\tilde{p}} \}$
    such that
    \[
        \bar{h}^{\gamma} + 2C_{\bar y} [ C_\infty \bar{h} + C_{\rho}\bar \rho  ]^\frac{1}{2}(\Sigma(\bar y) +1)^{\frac{1}{2}} +
        C_{1,\rho} \bar\rho
        \leq \frac{1}{2 C_{1,h_3,\rho} C_{2,h_3,\rho}} \quad \text{and} \quad    C_\infty \bar{h} + C_{\rho}\bar \rho < r_*,
    \]
    where $ C_{1,h_3,\rho}$ and $C_{2,h_3,\rho}$ are defined in \eqref{eq:step1} and \eqref{eq:step2}, respectively. From this, \eqref{eq:K-E-bary}, and \eqref{eq:esti1}, we conclude that
    \begin{equation*}
        \bar{h}^{\gamma} +\norm{Z_{y_u, \bar y}}_{L^2(\Omega)} + \norm{V(y_u, \norm{y_h - y_u}_{L^\infty(\Omega)})}_{L^2(\Omega)}   \leq \frac{1}{2 C_{1,h_3,\rho} C_{2,h_3,\rho}}
    \end{equation*}
    for all  $h \in (0, \bar h), u \in  \overline B_{L^2(\Omega)}(\bar u, \bar \rho) \cap \mathcal{U}_{ad}$.

    The combination of this with \eqref{eq:step1} and \eqref{eq:step2} yields
    $\norm{\tilde{\varphi} - \varphi_h}_{H^1_0(\Omega)}  \leq 2C_{2,h_3,\rho}h \norm{v}_{L^2(\Omega)}$,
    and together with \eqref{eq:step1}, we obtain
    $\norm{\tilde{\varphi} - \varphi_h}_{L^2(\Omega)}  \leq C\epsilon_h^{u} \norm{v}_{L^2(\Omega)}$.
    Combining the last two estimates with \cref{lem:adjoint-disrete-half} and the triangle inequality, we arrive at the desired conclusion.
\end{proof}

\section{Discretization of the control problem} \label{sec:P-descretized}
In this section, we discretize the control problem \eqref{eq:P}, show convergence of the discretizations, and derive error estimates of the discrete optimal solutions.
In the following, we will consider three different discretizations of the control:
\begin{enumerate}[label=(\roman*)]
    \item \emph{variational discretization}: $\mathcal{U}_h = L^\infty(\Omega)$ (see, e.g., \cite{Hinze2005});
    \item \emph{piecewise constant discretization}:
        \begin{equation*}
            \mathcal{U}_h = \mathcal{U}_h^0 := \left\{ u \in L^\infty(\Omega)  \mid  u_{|T} \in \R  \text{ for all } T \in \mathcal{T}_h \right\};
        \end{equation*}
    \item
        \emph{continuous piecewise linear discretization}:
        \begin{equation*}
            \mathcal{U}_h = \mathcal{U}_h^1 := \left\{ u \in L^\infty(\Omega)  \mid  u_{|T}  \in \mathcal{P}_1 \text{ for all } T \in \mathcal{T}_h \right\}.
        \end{equation*}
\end{enumerate}
Unless specified, any claim for $\mathcal{U}_h$ should be understood to hold for all three cases.
For any $h >0$, we now set $\mathcal{U}_{ad,h} := \mathcal{U}_{ad} \cap \mathcal{U}_h$.
If $\mathcal{U}_h = \mathcal{U}_h^0$, then by $\mathcal{I}_h$ we denote the  linear projection from $L^2(\Omega)$ onto $\mathcal{U}_h^0$. If $\mathcal{U}_h = \mathcal{U}_h^1$, then $\mathcal{I}_h: L^2(\Omega) \to \mathcal{U}_h^1$ denotes the Carstensen quasi-interpolation operator \cite{Carstensen1999}. In both situations, we have $\mathcal{I}_h u \to u$ strongly in $L^2(\Omega)$ as $h \to 0$ for all $u \in L^2(\Omega)$ and $\mathcal{I}_h u \in \mathcal{U}_{ad,h}$ for all $u \in \mathcal{U}_{ad}$; see, e.g. \cite{ReyesMeyerVexler2008}.

\medskip
For any $h >0$, we define the discretized optimal control problem
\begin{equation}
    \label{eq:discrete-P-nonreduced}
    \tag{$P_h$}
    \min\{ J(y_h,u_h) : u_h \in \mathcal{U}_{ad,h} \, \text{and } y_h \in V_h    \, \text{satisfies \eqref{eq:state-discrete} for $u:= u_h$} \}
\end{equation}
with
\[
    J(y,u) := \int_\Omega L(x,y(x)) \dx + \frac{\nu}{2} \norm{u}^2_{L^2(\Omega)}.
\]
Note that the discrete operator $S_h$ defined in \eqref{eq:discrete-operator} does not appear in \eqref{eq:discrete-P-nonreduced}, since this operator is well-defined only locally by \cref{thm:uniqueness-state}.

\subsection{Convergence of discrete minimizers} \label{sec:convergence}
We first have the convergence of minimizers of \eqref{eq:discrete-P-nonreduced}.
\begin{theorem}[{ cf. Theorem 4.1 in \cite{CasasMateosRosch2021}}]
    \label{thm:convergence-discrete-solutions}
    Assume that $a$ satisfies the growth condition
    \begin{equation}
        \label{eq:growth-condition-a}
        |a(t)| \leq C_0 + C_1 |t|^m \quad \text{for all } t \in \R
    \end{equation}
    for some positive constants $C_0, C_1$ and $m \geq 1$.
    Then there exists an $\tilde{h}_0 >0$ such that \eqref{eq:discrete-P-nonreduced} admits at least one global minimizer $(\bar y_h, \bar u_h)$ for all $0 < h < \tilde{h}_0$.
    Moreover, if $\{(\bar y_h,\bar u_h)\}_{0 < h < \tilde{h}_0}$ is a sequence of solutions to \eqref{eq:discrete-P-nonreduced}, then there exists a subsequence that converges strongly in $H^1_0(\Omega) \times L^2(\Omega)$ to some $(\bar{y}, \bar{u})$ as $h \to 0$, where $\bar {u}$ is a global solution to \eqref{eq:P}.
\end{theorem}
\begin{proof}
    The existence of discrete solutions to \eqref{eq:discrete-P-nonreduced} is proven similarly to Claim 1 in the proof of Theorem 4.1 in \cite{CasasMateosRosch2021}.
    Moreover, there exists a constant $\tilde{h}_0 >0$ such that \eqref{eq:discrete-P-nonreduced} admits at least one minimizer $(\bar y_h, \bar u_h)$ for all $0 < h < \tilde{h}_0$.
    The remainder of this proof is now divided into three steps as follows: \\
    \noindent\emph{$\bullet$ Claim 1: Weak convergence of $\{(\bar y_h, \bar u_h)\}$ to $(\bar y, \bar u)$ satisfying \eqref{eq:state}.} Indeed, the boundedness of $\{\bar u_h\}$ in $L^\infty(\Omega)$ and thus in $L^2(\Omega)$ is due to the $L^\infty$-boundedness of $\mathcal{U}_{ad}$. Since $(\bar y_h, \bar u_h)$ satisfies \eqref{eq:state-discrete}, there holds
    \begin{equation}
        \label{eq:state-discrete-for-uh}
        \int_{\Omega} (b + a(\bar y_h) ) \nabla \bar y_h \cdot \nabla v_h \dx = \int_{\Omega} \bar u_h v_h \dx\qquad\text{for all } v_h\in V_h.
    \end{equation}
    In particular, one has
    \[
        \int_\Omega (b + a(\bar y_h)) \nabla \bar y_h \cdot \nabla \bar y_h \dx = \int_\Omega \bar y_h \bar u_h \dx.
    \]
    Combining this with \cref{ass:PC1-func} and \cref{ass:b_func}, we deduce from the Cauchy--Schwarz and Poincar\'{e} inequalities the boundedness of $\{\bar y_h\}$ in $H^1_0(\Omega)$. From this and the compact embedding $H^1_0(\Omega) \Subset L^p(\Omega)$ for any $p \geq 1$, we can take a subsequence, denoted in the same way, of $\{(\bar y_h,\bar u_h)\}_{0 < h < \tilde{h}_0}$ that satisfies
    \begin{equation}
        \label{eq:yh-uh-limit}
        (\bar y_h, \bar u_h) \rightharpoonup (\bar y, \bar u) \, \text{in } H^1_0(\Omega) \times L^2(\Omega), \quad \bar y_h(x) \to \bar y(x) \, \text{for a.e. } x \in \Omega, \text{and} \quad \bar y_h \to \bar y \, \text{in } L^p(\Omega)
    \end{equation}
    for some $(\bar y, \bar u) \in H^1_0(\Omega) \times \mathcal{U}_{ad}$ and for any $p \geq 1$, for instance, $p :=4m$ with constant $m$ defined in \eqref{eq:growth-condition-a}. Thanks to the growth condition \eqref{eq:growth-condition-a}, we conclude from the last limit and the generalized Lebesgue Dominated Convergence Theorem that
    \begin{equation}
        \label{eq:ay-limit}
        a(\bar y_h) \to a(\bar y) \quad \text{strongly in } L^4(\Omega).
    \end{equation}
    Now take any $v \in H^1_0(\Omega) \cap W^{2,4}(\Omega)$ and choose $v_h := \Pi_hv \in V_h$. Then $v_h \to v$ in $W^{1,4}(\Omega)$ as $h\to 0$; see, e.g. \cite{Ciarlet2002}. Letting $h \to 0$ in the equation \eqref{eq:state-discrete-for-uh} and exploiting the first limit in \eqref{eq:yh-uh-limit} as well as the limit in \eqref{eq:ay-limit} then yields
    \[
        \int_\Omega (b+ a(\bar y)) \nabla \bar y \cdot \nabla v \dx = \int_\Omega \bar u v \dx \quad \text{for all } v \in H^1_0(\Omega) \cap H^2(\Omega),
    \]
    which, together with the density of $H^1_0(\Omega) \cap W^{2,4}(\Omega)$ in $H^1_0(\Omega)$, implies that $\bar y = S(\bar u)$. \\
    \noindent\emph{$\bullet$ Claim 2: Optimality of $\bar u$.} Let us show that $\bar u$ is a global solution of \eqref{eq:P2}. For that purpose, we first observe from the inclusion $\mathcal{U}_{ad,h} \subset \mathcal{U}_{ad}$ and from the first limit in \eqref{eq:yh-uh-limit} that $\bar u \in \mathcal{U}_{ad}$.
    Take $u \in \mathcal{U}_{ad}$ arbitrarily and choose $u_h :=u$ if $\mathcal{U}_h = L^\infty(\Omega)$ and $u_h := \mathcal{I}_h u$ if $\mathcal{U}_h = \mathcal{U}_h^i$ with $i=0,1$. One has $u_h \in \mathcal{U}_{ad,h}$ and $u_h \to u$ strongly in $L^2(\Omega)$ as $h \to 0$.
    In view of \cref{thm:error-state}, for $h$ small enough there exists at least one solution $y_h(u_h)$ of \eqref{eq:state-discrete} such that $y_h(u_h) \to y_u := S(u)$ strongly in $H^{1}_0(\Omega) \cap C(\overline\Omega)$.
    From this and the optimality of $(\bar y_h, \bar u_h)$, we have
    \begin{equation} \label{eq:optimality-u-esti}
        \liminf\limits_{h \to 0} J(\bar y_h, \bar u_h) \leq \liminf\limits_{h \to 0} J( y_h(u_h),  u_h) \leq \limsup\limits_{h \to 0} J( y_h(u_h),  u_h) = J(y_u,u) = j(u).
    \end{equation}
    On the other hand, it follows from the  limits in \eqref{eq:yh-uh-limit} and the weak lower semicontinuity of the $L^2$-norm that
    \begin{equation}
        \label{eq:optimality-u-esti2}
        \liminf\limits_{h \to 0} J(\bar y_h, \bar u_h) \geq J(\bar y, \bar u) = j(\bar u),
    \end{equation}
    which together with \eqref{eq:optimality-u-esti} gives $j(\bar u) \leq j(u)$. Since $u \in \mathcal{U}_{ad}$ was arbitrary, $\bar u$ is a global optimal solution to \eqref{eq:P2}.  \\
    \noindent\emph{$\bullet$ Claim 3: Strong convergence in $H^1_0(\Omega) \times L^2(\Omega)$.} In fact, by plugging $u:= \bar u$ into \eqref{eq:optimality-u-esti} and \eqref{eq:optimality-u-esti2} and using the limits for $\bar y_h$ in \eqref{eq:yh-uh-limit}, we can conclude that
    $\norm{\bar u_h}_{L^2(\Omega)}^2 \to \norm{\bar u}_{L^2(\Omega)}^2$. Combining this with the limit for $\bar u_h$ in \eqref{eq:yh-uh-limit} yields
    \begin{equation}
        \label{eq:strong-limit-u}
        \bar u_h \to \bar u \quad \text{strongly in } L^2(\Omega).
    \end{equation}
    It remains to prove the strong convergence of $\{\bar y_h\}$ in $H^1_0(\Omega)$. To this end, by the weak lower semicontinuity of the functional $H^{1}_0(\Omega) \ni y    \mapsto  \int_\Omega (b(x) + a(y(x))) \nabla y(x) \cdot \nabla y(x) \dx \in \R \cup \{ \infty \}$; see, e.g. Theorem 1.3 in \cite{Dacorogna}, we deduce from the first limit in \eqref{eq:yh-uh-limit} that
    \begin{equation}
        \label{eq:weak-lower-of-a}
        \int_\Omega (b+ a(\bar y)) \nabla \bar y \cdot \nabla \bar y \dx \leq \liminf\limits_{h \to 0} \int_\Omega (b+ a(\bar y_h)) \nabla \bar y_h \cdot \nabla \bar y_h \dx.
    \end{equation}
    Moreover, \eqref{eq:state-discrete-for-uh} and \eqref{eq:strong-limit-u} imply that
    \begin{align*}
        \liminf\limits_{h \to 0} \int_\Omega (b+ a(\bar y_h)) \nabla \bar y_h \cdot \nabla \bar y_h \dx & \leq \limsup \limits_{h \to 0} \int_\Omega (b+ a(\bar y_h)) \nabla \bar y_h \cdot \nabla \bar y_h \dx \\
        & = \limsup \limits_{h \to 0} \int_\Omega \bar u_h \bar y_h \dx = \int_\Omega \bar y \bar u\dx \\
        & = \int_\Omega (b+ a(\bar y)) \nabla \bar y \cdot \nabla \bar y \dx.
    \end{align*}
    Combining this with \eqref{eq:weak-lower-of-a}, one has
    \begin{equation}
        \label{eq:yh-strong-limit-esti1}
        \lim\limits_{h \to 0} \int_\Omega (b+ a(\bar y_h)) \nabla \bar y_h \cdot \nabla \bar y_h \dx = \int_\Omega (b+ a(\bar y)) \nabla \bar y \cdot \nabla \bar y \dx.
    \end{equation}
    Moreover, thanks to \cref{ass:b_func} and by the nonnegativity of $a$, there hold $b(x) - \frac{\underline{b}}{2} \geq \frac{\underline{b}}{2}>0$ and $\frac{\underline{b}}{2} + a(\bar y_h(x)) \geq \frac{\underline{b}}{2} >0$ for all $x \in \Omega$. We then deduce from the  limit in \eqref{eq:strong-limit-u}, the weak limit of $\{\bar y_h \}$ in $H^1_0(\Omega)$, and Theorem 1.3 in \cite{Dacorogna} that
    \begin{align*}
        \int_\Omega (b(x) - \frac{\underline{b}}{2})|\nabla \bar  y|^2 \dx & \leq \liminf\limits_{h \to 0} \int_\Omega (b(x) - \frac{\underline{b}}{2}) \nabla \bar y_h \cdot \nabla \bar y_h \dx \\
        & \leq \limsup\limits_{h \to 0} \int_\Omega (b(x) - \frac{\underline{b}}{2}) \nabla \bar y_h \cdot \nabla \bar y_h \dx \\
        & = \limsup\limits_{h \to 0} \left[ \int_\Omega \bar u_h \bar y_h \dx - \int_\Omega (\frac{\underline{b}}{2} + a(\bar y_h)) \nabla \bar y_h \cdot \nabla \bar y_h \dx \right] \\
        & = \int_\Omega \bar u \bar y \dx - \liminf \limits_{h \to 0}\int_\Omega (\frac{\underline{b}}{2} + a(\bar y_h)) \nabla \bar y_h \cdot \nabla \bar y_h \dx \\
        & \leq \int_\Omega \bar u \bar y \dx - \int_\Omega (\frac{\underline{b}}{2}+ a(\bar y)) \nabla \bar y \cdot \nabla \bar y \dx \\
        & = \int_\Omega (b(x) - \frac{\underline{b}}{2})|\nabla \bar y|^2 \dx,
    \end{align*}
    where we have employed the equations for $\bar y_h$ and $\bar y$ to derive the first and the last identities. We therefore have
    \begin{equation*}
        \lim\limits_{h \to 0} \int_\Omega \left(b(x) - \frac{\underline{b}}{2}\right) \nabla \bar y_h \cdot \nabla \bar y_h \dx  = \int_\Omega \left(b(x) - \frac{\underline{b}}{2}\right)|\nabla \bar  y|^2 \dx,
    \end{equation*}
    which, along with the weak limit, yields, the strong convergence of $\{\bar y_h\}$ in $H^1_0(\Omega)$.
\end{proof}

Next, we prove a kind of converse theorem. More precisely, we assume that $\bar u \in  \mathcal{U}_{ad}$
is a \emph{strict local minimum} of \eqref{eq:P} with associated state $\bar y$, i.e. there exists a constant $\bar \epsilon >0$ such that
\begin{equation*}
    j(\bar u) < j(u) \quad \text{for all } u \in \overline B_{L^2(\Omega)}(\bar u,\bar \epsilon) \cap \mathcal{U}_{ad} \, \text{with } u \neq \bar u.
\end{equation*}
We can obviously assume that $\bar \epsilon < \rho$. Here $\rho$ is defined in \cref{thm:uniqueness-state}. We therefore can put the discrete operator $S_h$ into \eqref{eq:discrete-P-nonreduced}. Then, for any $h \in (0, h_2)$, we consider the discretized optimal control problem defined via
\begin{equation}
    \label{eq:discrete-P}
    \tag{$P_h^{\bar \epsilon}$}
    \min_{ u_h \in \mathcal{U}_{ad,h} \cap \overline B_{L^2(\Omega)}(\bar u,\bar \epsilon)} j_h(u_h)
\end{equation}
and the discretized cost functional given by
\begin{equation*}
    j_h : \overline B_{L^2(\Omega)}(\bar u, \rho)\to\R,\qquad j_h(u) := \int_\Omega L(x, (S_h(u))(x)) \dx + \frac{\nu}{2} \norm{u}_{L^2(\Omega)}^2.
\end{equation*}
Using \cref{thm:diff-discrete-state} and \cref{ass:cost_func}, we can show differentiability of $j_h$. The proof of the following result is straightforward and therefore omitted.
\begin{theorem}
    \label{thm:diff-cost-discrete}
    For any $h \in (0, h_2)$, the discrete cost functional $j_h: \overline B_{L^2(\Omega)}(\bar u, \rho) \to \R$ is of class $C^1$, and its derivative at $u \in B_{L^2(\Omega)}(\bar u, \rho)$ is given by
    \begin{equation*}
        j_h'(u)w = \int_\Omega ( \varphi_h(u) + \nu u ) w \dx \qquad \text{for all }w\in L^2(\Omega),
    \end{equation*}
    where $\varphi_h(u) \in V_h$ is the unique solution to \eqref{eq:adjoint-state-discrete} with $v :=\frac{\partial L}{\partial y}(\cdot, S_h(u))$.
\end{theorem}

Since $\mathcal{I}_h \bar u \to \bar u$ strongly in $L^2(\Omega)$ as $h \to 0$, there exists a constant $\tilde{h} = \tilde{h}(\bar \epsilon) >0$ such that the admissible set of \eqref{eq:discrete-P} is nonempty for all $0 < h < \tilde{h}$.
We now provide a result on the existence of global minimizers and the associated optimality conditions of \eqref{eq:discrete-P}. Its proof is elementary and is thus omitted.
\begin{theorem}
    \label{thm:1st-OS-discrete}
    There exists a constant $h_* \in (0, h_2)$ such that for any $h \in (0,h_*)$,  \eqref{eq:discrete-P} admits at least one global minimizer $\bar u_h\in\mathcal{U}_{ad,h} \cap \overline B_{L^2(\Omega)}(\bar u,\bar \epsilon)$. Moreover,  there exists a function $\bar \varphi_h \in V_h$ that together with $\bar u_h$ and $\bar y_h := S_h(\bar u_h)$ satisfies
    \begin{subequations}
        \label{eq:1st-OS-discrete}
        \begin{align}
            & \int_{\Omega}  (b + a(\bar y_h) ) \nabla \bar \varphi_h \cdot \nabla w_h + \1_{\{\bar y_h \neq \bar t \}}a'(\bar y_h) w_h \nabla \bar y_h  \cdot \nabla \bar  \varphi_h \dx = \int_\Omega \frac{\partial L}{\partial y}(x, \bar y_h)w_h \dx, \label{eq:OS1st-adjoint-discrete} \\
            & \int_\Omega ( \bar \varphi_h + \nu \bar u_h )(u_h - \bar u_h) \dx \geq 0 \quad \text{for all } w_h \in V_h, u_h \in \mathcal{U}_{ad,h} \cap \overline B_{L^2(\Omega)}(\bar u,\bar \epsilon). \label{eq:OS1st-normal-discrete}
        \end{align}
    \end{subequations}
\end{theorem}

We now state a convergence result in $L^2(\Omega)$, whose proof is similar to that of Theorem 4.2 in \cite{ReyesDhamo2016} and is thus omitted here.
\begin{theorem}
    \label{thm:convergence-L2}
    Let $\{\bar u_h\}$ be the sequence of discrete solutions to \eqref{eq:discrete-P}, defined in \cref{thm:1st-OS-discrete}. Then $\norm{\bar u_h - \bar u}_{L^2(\Omega)} \to 0$ as $h \to 0^+$.
\end{theorem}

\begin{remark} \label{rem:projection-simplified}
    By \cref{thm:convergence-L2}, it holds that $\norm{\bar u_h - \bar u}_{L^2(\Omega)}  \leq \frac{{\bar \epsilon}}{2}$ for all $h \in (0, \tilde{h})$ and for some $\tilde{h}>0$.
    Now for any $h \in (0, \tilde{h})$ and $u_h \in \mathcal U_{ad,h}$, we have that $w_h := t(u_h - \bar u_h) + \bar u_h \in \overline B_{L^2(\Omega)}(\bar u_h, \frac{{\bar \epsilon}}{2})$ for $t>0$ small enough and hence that $w_h \in \mathcal{U}_{ad,h} \cap \overline B_{L^2(\Omega)}(\bar u, {\bar \epsilon})$. The variational inequality \eqref{eq:OS1st-normal-discrete} then implies that
    \begin{equation*}
        \int_\Omega ( \bar \varphi_h + \nu \bar u_h )(u_h - \bar u_h) \dx \geq 0 \quad \text{for all }  u_h \in \mathcal U_{ad,h}.
    \end{equation*}
\end{remark}

\begin{remark}
    \label{rem:coincidence-minimizer}
    In view of \cref{thm:convergence-L2}, there exists a constant $h_{\bar \epsilon} >0$ such that any solution $\bar u_h$ of \eqref{eq:discrete-P} belongs to the open ball $B_{L^2(\Omega)}(\bar u, \bar \epsilon)$. By \cref{thm:uniqueness-state}, $(S_h(\bar u_h), \bar u_h)$ is thus a local minimizer of \eqref{eq:discrete-P-nonreduced}. We have therefore shown that any strict local solution of \eqref{eq:P} can be approximated by local optimal controls of \eqref{eq:discrete-P-nonreduced}.
\end{remark}

In order to show convergence in $L^\infty(\Omega)$, we first need the following lemma.
\begin{lemma}
    \label{lem:continuity-discrete}
    Let $q >2$ be given and let $\bar h$ and $\bar \rho$ be defined in \cref{thm:adjoint-discrete}.
    If $\Sigma(\bar y) <\infty$, then for any $h \in (0, \bar h)$ and $u, v \in \mathcal{U}_{ad}$ such that $v \in B_{L^2(\Omega)}(\bar u, \bar\rho) \cap \mathcal{U}_{ad}$, there hold
    \begin{align}
        & \norm{y_u - y_h(v)}_{L^2(\Omega)} + h\norm{y_u - y_h(v)}_{H^1_0(\Omega)}  \leq C (h^2 + \norm{u-v}_{L^2(\Omega)} ), \label{eq:continuity-discrete-state} \\
        & \norm{\varphi_u - \varphi_v}_{H^2(\Omega)} \leq C \norm{u-v}_{L^2(\Omega)}, \label{eq:continuityH2-adjoint-discrete-noh}\\
        & \norm{\varphi_u - \varphi_h(v)}_{L^2(\Omega)} \leq C ( \epsilon_h^{v} + \norm{u-v}_{L^2(\Omega)} ), \label{eq:continuityL2-adjoint-discrete}\\
        & \norm{\varphi_u - \varphi_h(v)}_{H^1_0(\Omega)} \leq C (h  + \norm{u-v}_{L^2(\Omega)} ), \label{eq:continuityH1-adjoint-discrete}  \\
        & \norm{\varphi_u - \varphi_h(v)}_{L^\infty(\Omega)} \leq C (h^{1- \frac{2}{q}}  + \norm{u-v}_{L^2(\Omega)}), \label{eq:continuityLinfty-adjoint-discrete}
    \end{align}
    for some constant $C$ independent of $u, v$, and $h$.
    Here $y_u :=S(u)$ and $y_h(v) := S_h(v)$, while
    $\varphi_u$ is the unique solution to \eqref{eq:adjoint-state}  for $v := \frac{\partial L}{\partial y}(\cdot, y_u)$ and  $\varphi_h(v)$ is the unique solution to \eqref{eq:adjoint-state-discrete} for $y_h := S_h(v)$ and  $v :=\frac{\partial L}{\partial y}(\cdot, y_h(v))$.
\end{lemma}
\begin{proof}
    First, a standard argument yields \eqref{eq:continuity-discrete-state}.
    For the other estimates, let $\varphi_{v,h}$ be the solution to \eqref{eq:adjoint-state} for $v:=\frac{\partial L}{\partial y}(\cdot, y_h(v))$ and $y_u$ replaced by $y_v := S(v)$.
    We need to show that
    \begin{equation} \label{eq:continuity-auxi2}
        \norm{\varphi_{v,h} - \varphi_u}_{H^2(\Omega)} \leq C ( \norm{{u-v}}_{L^2(\Omega)} + h^2 )
    \end{equation}
    for some constant $C>0$ independent of $u$, $v$, and $h$. To this end, we subtract the equations for $\varphi_u$ and $\varphi_{v,h}$ to obtain that $\varphi_{v,h} - \varphi_u \in H^1_0(\Omega)$ and
    \begin{equation}
        \label{eq:continuity-auxi2-diffeq}
        -\dive [(b + a(y_v) )\nabla (\varphi_{v,h} - \varphi_u) ] + \1_{\{ y_v \neq \bar t \} } a'(y_v) \nabla y_{v} \cdot \nabla (\varphi_{v,h} - \varphi_u) =  g_{u,v,h}
    \end{equation}
    with
    \begin{equation*}
        g_{u,v,h} :=- \dive [( a(y_u) - a(y_v) ) \nabla \varphi_u ] + Z_{y_u,y_v}\cdot \nabla \varphi_u + \frac{\partial L}{\partial y}(\cdot, y_h(v)) - \frac{\partial L}{\partial y}(\cdot, y_u).
    \end{equation*}
    \Cref{thm:adjoint-equation} and \cref{ass:cost_func} imply that $\varphi_u\in H^2(\Omega)$. From this, the product formula, the chain rule \cite{Gilbarg_Trudinger}, and the finiteness of the set $\{\bar t\}$, we deduce that
    \begin{equation*}
        \dive [( a(y_u) - a(y_v) ) \nabla \varphi_u ] = Z_{y_u,y_v}  \cdot \nabla \varphi_u + ( a(y_u) - a(y_v) ) \Delta \varphi_u.
    \end{equation*}
    This shows that
    \begin{equation*}
        g_{u,v,h} = ( a(y_v) - a(y_u) ) \Delta \varphi_u + \frac{\partial L}{\partial y}(\cdot, y_h(v)) - \frac{\partial L}{\partial y}(\cdot, y_u) \in L^2(\Omega).
    \end{equation*}
    The standard stability estimate for the solution $\varphi_{v,h}-\varphi_u$ to \eqref{eq:continuity-auxi2-diffeq} thus gives
    \begin{equation*}
        \norm{\varphi_{v,h} - \varphi_u}_{H^2(\Omega)} \leq C \norm{g_{u,v,h}}_{L^2(\Omega)} \leq C [\norm{y_u - y_v}_{L^\infty(\Omega)} + \norm{y_{h}(v) - y_u}_{L^2(\Omega)}  ]
    \end{equation*}
    for some constant $C>0$ not depending on $u$, $v$, and $h$, where we have employed the boundedness of $\{S(w) \mid w \in \mathcal{U}_{ad} \}$ in $C(\overline\Omega)$, the fact that $\norm{\Delta \varphi_u}_{L^2(\Omega)} \leq C$ due to \cref{thm:adjoint-equation},  and \cref{ass:PC1-func,ass:cost_func} to derive the last estimate. From this, \eqref{eq:continuity-discrete-state}, and the fact that $\norm{y_u - y_v}_{L^\infty(\Omega)} \leq C \norm{u-v}_{L^2(\Omega)}$, we obtain \eqref{eq:continuity-auxi2}.
    The estimate \eqref{eq:continuityH2-adjoint-discrete-noh} is shown by a similar argument.

    We now prove \eqref{eq:continuityL2-adjoint-discrete}--\eqref{eq:continuityLinfty-adjoint-discrete}. According to the triangle inequality, \eqref{eq:continuity-auxi2}, \cref{thm:adjoint-discrete}, and the boundedness in $L^2(\Omega)$ of $\{ \frac{\partial L}{\partial y}(\cdot, y_h(v)) \mid v \in B_{L^2(\Omega)}(\bar u, \bar \rho) \cap \mathcal{U}_{ad}, h \leq h_2 \}$, we obtain \eqref{eq:continuityL2-adjoint-discrete} and \eqref{eq:continuityH1-adjoint-discrete}. Finally, for \eqref{eq:continuityLinfty-adjoint-discrete}, we first see from the continuous embedding $H^1_0(\Omega) \hookrightarrow L^q(\Omega)$, the interpolation error and inverse estimates \cite{BrennerScott2008} for $N=2$ that
    \begin{equation*}
        \begin{aligned}
            \norm{\varphi_{v,h} - \varphi_h(v)}_{L^\infty(\Omega)} &\leq \norm{\varphi_{v,h} -\Pi_h \varphi_{v,h}}_{L^\infty(\Omega)} + \norm{\Pi_h \varphi_{v,h} - \varphi_h(v)}_{L^\infty(\Omega)} \\
            & \leq C h^{1} \norm{\varphi_{v,h}}_{H^2(\Omega)} + C h^{-\frac{2}{q}}  \norm{\Pi_h \varphi_{v,h} - \varphi_h(v)}_{L^q(\Omega)}\\
            & \leq C h^{1} \norm{\varphi_{v,h}}_{H^2(\Omega)} + C h^{-\frac{2}{q}}  \norm{\Pi_h \varphi_{v,h} - \varphi_h(v)}_{H^1_0(\Omega)}\\
            & \leq C \left[ h^{1} \norm{\varphi_{v,h}}_{H^2(\Omega)} +  h^{-\frac{2}{q}}  \norm{\Pi_h \varphi_{v,h} - \varphi_{v,h }}_{H^1_0(\Omega)} +  h^{- \frac{2}{q}}  \norm{\varphi_{v,h} - \varphi_h(v)}_{H^1_0(\Omega)} \right],
        \end{aligned}
    \end{equation*}
    which together with \cref{thm:adjoint-discrete} and the interpolation error estimate from Theorem 4.4.20 in \cite{BrennerScott2008} yields
    \begin{equation*}
        \norm{\varphi_{v,h} - \varphi_h(v)}_{L^\infty(\Omega)}  \leq C \left[ h^{1 - \frac{2}{q}} \norm{\varphi_{v,h}}_{H^2(\Omega)}  +  h^{1 - \frac{2}{q}}  \norm{\frac{\partial L}{\partial y}(\cdot, y_h(v))}_{L^2(\Omega)} \right].
    \end{equation*}
    Combining this with the uniform boundedness of $\varphi_{v,h}$ in $H^2(\Omega)$ and of $\frac{\partial L}{\partial y}(\cdot, y_h(v))$ in $L^2(\Omega)$ for all  $v \in B_{L^2(\Omega)}(\bar u, \bar \rho) \cap \mathcal{U}_{ad}$ and $h \leq h_2$, we conclude that $\norm{\varphi_{v,h} - \varphi_h(v)}_{L^\infty(\Omega)}  \leq C h^{1 - \frac{2}{q}}$.
    From this, \eqref{eq:continuity-auxi2}, and the embedding $H^2(\Omega) \hookrightarrow L^\infty(\Omega)$, the triangle inequality thus leads to \eqref{eq:continuityLinfty-adjoint-discrete}.
\end{proof}

From \cref{thm:convergence-L2}, \cref{rem:coincidence-minimizer}, and the estimate \eqref{eq:continuityLinfty-adjoint-discrete} in \cref{lem:continuity-discrete}, we obtain the desired convergence result in $L^\infty(\Omega)$. Its proof is similar to that of Theorem 5.3 in \cite{CasasTroltzsch2011} with some modifications and it is thus omitted.
\begin{theorem}
    \label{thm:convergence-Linfty}
    Let $\{(\bar y_h,\bar u_h)\}$ be the sequence of discrete solutions to \eqref{eq:discrete-P-nonreduced} converging strongly to $(\bar y, \bar u)$ in $H^1_0(\Omega) \times L^2(\Omega)$.
    If $\Sigma(\bar y) < \infty$, then $\norm{\bar u_h - \bar u}_{L^\infty(\Omega)} \to 0$ as $h \to 0^+$.
\end{theorem}

\subsection{Error estimates for discrete minimizers} \label{sec:error-estimate}

We finally turn to error estimates for discrete local minimizers $\bar u_h$ under the second-order sufficient optimality condition \eqref{eq:2nd-OC-suff-explicit}. We need the following technical lemma.
\begin{lemma}
    \label{lem:limit-key}
    Let $\bar u \in \mathcal{U}_{ad}$ and $\bar y:= S(\bar u)$.
    Assume that
    $\{ \bar y = \bar t\}$ decomposes into finitely many connected components and that on each such connected component $\Cu$, either \eqref{eq:main-hypothesis-nonvanishing-gradient} or \eqref{eq:main-hypothesis-y-vanish-structure} is fulfilled.
    Let $\{v_n\} \subset  L^2(\Omega)$, $v \in L^2(\Omega)$,
    $\varphi \in C^1(\overline\Omega) \cap W^{2,1}(\Omega)$,
    and $\{s_n\} \in c_0^+$ be arbitrary such that $\bar u + s_n v_n \in \mathcal{U}_{ad}$, $\norm{v_n}_{L^2(\Omega)} = 1$ and $v_n \rightharpoonup v$ in $L^2(\Omega)$.
    Setting $u_n:= \bar u + s_n v_n$, $y_n :=S(u_n)$ for $n \geq 1$ and $w:= S'(\bar u)v$, then the following assertions hold:
    \begin{enumerate}[label=(\alph*)]
        \item \label{item:Z-func-limit} If $w_n \to  w$ in $C(\overline\Omega)$, then
            \begin{multline*}
                \limsup \limits_{n \to \infty} \frac{1}{s_n}  \int_\Omega Z_{y_n,\bar y}  \cdot \nabla  \varphi w_{n}  \dx = -    [a'_{0}(\bar t) - a'_{1}(\bar t)]\int_{\{\bar y=\bar t\}} \1_{\{ |\nabla \bar y| > 0 \}}  w^2  \frac{\nabla \bar y \cdot \nabla   \varphi}{|\nabla \bar y|} \dH^{1}(x)  \\
                + \int_\Omega [ \1_{\{ \bar y \neq \bar t \}} a''(\bar y)w^2 \nabla \bar y + a'(\bar y; w) \nabla w ] \cdot \nabla  \varphi \dx.
            \end{multline*}
        \item \label{item:key-limit-es} If in addition \eqref{eq:adjoint_OS} is fulfilled
            for $\bar \varphi := \varphi$,
            then
            \begin{equation*}
                \liminf\limits_{n \to \infty} \frac{j'(u_n)v_n - j'(\bar u)v_n}{s_n} = 2 Q(\bar u,\bar y,  \varphi; v)+\nu \left( 1- \norm{v}_{L^2(\Omega)}^2\right).
            \end{equation*}
    \end{enumerate}
\end{lemma}
\begin{proof}
    We first observe from \cref{thm:control2state-oper} and the embedding $L^2(\Omega) \hookrightarrow W^{-1,\tilde{p}}(\Omega)$  that
    \begin{equation}\label{eq:S-deri-keylem}
        \frac{y_n - \bar y}{s_n} \to w = S'(\bar u)v \, \text{in } W^{1,\tilde{p}}_0(\Omega)\hookrightarrow C(\overline\Omega).
    \end{equation}
    Moreover, by \cref{thm:control2state-oper} and the compact embedding $W^{2,p}(\Omega) \Subset C^1(\overline{\Omega})$ for some $p>2$, we deduce from the boundedness in $L^\infty(\Omega)$ of $\mathcal{U}_{ad}$ that
    \begin{equation}
        \label{eq:C1-convergence}
        y_n \to \bar y \quad \text{strongly in} \quad C^1(\overline\Omega).
    \end{equation}
    \emph{Ad \ref{item:Z-func-limit}:}
    We first deduce for $n\in \N$ large enough that
    \begin{equation}
        \label{eq:distance}
        \tau_n := \norm{y_n - \bar y}_{C(\overline\Omega)} \leq C s_n < \delta,
    \end{equation}
    where $\delta$ is the constant defined in \eqref{eq:Omega-123-sets}.
    Moreover, there exists a constant $M>0$ such that $\norm{y_n}_{C(\overline\Omega)}, \norm{\bar y}_{C(\overline\Omega)} \leq M$ for all $n \geq 1$.
    According to \cref{lem:Z-decomposition}, we have
    \begin{equation}
        \label{eq:Z-decom-error}
        Z_{y_n,\bar y} = Z_{y_n,\bar y}^{(1)} +Z_{y_n,\bar y}^{(2)}+Z_{y_n,\bar y}^{(3)}+Z_{y_n,\bar y}^{(4)}
    \end{equation}
    with $Z_{y_n,\bar y}^{(k)}$, $k = 1,2,3,4$, defined in \cref{lem:Z-decomposition}.
    For $Z_{y_n,\bar y}^{(1)}$, using \eqref{eq:distance} yields
    \begin{align*}
        Z_{y_n,\bar y}^{(1)} &= \1_{\{ \bar y \in (-\infty, \bar t), y_n \in (-\infty, \bar t) \}} [ a_0'(y_n) \nabla y_n - a_0'(\bar y) \nabla \bar y ] + \1_{\{ \bar y \in (\bar t, \infty), y_n \in (\bar t, \infty) \}} [ a_1'(y_n) \nabla y_n - a_1'(\bar y) \nabla \bar y ],\\
        & =[ \1_{\{\bar y \in (\infty, \bar t) \}} - \1_{ \{ \bar y \in (-\infty,\bar t), y_n \in  [\bar t, \bar t+ \delta)   \} } ]
        [ (a_{0}'(y_n) - a_0'(\bar y) ) \nabla y_n + a_0'(\bar y) \nabla (y_n - \bar y)   ] \\
        & \qquad + [ \1_{\{\bar y \in (\bar t, \infty) \}} - \1_{ \{ \bar y \in (\bar t, \infty), y_n \in (\bar t- \delta, \bar t]   \} } ]
        [ (a_{1}'(y_n) - a_1'(\bar y) ) \nabla y_n + a_1'(\bar y) \nabla (y_n - \bar y)   ].
    \end{align*}
    Since
    \[
        \1_{ \{ \bar y \in (-\infty,\bar t), y_n \in  [\bar t, \bar t+ \delta)   \} }, \1_{ \{ \bar y \in (\bar t, \infty), y_n \in (\bar t- \delta, \bar t]   \} } \leq \1_{\{ 0< |\bar y - \bar t| \leq \tau_n\} } \to 0
    \]
    a.e. in $\Omega$, we have from \eqref{eq:S-deri-keylem}, \eqref{eq:C1-convergence}, and the Lebesgue dominated convergence theorem that
    \begin{equation*}
        \begin{aligned}
            \frac{1}{s_n}\int_\Omega Z_{y_n,\bar y}^{(1)} \cdot \nabla \varphi w_n \dx  &\to  \int_\Omega  \1_{\bar y \in (-\infty, \bar t)} [ a_0''(\bar y)w \nabla \bar y + a_0'(\bar y) \nabla w  ] \cdot \nabla  \varphi w \dx  \\
            \MoveEqLeft[-1]+\int_\Omega  \1_{\bar y \in (\bar t, \infty)} [ a_1''(\bar y)w \nabla \bar y+ a_1'(\bar y) \nabla w  ] \cdot \nabla  \varphi w \dx \\
            &= \int_\Omega \1_{\{ \bar y \neq \bar t \}} [ a''(\bar y) w^2 \nabla \bar y + a'(\bar y)w \nabla w ]  \cdot \nabla  \varphi \dx.
        \end{aligned}
    \end{equation*}
    For $Z_{y_n,\bar y}^{(2)}$, we see from \eqref{eq:distance} and the fact $\nabla \bar y = 0$ a.e. on $\{\bar y = \bar t \}$ that
    \begin{equation*}
        Z_{y_n,\bar y}^{(2)} = \1_{\{ \bar y = \bar t, y_n \in (\bar t-\delta,\bar t) \}} a_{0}'(y_n) \nabla (y_n - \bar y) + \1_{\{ \bar y = \bar t, y_n \in (\bar t, \bar t + \delta) \}} a_{1}'(y_n) \nabla (y_n - \bar y).
    \end{equation*}
    Setting $\hat w_n := \frac{y_n - \bar y}{s_n}$ and exploiting \eqref{eq:S-deri-keylem} yields
    $w_n - \hat w_n \to w - w = 0$ in $W^{1,\tilde{p}}_0(\Omega)\hookrightarrow C(\overline\Omega)$. From this, the limit $y_n \to \bar y$ in $C(\overline\Omega)$, and the continuity of $a_{0}'$ and $a_1'$, the dominated convergence theorem implies that
    \begin{multline*}
        \lim \limits_{n \to \infty} \frac{1}{s_n} \int_\Omega Z_{y_n,\bar y}^{(2)} \cdot \nabla \varphi w_n \dx
        =   \lim \limits_{n \to \infty}  \int_\Omega [  \1_{\{ \bar y = \bar t, y_n \in (\bar t-\delta, \bar t) \}} a_{0}'(\bar t) \hat w_n    + \1_{\{ \bar y = \bar t, y_n \in (\bar t, \bar t + \delta) \}} a_{1}'(\bar t) \hat w_n  ] \nabla \hat w_n   \cdot \nabla \varphi  \dx.
    \end{multline*}
    As a result of \eqref{eq:directional-der-a1}  and the fact that  $\hat w_n < 0$ on $\{ \bar y = \bar t, y_n \in (\bar t-\delta, \bar t) \}$, there holds
    \begin{equation*}
        \1_{\{ \bar y =\bar t, y_n \in (\bar t-\delta, \bar t) \}} a_{0}'(\bar t) \hat w_n =\1_{\{ \bar y = \bar t, y_n \in (\bar t-\delta, \bar t) \}}  a' (\bar t; \hat w_n ).
    \end{equation*}
    Similarly, one has
    \begin{equation*}
        \1_{\{ \bar y =\bar t, y_n \in (\bar t, \bar t + \delta) \}} a_{1}'(\bar t) \hat w_n  = \1_{\{ \bar y = \bar t, y_n \in (\bar t, \bar t + \delta) \}} a' (\bar t; \hat w_n).
    \end{equation*}
    We thus have
    \begin{equation*}
        \begin{aligned}
            \lim \limits_{n \to \infty} \frac{1}{s_n} \int_\Omega Z_{y_n,\bar y}^{(2)} \cdot \nabla \varphi w_n \dx & =   \lim \limits_{n \to \infty} \int_\Omega  \1_{\{ \bar y = \bar t, y_n \in (\bar t-\delta, \bar t) \cup (\bar t, \bar t + \delta) \}}  a'(\bar t; \hat w_n) \nabla \hat w_n  \cdot \nabla \varphi  \dx\\
            & = \lim \limits_{n \to \infty}  \int_\Omega  \1_{\{ \bar y = \bar t \}}  a' (\bar t; \hat w_n) \nabla \hat w_n  \cdot \nabla \varphi  \dx,
        \end{aligned}
    \end{equation*}
    where we have used \eqref{eq:distance} and the fact that $\nabla y_n = \nabla \bar y = 0$ and so $\nabla \hat w_n =0$ a.e. on $\{ y_n = \bar y = \bar t \}$ to obtain the last identity. We thus conclude from the continuity of $a'(\bar t; \cdot)$ due to Proposition 2.49 in \cite{BonnansShapiro2000}, \eqref{eq:S-deri-keylem}, and the dominated convergence theorem that
    \begin{equation} \label{eq:Z2-limit}
        \lim \limits_{n \to \infty} \frac{1}{s_n} \int_\Omega Z_{y_n,\bar y}^{(2)} \cdot \nabla \varphi w_n \dx  = \int_\Omega \1_{\{\bar y =\bar t \}} a'(\bar y; w) \nabla w \cdot \nabla  \varphi \dx.
    \end{equation}
    For $Z_{y_n,\bar y}^{(4)}$, we have from \eqref{eq:Z-esti} that
    \begin{equation*}
        | Z_{y_n,\bar y}^{(4)} | \leq  C_M \left[ |y_n - \bar y|| \nabla \bar y| + |\nabla(y_n - \bar y)| \right]  \left(  \1_{\Omega_{y_n,\bar y}^{2}} + \1_{\Omega_{y_n,\bar y}^{3}} \right) \quad \text{a.e. in } \Omega.
    \end{equation*}
    This, together with the fact that $\1_{\Omega_{y_n,\bar y}^{2}} + \1_{\Omega_{y_n,\bar y}^{3}} \to 0$ a.e. in $\Omega$ as well as \eqref{eq:S-deri-keylem}, yields
    \begin{equation} \label{eq:Z4-limit}
        \lim \limits_{n \to \infty} \frac{1}{s_n} \int_\Omega Z_{y_n,\bar y}^{(4)} \cdot \nabla \varphi w_n \dx  = 0.
    \end{equation}
    It remains to estimate $Z_{y_n,\bar y}^{(3)}$. To this end, we first deduce from
    \eqref{eq:Omega-23-esti} and the coarea formula for Lipschitz mappings (see, e.g. \cite{Evans1992}, Thm.~2, p.~117) or \cite{AlbertiBianchiniCrippa2013}, Sec.~2.7) that
    \begin{equation*}
        \begin{aligned}
            s_n^{-1}\left| \int_\Omega \1_{\Omega^{2}_{y_n, \bar y}} \nabla \bar y \cdot \nabla  \varphi (w_n - \hat w_n)dx \right| & \leq s_n^{-1}\norm{w_n - \hat w_n}_{L^\infty(\Omega)} \norm{\nabla \varphi }_{L^\infty(\Omega)}  \int_\Omega \1_{\{ 0 <  \bar y - \bar t \leq \tau_n  \}} | \nabla \bar y| dx \\
            & = s_n^{-1}\norm{w_n - \hat w_n}_{L^\infty(\Omega)} \norm{\nabla \varphi }_{L^\infty(\Omega)} \int_{\bar t}^{\bar t+\tau_n} \int_{\{\bar y =t_i \}}d\Ha^1(x)dt  \\
            & \leq C s_n^{-1}\tau_n \norm{w_n - \hat w_n}_{L^\infty(\Omega)} \norm{\nabla \varphi }_{L^\infty(\Omega)} \to 0,
        \end{aligned}
    \end{equation*}
    where we have used \eqref{eq:distance} and the fact that $w_n - \hat w_n \to 0$ in $C(\overline\Omega)$ to obtain the last limit. Similarly, $s_n^{-1}| \int_\Omega \1_{\Omega^{3}_{y_n, \bar y}} \nabla \bar y \cdot \nabla  \varphi (w_n - \hat w_n)dx | \to 0$.  From these limits and the definition of $Z_{y_n,\bar y}^{(3)}$, we deduce that
    \[
        s_n^{-1} \int_\Omega | Z_{y_n,\bar y}^{(3)} \cdot \nabla \varphi ( w_n - \hat w_n) | \dx  \to 0
    \]
    and thus
    \begin{equation} \label{eq:Z3-limit-auxi1}
        \lim \limits_{n \to \infty}  \frac{1}{s_n} \int_\Omega Z_{y_n,\bar y}^{(3)} \cdot \nabla \varphi w_n   \dx = \lim \limits_{n \to \infty}  \frac{1}{s_n} \int_\Omega Z_{y_n,\bar y}^{(3)} \cdot \nabla \varphi  \hat w_n  \dx
        = \lim \limits_{n \to \infty}  \frac{1}{s_n^2} \int_\Omega Z_{y_n,\bar y}^{(3)} \cdot \nabla \varphi  (y_n - \bar y)  \dx
    \end{equation}
    provided that one of these three limits exists.
    For $\zeta_i(\bar u,\bar y;s_n,v_n)$, $i=0,1$, defined in \eqref{eq:zeta-func}, one has
    \begin{equation*}
        \begin{aligned}
            P_n &:= Z_{y_n,\bar y}^{(3)} (y_n - \bar y) + 2 \sum_{ i=0}^1 \zeta_i(\bar u,\bar y;s_n,v_n) \nabla \bar y\\
            & = [a_0'(\bar t) -a_1'(\bar t)][ \1_{\Omega_{y_n,\bar  y}^{2}}    -  \1_{\Omega_{y_n,\bar y}^{3}} ]  \nabla \bar y (y_n - \bar y) +2 [a_0'(\bar t) -a_1'(\bar t)](\bar t - y_n)[ \1_{\Omega_{y_n,\bar  y}^{2}}    -  \1_{\Omega_{y_n,\bar y}^{3}} ]\nabla \bar y \\
            & = [a_0'(\bar t) -a_1'(\bar t)]  (2 \bar t - \bar y - y_n) [\1_{\Omega_{y_n,\bar y}^{2}} -\1_{\Omega_{y_n,\bar y}^{3}}  ] \nabla \bar y.
        \end{aligned}
    \end{equation*}
    By \eqref{eq:Pn-auxi} and \eqref{eq:C1-convergence}, we deduce that
    \[
        \lim \limits_{n \to \infty}  \frac{1}{s_n^2} \int_\Omega P_n \cdot \nabla \varphi    \dx =0.
    \]
    Combining this with \eqref{eq:Z3-limit-auxi1} and \eqref{eq:sigma-tilde}, we can conclude that
    \begin{equation*}
        \begin{aligned}
            \limsup \limits_{n \to \infty} \frac{1}{s_n} \int_\Omega Z_{y_n,\bar y}^{(3)} \cdot \nabla \varphi w_n \dx
            &= \limsup \limits_{n \to \infty} \left[ \frac{1}{s_n^2} \int_\Omega P_n \cdot \nabla \varphi  \dx - \frac{2}{s_n^2} \int_\Omega \sum_{ i = 0}^1 \zeta_i(\bar u,\bar y;s_n,v_n) \nabla \bar y  \cdot \nabla \varphi  \dx \right]\\
            & = - 2 \tilde{Q}(\bar u, \bar y, \varphi;\{s_n\},v).
        \end{aligned}
    \end{equation*}
    Together with \eqref{eq:Z-decom-error}--\eqref{eq:Z4-limit}, we obtain assertion \ref{item:Z-func-limit} from \eqref{eq:explicit-formular}.

    \emph{Ad \ref{item:key-limit-es}:} Defining the functional $G: L^\infty(\Omega) \to \R$ via
    $G(y) := \int_\Omega L(x, y(x)) \dx$
    and employing \cref{ass:cost_func}, we deduce that $G$ is of class $C^2$ and that its derivatives are given by
    \begin{equation*}
        G'(y)y_1 := \int_\Omega \frac{\partial L}{\partial y}(x,y(x))y_1(x) \dx \quad \text{and} \quad  G''(y)y_1y_2 := \int_\Omega \frac{\partial^2 L}{\partial y^2}(x,y(x))y_1(x)y_2(x) \dx
    \end{equation*}
    for all $y, y_1, y_2 \in L^\infty(\Omega)$. We see from the chain rule that for any $u, v \in L^2(\Omega)$,
    \begin{equation*}
        j'(u)v = G'(S(u))S'(u)v + \nu \int_\Omega uv \dx.
    \end{equation*}
    This, together with a Taylor expansion and the fact that $\norm{v_n}_{L^2(\Omega)} = 1$, yields
    \begin{multline}
        \label{eq:key-auxi1}
        \frac{1}{s_n} [j'(u_n)v_n - j'(\bar u)v_n ]  = \frac{1}{s_n} [ G'(y_n)S'(u_n)v_n - G'(\bar y)S'(\bar u)v_n ] + \nu  \int_\Omega v_n^2 \dx\\
        \begin{aligned}[t]
            & = \frac{1}{s_n} [ G'(y_n) - G'(\bar y) ]S'(\bar u)v_n + \frac{1}{s_n} G'(y_n)[ S'(u_n)v_n - S'(\bar u)v_n ] + \nu \\
            & = \frac{1}{s_n} \int_0^1 G''(\bar y + s(y_n - \bar y))(y_n - \bar y)S'(\bar u)v_n \, ds  + \frac{1}{s_n} G'(\bar y)[ S'(u_n)v_n - S'(\bar u)v_n ]   \\
            & \qquad + \frac{1}{s_n} [G'(y_n) - G'(\bar y) ] [ S'(u_n)v_n - S'(\bar u)v_n ] + \nu.
        \end{aligned}
    \end{multline}
    Obviously, the third term on the right-hand side of \eqref{eq:key-auxi1} tends to $0$ since $v_n \to v$ in $W^{-1,\tilde{p}}(\Omega)$, $S'$ is continuous, and $G$ is of class $C^2$. Moreover, it follows from \cref{ass:cost_func}, \eqref{eq:S-deri-keylem}, and the dominated convergence theorem that the first term on the right-hand side of \eqref{eq:key-auxi1} tends to $G''(\bar y)(S'(\bar u)v)^2$.
    It remains to estimate the limes inferior of the second term on the right-hand side of \eqref{eq:key-auxi1}.
    Subtracting the equations for $z_n^{(1)}:= S'(u_n)v_n$ and $z_n^{(2)}:= S'(\bar u)v_n$, we find that $z_n := z_n^{(1)} - z_n^{(2)}\in H^1_0(\Omega)$ satisfies
    \begin{equation} \label{eq:subtraction}
        -\dive [(b + a(\bar y)) \nabla z_n  + \1_{\{\bar y \neq \bar t \}} a'(\bar y) \nabla \bar  y z_n ] =
        \dive [  (a(y_n) - a(\bar y)) \nabla z_{n}^{(1)} +  Z_{y_n,\bar y} z_{n}^{(1)}] =:
        g_{n}.
    \end{equation}
    We then have $z_n = S'(\bar u)g_n$, which together with \eqref{eq:adjoint_OS} yields
    \begin{equation} \label{eq:second-term}
        G'(\bar y)[z_n^{(1)} - z_n^{(2)} ]  = \langle G'(\bar y), z_n \rangle_{H^{-1}(\Omega), H^1_0(\Omega)} = \langle  g_n, \varphi \rangle_{H^{-1}(\Omega), H^1_0(\Omega)} = - (B_n + C_n)
    \end{equation}
    for $B_n := \int_\Omega (a(y_n) - a(\bar y)) \nabla z_{n}^{(1)} \cdot \nabla  \varphi \dx$ and $C_n :=  \int_\Omega Z_{y_n, \bar y}  \cdot \nabla  \varphi z_{n}^{(1)} \dx$.
    As a result of \cref{thm:control2state-oper} and the fact that $v_n \to v$ in $W^{-1,\tilde{p}}(\Omega)$, there holds $z_n^{(1)} \to S'(\bar u)v$ in $W^{1,\tilde{p}}_0(\Omega)$. Besides, from \eqref{eq:S-deri-keylem} and Lemma 3.5 in \cite{ClasonNhuRosch}, we have
    \begin{equation*}
        \frac{a(y_n(x)) - a(\bar y(x))}{s_n} \to a'(\bar y(x); (S'(\bar u)v)(x)) \quad \text{for all } x \in \overline\Omega.
    \end{equation*}
    The dominated convergence theorem thus implies that
    \begin{equation*}
        \frac{1}{s_n} B_n \to \int_\Omega a'(\bar y; S'(\bar u)v) \nabla (S'(\bar u)v) \cdot \nabla  \varphi \dx.
    \end{equation*}
    This, along with \eqref{eq:second-term} and assertion \ref{item:Z-func-limit}, ensures that
    \begin{multline*}
        \liminf \limits_{n \to \infty} \frac{1}{s_n} G'(\bar y)[ S'(u_n)v_n - S'(\bar u)v_n ] =     [a'_{0}(\bar t) - a'_{1}(\bar t)]\int_{\{\bar y=\bar t\}}  (S'(\bar u)v)^2   \frac{\nabla \bar y \cdot \nabla  \varphi}{|\nabla \bar y|} \dH^{1}(x)\\
        \begin{aligned}
            &- 2 \int_\Omega a'(\bar y; S'(\bar u)v)\nabla (S'(\bar u )v) \cdot \nabla \varphi \dx
            - \int_\Omega  \1_{\{ \bar y \neq \bar t \}} a''(\bar y)(S'(\bar u)v)^2 \nabla \bar y\cdot \nabla \varphi \dx.
        \end{aligned}
    \end{multline*}
    Using these limits, \eqref{eq:key-auxi1}, and the identity for $Q$ in \eqref{eq:2nd-OC-suff-explicit} (see also Theorem 3.19 in \cite{ClasonNhuRosch_levelset}), we arrive at \ref{item:key-limit-es}.
\end{proof}

The following theorem is one of main results of the paper, which extends Theorem 2.14 in \cite{CasasTroltzsch2012} (see, also, Lemma 5.2 in \cite{CasasMateosRosch2021}) to the case where the cost functional $j$ is of class $C^1$ but not necessarily $C^2$.
\begin{theorem} \label{thm:error-esti-general}
    Let $\{(\bar y_h,\bar u_h)\}$ be the sequence of discrete solutions to \eqref{eq:discrete-P-nonreduced} converging strongly to $(\bar y, \bar u)$ in $H^1_0(\Omega) \times L^2(\Omega)$.
    Assume that
    $\{ \bar y = \bar t\}$ decomposes into finitely many connected components and that on each such connected component $\Cu$, either \eqref{eq:main-hypothesis-nonvanishing-gradient} or \eqref{eq:main-hypothesis-y-vanish-structure} is fulfilled.
    Assume further that that the second-order sufficient condition \eqref{eq:2nd-OC-suff-explicit} is fulfilled.
    Then there exist constants $C >0$ and $\hat h \in (0, \min\{ \bar h, h_* \})$ such that
    \begin{equation}
        \label{eq:key-error-esti}
        \norm{\bar u_h - \bar u}_{L^2(\Omega)}^2 \leq C [ (\epsilon_h^{\bar u_h})^2+ \norm{\bar u- u_h}_{L^2(\Omega)}^2 + j'(\bar u)(u_h - \bar u) ]
    \end{equation}
    for all $h \in (0, \hat h)$ and $u_h \in \mathcal{U}_{ad,h} \cap \overline B_{L^2(\Omega)}(\bar u, \hat \epsilon)$  with $ \hat \epsilon := \min\{ \bar \epsilon, \bar \rho \}$. Here $\epsilon_h^{u}$ is defined as in \eqref{eq:epsilon-h}.
\end{theorem}
\begin{proof}
    We first observe from \cref{prop:E-func-finite} that $\Sigma(\bar y) < \infty$.
    For simplicity of notation, we set $\epsilon_{h} := \epsilon_h^{\bar u_h}$.
    We first show that
    \begin{multline}
        \label{eq:key-error-auxi1}
        [ j'(\bar u_h)  - j'(\bar u)](\bar u_h - \bar u) \leq j'(\bar u)(u_h - \bar u)  \\
        +  C[ \epsilon_h \norm{\bar u_h - \bar u}_{L^2(\Omega)} + \epsilon_h \norm{u_h - \bar u}_{L^2(\Omega)} + \norm{u_h - \bar u}_{L^2(\Omega)} \norm{\bar u_h - \bar u}_{L^2(\Omega)}]
    \end{multline}
    for some constant $C>0$,
    for all $u_h \in \mathcal{U}_{ad,h} \cap \overline B_{L^2(\Omega)}(\bar u, {\hat \epsilon})$ and $h \in (0, \min\{ \bar h, h_{*} \})$. To this end, let us take any $h \in (0, \min\{ \bar h, h_{*} \})$, $u \in \mathcal{U}_{ad} \cap \overline B_{L^2(\Omega)}(\bar u, {\hat \epsilon})$, and $u_h \in \mathcal{U}_{ad,h} \cap \overline B_{L^2(\Omega)}(\bar u, {\hat \epsilon})$.
    We deduce from \eqref{eq:obje-deri}, \cref{thm:diff-cost-discrete}, \cref{lem:continuity-discrete}, and the Cauchy--Schwarz inequality that
    \begin{equation} \label{eq:key-error-auxi2}
        \left| [ j'_h(u)  - j'(u)](u_h - \bar u) \right|  =  \left| \int_\Omega (\varphi_h(u) - \varphi_u)(u_h - \bar u) \dx \right| \leq C \epsilon_h^{u} \norm{u_h - \bar u}_{L^2(\Omega)}.
    \end{equation}
    Moreover, we deduce from $j_h'(\bar u_h)(\bar u_h - u_h) \leq 0$ and $j'(\bar u)(\bar u - \bar u_h)  \leq 0$ that
    \begin{multline} \label{eq:key-error-auxi3}
        [ j'(\bar u_h)  - j'(\bar u)](\bar u_h - \bar u) \\
        \begin{aligned}[t]
            & = [j_h'(\bar u_h)  - j'(\bar u_h) ] ( \bar u - \bar u_h) + [j_h'(\bar u_h)  - j'(\bar u) ] ( u_h -\bar u)   + j_h'(\bar u_h)(\bar u_h - u_h) + j'(\bar u)(u_h - \bar u_h) \\
            &  \leq [j_h'(\bar u_h)  - j'(\bar u_h) ] ( \bar u - \bar u_h) + [j_h'(\bar u_h)  - j'(\bar u) ] ( u_h -\bar u)   + j'(\bar u)[(u_h - \bar u) + (\bar u - \bar u_h)] \\
            & \leq [j_h'(\bar u_h)  - j'(\bar u_h) ] ( \bar u - \bar u_h) + [(j_h'(\bar u_h) - j'(\bar u_h)) +( j'(\bar u_h)  - j'(\bar u)) ] ( u_h -\bar u)   + j'(\bar u)(u_h - \bar u).
        \end{aligned}
    \end{multline}
    Applying \eqref{eq:key-error-auxi2} yields that
    \begin{equation} \label{eq:key-error-auxi4}
        \left\{
            \begin{aligned}
                & | [ j'_h(\bar u_h)  - j'(\bar u_h)](\bar u - \bar u_h) |  \leq C \epsilon_h \norm{\bar u - \bar u_h}_{L^2(\Omega)},\\
                & | [ j'_h(\bar u_h)  - j'(\bar u_h)](u_h - \bar u) |  \leq C \epsilon_h  \norm{\bar u - u_h}_{L^2(\Omega)}.
            \end{aligned}
        \right.
    \end{equation}
    Using \eqref{eq:obje-deri}, \eqref{eq:continuityH2-adjoint-discrete-noh}, $H^2(\Omega) \hookrightarrow L^2(\Omega)$, and the Cauchy--Schwarz inequality yields
    \begin{equation} \label{eq:key-error-esti2}
        |[ j'(\bar u_h)  - j'(\bar u) ] ( u_h -\bar u)  |  =\left| \int_\Omega [ (\varphi_{\bar u_h} - \bar \varphi) + \nu(\bar u_h - \bar u) ](u_h - \bar u)\dx \right|
        \leq C  \norm{\bar u_h - \bar u}_{L^2(\Omega)}\norm{u_h - \bar u}_{L^2(\Omega)}.
    \end{equation}
    From this, \eqref{eq:key-error-auxi3}, and \eqref{eq:key-error-auxi4}, we derive \eqref{eq:key-error-auxi1}.

    We now prove the conclusion of the theorem by contradiction. To that purpose, we suppose that there exist $h_n \to 0^+$ and $u_{h_n} \in \mathcal{U}_{ad, h_n} \cap \overline B_{L^2(\Omega)}(\bar u, {\hat \epsilon})$ such that
    \begin{equation*}
        \norm{\bar u_{h_n} - \bar u}_{L^2(\Omega)}^2 > n  \left[ ( \epsilon_{h_n} )^2 + \norm{\bar u- u_{h_n}}_{L^2(\Omega)}^2 + j'(\bar u)(u_{h_n} - \bar u) \right] \quad \text{for all } n \geq 1
    \end{equation*}
    or, equivalently, with $s_n := \norm{\bar u_{h_n} - \bar u}_{L^2(\Omega)}$ that
    \begin{equation}
        \label{eq:key-error-contradiction}
        \frac{1}{n} > \frac{(\epsilon_{h_n})^2}{s_n^2} + \frac{\norm{\bar u- u_{h_n}}_{L^2(\Omega)}^2}{s_n^2}  + \frac{j'(\bar u)(u_{h_n} - \bar u)}{s_n^2} \quad  \text{for all } n \geq 1.
    \end{equation}
    By setting $v_n := \frac{\bar u_{h_n} - \bar u}{s_n}$,
    and by extracting a subsequence if necessary,
    we have
    \begin{equation*}
        s_n \to 0^+, \quad \norm{v_n}_{L^2(\Omega)} = 1, \quad \bar u_{h_n} = \bar u + s_n v_{n}, \quad v_n \rightharpoonup v \ \text{in } L^2(\Omega) \quad \text{for some } v \in L^2(\Omega).
    \end{equation*}
    We first show that $v$ is an element of the critical cone $\mathcal{C}({\mathcal{U}_{ad};\bar u})$ defined in \eqref{eq:critical-cone}.
    To this end, we first deduce that $v \geq 0$ a.e. on $\{\bar u =\alpha \}$ and $v \leq 0$ a.e. on $\{\bar u = \beta \}$. Moreover, since $j'(\bar u)v_n \geq 0$, there holds $j'(\bar u)v \geq 0$. On the other hand, from \eqref{eq:key-error-auxi4} and \eqref{eq:key-error-esti2} for $u_h := \bar u_{h_n}$, we obtain that
    \begin{equation*}
        \limsup \limits_{n \to \infty} \left\{  [ j'(\bar u_{h_n}) - j'_{h_n}(\bar u_{h_n}) ]v_{n} + [ j'(\bar u) - j'(\bar u_{h_n})]v_{n} \right\}  \leq \lim\limits_{n \to \infty} C (\epsilon_{h_n} + s_n) = 0,
    \end{equation*}
    which yields
    \begin{equation*}
        \begin{aligned}
            j'(\bar u)v & = \lim\limits_{n \to \infty} j'(\bar u)v_n  = \lim\limits_{n \to \infty} [ j'_{h_n}(\bar u_{h_n}) v_{n} + ( j'(\bar u_{h_n}) - j'_{h_n}(\bar u_{h_n} ))v_n + (j'(\bar u) - j'(\bar u_{h_n} ))v_n ] \\
            & \leq \lim\limits_{n \to \infty} j'_{h_n}(\bar u_{h_n}) v_{n} = \lim\limits_{n \to \infty} \frac{1}{s_n} [ j'_{h_n}(\bar u_{h_n}) (u_{h_n} - \bar u) + j'_{h_n}(\bar u_{h_n}) (\bar u_{h_n} - u_{h_n})].
        \end{aligned}
    \end{equation*}
    From this and the fact that $j'_{h_n}(\bar u_{h_n}) (\bar u_{h_n} - u_{h_n}) \leq 0$, we obtain
    \begin{equation*}
        j'(\bar u)v   \leq \lim\limits_{n \to \infty} \frac{1}{s_n} j'_{h_n}(\bar u_{h_n}) (u_{h_n} - \bar u) \leq \lim\limits_{n \to \infty} \norm{\varphi_{h_n}(\bar u_{h_n}) + \nu \bar u_{h_n}}_{L^2(\Omega)} \frac{\norm{u_{h_n} - \bar u}_{L^2(\Omega)}}{s_n} \to  0,
    \end{equation*}
    where we have used \cref{thm:diff-cost-discrete} and the Cauchy--Schwarz inequality to derive the last estimate and the boundedness of $\{\norm{\varphi_{h_n}(\bar u_{h_n}) + \nu \bar u_{h_n}}_{L^2(\Omega)} \}$ (due to \cref{lem:continuity-discrete}) as well as \eqref{eq:key-error-contradiction} to pass to the limit. There therefore holds that $j'(\bar u)v = 0$. This and Lemma 4.11 in \cite{BayenBonnansSilva2014} lead to $v(x) = 0$ whenever $\bar\varphi(x) + \nu \bar u(x) \neq 0$. We thus have $v \in \mathcal{C}({\mathcal{U}_{ad};\bar u})$.

    We now derive a contradiction and thus complete the proof. To this end, we divide \eqref{eq:key-error-auxi1} (with $h:=h_n$) by $s_n^2$ to obtain
    \begin{equation*}
        \frac{1}{s_n}[ j'(\bar u_{h_n})  - j'(\bar u)]v_n \leq \frac{ j'(\bar u)(u_{h_n} - \bar u)}{s_n^2} +  C\left( \frac{\epsilon_{h_n}}{s_n} +  \frac{\epsilon_{h_n}}{s_n} \frac{ \norm{u_{h_n} - \bar u}_{L^2(\Omega)} }{s_n}+ \frac{ \norm{u_{h_n} - \bar u}_{L^2(\Omega)} }{s_n} \right).
    \end{equation*}
    Taking the limes inferior as $n \to \infty$, employing \eqref{eq:key-error-contradiction}, and using \cref{lem:limit-key}\,(ii), we conclude that
    \begin{equation} \label{eq:contradition}
        2Q(\bar u, \bar y, \bar \varphi; v) + \nu \left(1 - \norm{v}_{L^2(\Omega)}^2\right) \leq 0.
    \end{equation}
    Combining this with \eqref{eq:2nd-OC-suff-explicit}  and the fact that $\norm{v}_{L^2(\Omega)} \leq \liminf \limits_{n \to \infty} \norm{v_n}_{L^2(\Omega)} =1$, we have
    $v = 0$. Inserting this into \eqref{eq:contradition} leads to $0<\nu \leq 0$, which is the desired contradiction.
\end{proof}

\begin{theorem}[variational discretization]
    \label{thm:error-estimate-L2}
    Assume that $\mathcal{U}_h = L^\infty(\Omega)$.
    Let $\{(\bar y_h,\bar u_h)\}$ be the sequence of discrete solutions to \eqref{eq:discrete-P-nonreduced} converging strongly to $(\bar y, \bar u)$ in $H^1_0(\Omega) \times L^2(\Omega)$.
    Under all assumptions of \cref{thm:error-esti-general}, there exists a constant $C >0$ such that for any  $h \in (0, \hat h)$,
    \begin{equation}
        \label{eq:error-estimate-L2}
        \norm{\bar u_{h} - \bar u}_{L^2(\Omega)} \leq C \left( h^{1 + \gamma} + \sigma_0^\frac{1}{2}  \norm{\Sigma_{\kappa_h}(\bar y)}_{L^1(\Omega)}^\frac{1}{2}  h^{\frac{3}{2}}  \right)
        \leq C h^{1+ \gamma}
    \end{equation}
    with $\kappa_h := C_{\infty} h + \norm{S(\bar u_h) - \bar y}_{L^\infty(\Omega)}$,  and $\sigma_0$,
    $\gamma \in (0, \frac{1}{2})$,
    and $\Sigma_{\kappa_h}$ defined in \eqref{eq:sigma-i},
    \eqref{eq:gamma-exponent},
    and \eqref{eq:KE-func}, respectively.
\end{theorem}
\begin{proof}
    Choosing $u_h := \bar u$ in \eqref{eq:key-error-esti} yields
    \begin{equation}
        \label{eq:L2-error-VD}
        \norm{\bar u_{h} - \bar u}_{L^2(\Omega)} \leq C \epsilon_h^{\bar u_h} \quad \text{for all }h \in (0, \hat h).
    \end{equation}
    Setting $r_h := \norm{S_h(\bar u_h) - S(\bar u_h)}_{L^\infty(\Omega)}$ and using \eqref{eq:state-error-infty} yield $r_h \leq C_\infty h$.
    Exploiting \eqref{eq:epsilon-h}, the Cauchy--Schwarz inequality, \cref{prop:K-E-esti}, the estimate \eqref{eq:nonsmooth-bary}, and the monotonic growth of $V(\bar y, \cdot)$, there holds
    \begin{equation*}
        \begin{aligned}
            (\epsilon_h^{\bar u_h} )^2 &\leq 3h^{2 + 2\gamma} + 3h^2 \norm{V(S(\bar u_h), r_h)}_{L^2(\Omega)}^2 + 3h^2 \norm{Z_{S(\bar u_h), \bar y}}_{L^2(\Omega)}^2  \\
            & \leq 3h^{2 + 2\gamma} + Ch^2  \left[ \norm{V(\bar y, r_h + \norm{S(\bar u_h) - \bar y}_{L^\infty(\Omega)}) }_{L^2(\Omega)}^2 + \norm{S(\bar u_h) - \bar y}_{H^1_0(\Omega)}^2 + \norm{\bar u_h - \bar u}_{L^2(\Omega)}^2\right] \\
            & \leq 3h^{2 + 2\gamma} + Ch^2  \left[ \norm{V(\bar y, C_\infty h+ \norm{S(\bar u_h) - \bar y}_{L^\infty(\Omega)}) }_{L^2(\Omega)}^2 + \norm{S(\bar u_h) - \bar y}_{H^1_0(\Omega)}^2  + \norm{\bar u_h - \bar u}_{L^2(\Omega)}^2 \right] \\
            & \leq C\left[h^{2 + 2\gamma} + h^2 \sigma_0 \norm{\Sigma_{\kappa_h}(\bar y)}_{L^1(\Omega)} (C_\infty h + \norm{S(\bar u_h) - \bar y}_{L^\infty(\Omega)}) + h^2 \norm{S(\bar u_h) - \bar y}_{H^1_0(\Omega)}^2 + h^2 \norm{\bar u_h - \bar u}_{L^2(\Omega)}^2 \right]\\
            & \leq C\left[h^{2 + 2\gamma} + h^2 \sigma_0 \norm{\Sigma_{\kappa_h}(\bar y)}_{L^1(\Omega)} (h +\norm{\bar u_{h} - \bar u}_{L^2(\Omega)}) + h^2 \norm{\bar u_{h} - \bar u}_{L^2(\Omega)}^2 \right],
        \end{aligned}
    \end{equation*}
    where we have used \cref{thm:control2state-oper} to derive the last inequality. From this and \eqref{eq:L2-error-VD}, a simple computation gives \eqref{eq:error-estimate-L2}.
\end{proof}

\begin{remark}
    \label{rem:order-convergence}
    If the constant $p_*$ in \cref{thm:control2state-oper} is large enough, then, for arbitrary small $\epsilon >0$, we can take $\gamma := \frac{1}{2} - \epsilon$ in \eqref{eq:gamma-exponent}  by choosing $p_0$ close to $2$ enough.
    Therefore, the order of convergence in \cref{thm:error-estimate-L2} becomes $O(h^{\frac{3}{2} - \epsilon})$. This order is less than the one for the smooth situation investigated in \cite{CasasTroltzsch2012}, there the authors showed that the order of convergence associated with the variational discretization is $O(h^2)$. This fact can be attributed to the nondifferentiability of the function $a$ in the state equation as we will see later in the numerical example section.
\end{remark}

Similarly, we obtain from \cref{thm:error-esti-general} error estimates for piecewise constant and continuous piecewise linear controls.
\begin{theorem}[piecewise constant discretization and continuous piecewise linear discretization]
    \label{thm:error-estimate-L2-piecewise-constant}
    Assume that
    $\mathcal{U}_h=\mathcal{U}_h^i$, $i=0,1$.
    Let $\{(\bar y_h,\bar u_h)\}$ be the sequence of discrete solutions to \eqref{eq:discrete-P-nonreduced} converging strongly to $(\bar y, \bar u)$ in $H^1_0(\Omega) \times L^2(\Omega)$.
    Under all assumptions of \cref{thm:error-esti-general}, there exist constants $C >0$ and $\hat h_* \in (0,\hat h)$ such that
    \begin{equation}
        \label{eq:error-estimate-L2-piecewise-constant}
        \norm{\bar u_{h} - \bar u}_{L^2(\Omega)} \leq C h \quad \text{for all } h \in (0, \hat h_*).
    \end{equation}
\end{theorem}
\begin{proof}
    According to \cref{thm:1st-OC}, $\bar\varphi$ and $\bar u$ are Lipschitz continuous on $\overline\Omega$. Hence constants $C,C_1 >0$ and $\hat h_* \in (0,\hat h)$ exist such that for any $h \in (0, \hat h_*)$, there exists a $u_h \in \mathcal{U}_{ad, h}$ satisfying $\norm{\bar u - u_h}_{L^\infty(\Omega)} \leq C_1 h$ and
    $j'(\bar u)(\bar u - u_h) = 0$ for the case $\mathcal{U}_h=\mathcal{U}_h^0$; see, e.g. Lemma 4.17 in \cite{CasasMateosRaymond2007}, as well as
    $|j'(\bar u)(\bar u - u_h)| \leq Ch^2$ for the case $\mathcal{U}_h=\mathcal{U}_h^1$;  see the proof of Theorem 5.4 in \cite{CasasMateosRosch2021}.
    Combining this with \eqref{eq:key-error-esti} and the fact that $\epsilon_{h}^{\bar u_h} \leq C_2h$ for all $h \in (0, \hat h)$ yields \eqref{eq:error-estimate-L2-piecewise-constant}.
\end{proof}

\subsection{Numerical example} \label{sec:numerical-confirmation}
We conclude this section with a preliminary numerical example for the variational discretization of the optimal control problem. Specifically, we consider the problem
\begin{equation}
    \label{eq:P-exam}
    \left\{
        \begin{aligned}
            \min_{u\in L^\infty(\Omega)} &j(u) := \frac{1}{2} \int_\Omega (y_u(x) - y_d(x))^2 \dx + \frac{\nu}{2} \int_\Omega u(x)^2 \dx \\
            \text{s.t.} \quad & -\dive [(b(x) + \max\{ y_u-m,0\}) \nabla y_u] = u \quad \text{in } \Omega, \quad y_u = 0 \, \text{on } \partial\Omega, \\
            & 0 \leq  u(x) \leq 4\pi^2  \quad \text{a.e. } x \in \Omega,
        \end{aligned}
    \right.
\end{equation}
where $\Omega := (0,1)^2 \subset \R^2$, $m \in (0,1]$, $\nu := 1$,
$b(x) := b(x_1,x_2) = 2 - \max\{\sin \pi x_1 \sin \pi x_2-m,0\}$, and
\[
    y_d(x) := (1 + 16 \nu \pi^4) \sin \pi x_1 \sin \pi x_2 + 4 \nu \pi^4 \1_{\{\sin \pi x_1 \sin \pi x_2 >m \}} (\cos^2 \pi x_1 \sin^2\pi x_2 + \sin^2 \pi x_1 \cos^2 \pi x_2 ).
\]
This problem fits the general setting with  $\alpha = 0$, $\beta = 4\pi^2$,  $a(t) = \max\{t-m, 0\}$; i.e., $a_0(t) :=0$, $a_1(t) := t-m$, and $\bar t := m$.
Setting $\bar u := 4\pi^2 \bar y$ with $\bar y:= \sin \pi x_1 \sin \pi x_2$ and $\bar \varphi := - \nu \bar u$, it is straightforward to verify that $(\bar u, \bar y, \bar\varphi)$ satisfies the first-order optimality condition \eqref{eq:1st-OS} associated with \eqref{eq:P-exam}.
We shall now show that there exists an $\epsilon >0$ such that if $0 \leq 1-m < \epsilon$ then all assumptions of \cref{thm:2nd-OS-suf}, and thus of \cref{thm:error-estimate-L2}, are fulfilled.
First, for $\bar t = m=1$ we have
\begin{equation*}
    \{ \bar y = \bar t \} = \{ \bar y = 1 \} = \{ (0.5, 0.5) \} \quad \text{and} \quad \{ \bar y < \bar t \} = \Omega \backslash \{ (0.5, 0.5) \}.
\end{equation*}
Moreover, $\nabla \bar y$ obviously vanishes on $\{\bar y = 1\}$, and \cref{lem:structural} in the Appendix shows that the structural assumption in  \eqref{eq:main-hypothesis-y-vanish-structure} holds.
For $ \bar t = m \in (0,1)$, a simple computation shows that
\[
    |\nabla \bar y(x)| > 0 \quad \text{for all} \quad x = (x_1,x_2) \in \{ \bar y = \bar t \},
\]
which validates \eqref{eq:main-hypothesis-nonvanishing-gradient}.
It remains to show the existence of a number $\epsilon>0$ such that the second-order sufficient optimality condition \eqref{eq:2nd-OC-suff-explicit} is fulfilled, provided that $0 \leq 1 -m < \epsilon$.
To this end, by virtue of \eqref{eq:directional-der-a1}, there holds for a.e. $(x_1,x_2) \in \Omega$ that $a'(\bar y(x_1,x_2); s) = \1_{\{ \bar y(x_1,x_2) > \bar t \}}s$ for all $s \in \R$.  From this and the fact that $\1_{ \{\bar y \neq \bar t \}} a'' \equiv 0$, we have
\begin{multline*}
    Q_s(\bar u, \bar y, \bar\varphi;v,v) + Q_1(\bar u, \bar y, \bar\varphi;v,v) \\
    = \frac{1}{2} \norm{S'(\bar u)v}_{L^2(\Omega)}^2 + \frac{\nu}{2} \norm{v}_{L^2(\Omega)}^2 - \int_\Omega \1_{\{ \bar y > \bar t \}} S'(\bar u)v \nabla S'(\bar u)v \cdot \nabla \bar\varphi \dx  \quad \text{for all } v \in L^2(\Omega).
\end{multline*}
Since $\bar\varphi = - \nu \bar u = - 4\nu \pi^2 \bar y$, there hold
\begin{multline*}
    Q_s(\bar u, \bar y, \bar\varphi;v,v) + Q_1(\bar u, \bar y, \bar\varphi;v,v)
    = \frac{1}{2} \norm{S'(\bar u)v}_{L^2(\Omega)}^2 + \frac{\nu}{2} \norm{v}_{L^2(\Omega)}^2 + 4 \nu \pi^2 \int_\Omega \1_{\{ \bar y > \bar t \}} S'(\bar u)v \nabla S'(\bar u)v \cdot \nabla \bar y \dx
\end{multline*}
and
\begin{multline*}
    Q_2(\bar u, \bar y, \bar \varphi; v) =  \frac{1}{2}[a_0'(\bar t) - a_1'(\bar t)] \int_{\{\bar y=\bar t\}} \1_{\{ |\nabla \bar y| > 0 \}} (S'(\bar u)v)^2 \frac{\nabla \bar y \cdot \nabla \bar \varphi}{|\nabla \bar y|} \dH^1{(x)} \\
    = 2\nu \pi^2 \int_{\{\bar y=\bar t\}}  (S'(\bar u)v)^2 |\nabla \bar y| \dH^1{(x)} \geq 0
\end{multline*}
for all $v \in L^2(\Omega)$. Consequently, we have
\begin{equation}
    \label{eq:curvature-exam}
    Q(\bar u, \bar y, \bar\varphi;v) \geq \frac{1}{2} \norm{S'(\bar u)v}_{L^2(\Omega)}^2 + \frac{\nu}{2} \norm{v}_{L^2(\Omega)}^2 + 4 \nu \pi^2 \int_\Omega \1_{\{ \bar y > \bar t \}} S'(\bar u)v \nabla S'(\bar u)v \cdot \nabla \bar y \dx
\end{equation}
for all $v \in L^2(\Omega)$. We now estimate the last term in the right-hand side of \eqref{eq:curvature-exam}. For that purpose, we observe that
\begin{align*}
    \1_{\{ \bar y > \bar t \}}|\nabla \bar y|^2 &=  \pi^2 \1_{\{\sin \pi x_1 \sin \pi x_2 > m \}} (\cos^2 \pi x_1 \sin^2\pi x_2 + \sin^2 \pi x_1 \cos^2 \pi x_2 )\\
    & = \pi^2 \1_{\{\sin \pi x_1 \sin \pi x_2 > m \}} (\sin^2 \pi x_1 + \sin^2\pi x_2 -2 \sin^2 \pi x_1 \sin^2 \pi x_2 ) \\
    & \leq  \pi^2(2-2m^2)
\end{align*}
and thus
\[
    \1_{\{ \bar y > \bar t \}}|\nabla \bar y| \leq \pi \sqrt{2(1-m^2)} \leq 2 \pi \sqrt{1-m}
\]
in $\Omega$, where we have exploited that $1 < 1+m \leq 2$.
From this and the Cauchy--Schwarz inequality, there holds
\begin{equation*}
    \begin{aligned}
        \left| \int_\Omega \1_{\{ \bar y > \bar t \}} S'(\bar u)v \nabla S'(\bar u)v \cdot \nabla \bar y \dx \right| & \leq \norm{\1_{\{ \bar y > \bar t \}}\nabla \bar y}_{L^\infty(\Omega)} \norm{S'(\bar u)v}_{L^2(\Omega)}\norm{\nabla S'(\bar u)v}_{L^2(\Omega)} \\
        & \leq 2\pi C^2 \sqrt{1-m}\norm{v}_{L^2(\Omega)}^2
    \end{aligned}
\end{equation*}
for all $v \in L^2(\Omega)$ and for some constant $C$ independent of $m$ and $v$. Here we have used the fact that $\norm{S'(\bar u)v}_{H^1_0(\Omega)} \leq C\norm{v}_{L^2(\Omega)}$ due to \cref{thm:control2state-oper}. Combining this with \eqref{eq:curvature-exam} yields
\begin{equation*}
    \begin{aligned}
        Q(\bar u, \bar y, \bar\varphi;v) & \geq \frac{1}{2} \norm{S'(\bar u)v}_{L^2(\Omega)}^2 + \frac{\nu}{2} \norm{v}_{L^2(\Omega)}^2 - 8C^2 \nu \pi^3  \sqrt{1-m}\norm{v}_{L^2(\Omega)}^2 \\
        & \geq \frac{1}{2} \norm{S'(\bar u)v}_{L^2(\Omega)}^2 + \frac{\nu}{4} \norm{v}_{L^2(\Omega)}^2 >0
    \end{aligned}
\end{equation*}
for all $v \in L^2(\Omega)$, $v \neq 0$, provided that $0 \leq 1 - m < \epsilon$ with positive constant $\epsilon$ satisfying
\[
    \frac{1}{4} - 8C^2  \pi^3  \sqrt{\epsilon} = 0.
\]
We have therefore verified that all hypotheses of \cref{thm:2nd-OS-suf} and of \cref{thm:error-estimate-L2} are fulfilled. Finally, \cref{lem:Sigma-compu} in the Appendix shows that $\Sigma(\bar y) = 0$ if $\bar t = m =1$ and  $\Sigma(\bar y) > 0$ if $\bar t =m \in (0,1)$.

We now consider the discrete approximation $\bar u_h$ of $\bar u$. Thanks to \cref{rem:projection-simplified,rem:coincidence-minimizer}, we have $\bar u_h = \Proj_{[\alpha,\beta]}(-\frac{1}{\nu} \bar \varphi_h)$ with $\Proj_{[\alpha,\beta]}$ denoting the pointwise a.e. projection mapping onto the interval $[\alpha,\beta]$. From this, \eqref{eq:state-discrete}, and \eqref{eq:OS1st-adjoint-discrete}, $\bar y_h$ and $\bar\varphi_h$ satisfy
\begin{equation}
    \label{eq:OS-discrete-exam}
    \left\{
        \begin{aligned}
            &\int_\Omega ( b(x) + \max\{\bar y_h -m, 0 \} ) \nabla \bar y_h \cdot \nabla w_h \dx = \int_\Omega \Proj_{[\alpha,\beta]}(-\frac{1}{\nu} \bar \varphi_h) w_h \dx,\\
            &\int_\Omega ( b(x) + \max\{\bar y_h -m, 0 \} ) \nabla \bar \varphi_h  \cdot \nabla v_h + \1_{ \{\bar y_h >m\} } \nabla \bar y_h \cdot \nabla \bar \varphi_h v_h \dx = \int_\Omega (\bar y_h - y_d) v_h \dx
        \end{aligned}
    \right.
\end{equation}
for all $w_h,v_h \in V_h$. Since the nonlinearities of \eqref{eq:OS-discrete-exam} are semi-smooth, it is reasonable to solve this system by a semi-smooth Newton (SSN) method; see, e.g., \cite{ItoKunisch2008,Ulbrich2011} as well as \cite{ClasonNhuRosch}. Setting $y_0 := 128 x_1x_2(1-x_1)(1-x_2)$, we notice that the sets $\{ y_0 > \bar t\}$ and $\{ y_0 < \bar t \}$ have positive measures for all $m \in (1-\epsilon,1]$.
The starting point for the discrete SSN method solved \eqref{eq:OS-discrete-exam} is then taken as $(y_{h,0}, \varphi_{h, 0}) := (\Proj_{V_h}(y_0), 0)$ for different mesh sizes $h$, where  $\Proj_{V_h}$ stands for the $L^2$ projection mapping onto $V_h$. The integrals over elements are approximated with a quadrature scheme. This introduces a variational crime which however does not reduce the expected approximation order for piecewise linear functions. In all our tests, the SSN method converged in four or five or six iterations.

We report the resulting discretization errors $\|\bar u_h - \bar u\|_{X}$ for $X= L^2(\Omega)$ and for $h\in \{\frac{\sqrt{2}}{2^n} \mid 4 \leq n \leq 10\}$ as well as the experimental order of convergence
\begin{equation*}
    EOC_{L^p}(n) := \frac{\log(\norm{\bar u - \bar u_{h_{n+1}}}_{L^p(\Omega)}) - \log(\norm{\bar u - \bar u_{h_{n}}}_{L^p(\Omega)})}{\log(h_{n+1}) - \log(h_n)}
\end{equation*}
in \cref{tab:convergence} for both cases $m=1$ and $m=0.95$.
\begin{table}[t]
    \centering
    \begin{tabular}{%
            S[table-format=1.2e-1,scientific-notation=true,round-mode=places,round-precision=2]
            S[table-format=1.2e-1,scientific-notation=true,round-mode=places,round-precision=2]
            S[table-format=1.2e-1,scientific-notation=true,round-mode=places,round-precision=4]
            S[table-format=1.2e-1,scientific-notation=true,round-mode=places,round-precision=2]
            S[table-format=1.2e-1,scientific-notation=true,round-mode=places,round-precision=4]
        }
        \toprule
        & \multicolumn{2}{c}{$m=1$} & \multicolumn{2}{c}{$m=0.95$} \\
        \cmidrule(lr){2-3} \cmidrule(lr){4-5}
        {$h_n$}  & {$\|\bar u - \bar u_{h_n}\|_{L^2(\Omega)}$} &  {$EOC_{L^2}(n)$} &   {$\|\bar u - \bar u_{h_n}\|_{L^2(\Omega)}$} &  {$EOC_{L^2}(n)$} \\
        \midrule
        0.08838834764831845 & 0.18868399994780682 & 1.9842206675459904 & 0.16611404078219316 & 2.30149884431476 \\
        0.04419417382415922 & 0.047689759843375215 & 1.9960412352638899 & 0.03369660566403936 & 1.499478604870065 \\
        0.02209708691207961 & 0.011955200141487647 & 1.9990094098708553 & 0.01191785556126703 & 1.703223336241136 \\
        0.011048543456039806 & 0.002990852924165744 & 1.9997522957092797 & 0.0036599639201972353 & 2.2944527855422274\\
        0.005524271728019903 & 0.0007478416210831529 & 1.9999380703342444 & 0.000746066761816203& 1.5712691960881326\\
        0.0027621358640099515 & 0.0001869684309751765 & 1.9999845171567776 & 0.0002510605828845355 & 2.277237268364241\\
        0.0013810679320049757 & 4.674260937760375e-05& & 5.179189573581347e-05 & \\
        \bottomrule
    \end{tabular}
    \caption{discretization errors and experimental orders of convergence (EOC) in $L^2$ for the optimal control $\bar u$ in dependence of $h_n$}
    \label{tab:convergence}
\end{table}
For the situation $m=1$, the results indicate an EOC of $2$,
which indicates that we are observing a superconvergence property; compare \cite{MeyerRosch:2004}.
The EOC for this case is consistent with the guaranteed rate of $O(h^2)$ shown for the smooth problem in \cite{CasasTroltzsch2012}. An suitable explanation for this could be that in the case $m=1$, we have the following identity
\[
    a(\bar y(x_1,x_2)) = \max\{ \bar y(x_1,x_2) - 1, 0 \} = 0 \quad \text{for all } (x_1,x_2) \in \Omega,
\]
and thus the coefficient $a(y)$ of the state equation in \eqref{eq:P-exam} is in fact smooth at the optimal state.
When $m = 0.95$ (and the coefficient is nonsmooth), the values of EOC are not stable, however their minimum value is approximately equal to $1.5$, which fits the theoretical study shown in \cref{rem:order-convergence}.

\section{Conclusions}
We have studied the numerical approximation of an optimal control problem governed by a quasilinear elliptic equation with nonsmooth coefficient in the divergence part. The convergence of a sequence of  minimizers of some discrete control problems to a global minimizer of the original problem is shown.
A priori error estimates for three types of discretizations  (variational, piecewise constant, and continuous piecewise linear discretizations) are derived under an explicit second-order sufficient condition for the continuous optimal control problem and a structural assumption on the optimal state.
The estimate for variational discretization corrobates the proven rate, although the observed rate is higher, which motivates follow-up work on rate optimality or superconvergence properties for optimal control of nonsmooth quasilinear equations.

\section*{Acknowledgments}
This work was supported by the DFG under the grants CL 487/2-1 and RO 2462/6-1, both within the priority programme SPP 1962 ``Nonsmooth and Complementarity-based Distributed Parameter Systems: Simulation and Hierarchical Optimization''.
Part of this work was completed during a visit of the second author to the Vietnam Institute for Advanced Study in Mathematics (VIASM).
The second author would like to thank the VIASM for their financial support and hospitality.

\appendix

\section{Verification of a structural assumption and computation of the jump functional} \label{sec:structural-jump-func}
\begin{lemma} \label{lem:structural}
    Let $\Omega := \{x = (x_1, x_2) \in \R^2 \mid 0 < x_1, x_2 < 1 \}$ and let $y(x_1,x_2) := \sin(\pi x_1)\sin(\pi x_2)$. Then there exists a constant $c_s>0$ such that
    \begin{equation}
        \label{eq:structural-verification}
        \meas_{\R^2}\left(\{ | y - 1 | < r \} \right) \leq c_s r
    \end{equation}
    for all $r >0$ small enough.
\end{lemma}
\begin{proof}
    Take any sufficiently small $r >0$ satisfying $(1- r, 1) \subset (0,1)$. A simple computation yields
    \begin{equation*}
        \left\{  | y - 1 | < r \right\}  = \left\{ \tfrac{1}{\pi}\arcsin \tfrac{1-r}{\sin \pi x_2} < x_1 <1 - \tfrac{1}{\pi}\arcsin \tfrac{1-r}{\sin \pi x_2} , \tfrac{1}{\pi}\arcsin (1-r) < x_2 <  1- \tfrac{1}{\pi}\arcsin (1- r) \right\}.
    \end{equation*}
    We thus obtain
    \begin{equation} \label{eq:structural-esti}
        \meas_{\R^2}\left(\{ | y - \bar t | < r \} \right)  = \int_{\tfrac{1}{\pi}\arcsin (1-r)}^{1- \tfrac{1}{\pi}\arcsin (1- r)}\left( 1 - \tfrac{2}{\pi}\arcsin \tfrac{1-r}{\sin \pi x_2} \right) dx_2  \leq \left[1 - 2\tfrac{1}{\pi}\arcsin (1- r)\right]^2,
    \end{equation}
    where we have used the fact that
    \[
        0 \leq 1 - \tfrac{2}{\pi}\arcsin \tfrac{1-r}{\sin \pi x_2} \leq 1 - 2\tfrac{1}{\pi}\arcsin (1- r) \quad \text{for all } 0 \leq x_2 \leq 1, r \in (0,1).
    \]
    L'Hospital's rule then shows that
    \begin{equation*}
        \lim\limits_{r \to 0^+} \frac{1 - 2\tfrac{1}{\pi}\arcsin (1- r)}{\sqrt{r}} = \frac{4}{\pi}  \lim\limits_{r \to 0^+}\sqrt{\frac{r}{1- (1-r)^2}} = \frac{2\sqrt{2}}{\pi},
    \end{equation*}
    which, along with \eqref{eq:structural-esti}, yields \eqref{eq:structural-verification}.
\end{proof}

\begin{lemma} \label{lem:Sigma-compu}
    Let $\Omega := \{x = (x_1, x_2) \in \R^2 \mid 0 < x_1, x_2 < 1 \}$ and let $y(x_1,x_2) := \sin(\pi x_1)\sin(\pi x_2)$. Then
    \begin{equation*}
        \Sigma(y) =
        \left\{
            \begin{aligned}
                & 8 \sigma_0\left( 1- \frac{2}{\pi}\arcsin \bar t \right) && \text{if} \quad \bar t \in (0,1),\\
                & 0 && \text{if} \quad \bar t = 1,\\
                & 4 \sigma_0 && \text{if} \quad \bar t = 0.
            \end{aligned}
        \right.
    \end{equation*}
\end{lemma}
\begin{proof}
    We distinguish the following cases:
    \begin{enumerate}[label =(\roman*)]
        \item For any $\bar t \in (0,1)$ and any sufficiently small $r >0$ satisfying $(\bar t - r, \bar t + r) \subset (0,1)$, a simple computation shows that
            \begin{multline*}
                \left\{  | y - \bar t | < r, \partial_{x_1} y > 0 \right\}  \\
                \begin{aligned}
                    & = \left\{ \tfrac{1}{\pi}\arcsin \tfrac{\bar t-r}{\sin \pi x_2} < x_1 < \tfrac{1}{\pi}\arcsin \tfrac{\bar t+r}{\sin \pi x_2} , \tfrac{1}{\pi}\arcsin (\bar t+r) < x_2 <  1- \tfrac{1}{\pi}\arcsin (\bar t+r) \right\}\\
                    \MoveEqLeft[-1]  \cup \left\{ \tfrac{1}{\pi}\arcsin \tfrac{\bar t-r}{\sin \pi x_2} < x_1 < \tfrac{1}{2} , \tfrac{1}{\pi}\arcsin (\bar t-r) < x_2 <  \tfrac{1}{\pi}\arcsin (\bar t+r) \right\}\\
                    \MoveEqLeft[-1]  \cup \left\{ \tfrac{1}{\pi}\arcsin \tfrac{\bar t-r}{\sin \pi x_2} < x_1 < \tfrac{1}{2} , 1- \tfrac{1}{\pi}\arcsin (\bar t+r) < x_2 < 1 -  \tfrac{1}{\pi}\arcsin (\bar t-r) \right\}.
                \end{aligned}
            \end{multline*}
            We thus obtain
            \begin{equation*}
                \begin{aligned}
                    \int_\Omega \1_{\left\{ | y - \bar t | < r, \partial_{x_1} y > 0 \right\}} |\partial_{x_1} y|\dx &=  - \frac{2}{\pi} \left[\cos(\arcsin(\bar t+r)) - \cos(\arcsin(\bar t-r)) \right]   \\
                    \MoveEqLeft[-1] - \frac{2}{\pi}(\bar t-r) \left[  \arcsin(\bar t+r) - \arcsin(\bar t-r)\right] + 2r \left(1 - \frac{2}{\pi} \arcsin(\bar t+r) \right).
                \end{aligned}
            \end{equation*}
            Applying L'Hospital's rule yields
            \begin{equation*}
                \lim\limits_{r \to 0^+} \frac{1}{r} \int_\Omega \1_{\left\{ | y - \bar t | < r, \partial_{x_1} y > 0 \right\}}|\partial_{x_1} y|  \dx = 2 \left(1 - \frac{2}{\pi} \arcsin \bar t \right).
            \end{equation*}
            Similarly, there hold
            \begin{align*}
                & \lim\limits_{r \to 0^+} \frac{1}{r} \int_\Omega \1_{\left\{ | y - \bar t| < r, \partial_{x_1} y < 0 \right\}} |\partial_{x_1} y|\dx = 2 \left(1 - \frac{2}{\pi} \arcsin \bar t \right), \\
                & \lim\limits_{r \to 0^+} \frac{1}{r}\int_\Omega \1_{\left\{ | y - \bar t | < r, \partial_{x_2} y > 0 \right\}} |\partial_{x_2} y|\dx = 2 \left(1 - \frac{2}{\pi} \arcsin \bar t \right),\\
                & \lim\limits_{r \to 0^+} \frac{1}{r}\int_\Omega \1_{\left\{| y - \bar t | < r, \partial_{x_2} y < 0 \right\}} |\partial_{x_2} y|\dx = 2 \left(1 - \frac{2}{\pi} \arcsin \bar t \right).
            \end{align*}
            By adding these four limits, we obtain
            \begin{equation*}
                \lim\limits_{r \to 0^+} \frac{1}{r} \int_\Omega \1_{\left\{ | y - \bar t | < r \right\}} \left[|\partial_{x_1} y| +|\partial_{x_2} y| \right]\dx = 8 \left(1 - \frac{2}{\pi} \arcsin \bar t \right).
            \end{equation*}

        \item For $\bar t =0$ and for any $r\in (0,1)$ sufficiently small, we see from a straightforward calculation that
            \begin{equation*}
                \begin{aligned}
                    \left\{  | y - 0 | < r, \partial_{x_1} y > 0 \right\}
                    & = \left\{ 0 < x_1 < \tfrac{1}{\pi}\arcsin \tfrac{r}{\sin \pi x_2} , \tfrac{1}{\pi}\arcsin r < x_2 <  1- \tfrac{1}{\pi}\arcsin r \right\}\\
                    \MoveEqLeft[-1]  \cup \left\{ 0 < x_1 < \tfrac{1}{2} , 0 < x_2 <  \tfrac{1}{\pi}\arcsin r \right\}  \cup \left\{ 0 < x_1 < \tfrac{1}{2} , 1- \tfrac{1}{\pi}\arcsin r < x_2 < 1  \right\}.
                \end{aligned}
            \end{equation*}
            Consequently, it holds that
            \begin{equation*}
                \int_\Omega \1_{\left\{ | y - 0 | < r, \partial_{x_1} y > 0 \right\}} |\partial_{x_1} y|\dx  =  - \frac{2}{\pi} \left[\cos(\arcsin r) - 1 \right]  + r \left(1 - \frac{2}{\pi} \arcsin r \right).
            \end{equation*}
            L'Hospital's rule then shows that
            \begin{equation*}
                \lim\limits_{r \to 0^+} \frac{1}{r} \int_\Omega \1_{\left\{ | y - 0 | < r, \partial_{x_1} y > 0 \right\}}|\partial_{x_1} y|  \dx = 1.
            \end{equation*}
            Similarly, we can conclude that
            \begin{equation*}
                \lim\limits_{r \to 0^+} \frac{1}{r} \int_\Omega \1_{\left\{ | y - 0 | < r \right\}} \left[|\partial_{x_1} y| +|\partial_{x_2} y| \right]\dx = 4.
            \end{equation*}

        \item For $\bar t =1$ and for any $r>0$ such that $r \in (0,1)$, we have
            \begin{equation*}
                \left\{  | y - 1 | < r, \partial_{x_1} y > 0 \right\}
                = \left\{ \tfrac{1}{\pi}\arcsin \tfrac{1-r}{\sin \pi x_2} < x_1 < \tfrac{1}{2}, \tfrac{1}{\pi}\arcsin (1-r) < x_2 <  1- \tfrac{1}{\pi}\arcsin (1-r) \right\}.
            \end{equation*}
            There therefore holds that
            \begin{equation*}
                \int_\Omega \1_{\left\{ | y - 1 | < r, \partial_{x_1} y > 0 \right\}} |\partial_{x_1} y|\dx  =  \frac{2}{\pi} \left[\cos(\arcsin (1-r)) + (1-r)\arcsin(1-r) \right]  +  r -1.
            \end{equation*}
            Again, L'Hospital's rule shows that
            \begin{equation*}
                \lim\limits_{r \to 0^+} \frac{1}{r} \int_\Omega \1_{\left\{ | y - 1 | < r, \partial_{x_1} y > 0 \right\}}|\partial_{x_1} y|  \dx = 0.
            \end{equation*}
            Similarly, we can deduce that
            \begin{equation*}
                \lim\limits_{r \to 0^+} \frac{1}{r} \int_\Omega \1_{\left\{ | y - 1 | < r \right\}} \left[|\partial_{x_1} y| +|\partial_{x_2} y| \right]\dx = 0.
                \qedhere
            \end{equation*}
    \end{enumerate}
\end{proof}

\section{Regularity of solutions to the adjusted linearized state equation} \label{sec:linearized-regularity}
Let $p_0, p_1$, and $\gamma$ be given as in \eqref{eq:gamma-exponent}. In order to show the $W^{1+\gamma,2}$-regularity of solutions to the adjusted linearized state equation \eqref{eq:adjusted-deri}, we need the following result for the function
\begin{equation*} \label{eq:fu-func}
    \mathbb{f}_{\bar u} := \1_{\{ \bar y \neq \bar t \}} a'(\bar y) \nabla \bar y,
\end{equation*}
where $\bar y := S(\bar u)$ is the optimal state corresponding to the control $\bar u$.

\begin{proposition}
    \label{prop:fu-regularity-origin}
    Assume that $\bar u \in L^\infty(\Omega)$ satisfies $\Sigma(\bar y) < \infty$. Then there holds
    \begin{equation*}
        \label{eq:fu-regularity-baru}
        \mathbb{f}_{\bar u} \in (W^{\gamma,p_0}(\Omega))^2 = (W^{\gamma,p_0}_0(\Omega))^2.
    \end{equation*}
\end{proposition}
\begin{proof}
    Since $\gamma < \frac{1}{p_0}$, there holds that
    \[
        W^{\gamma,p_0}(\Omega) = W^{\gamma,p_0}_0(\Omega);
    \]
    see, e.g. Corollary 1.4.4.5 in \cite{Grisvard1985}.
    Thanks to \cref{thm:control2state-oper}, one has $\bar y \in W^{2, p_1}(\Omega) \hookrightarrow C^1(\overline\Omega)$ and thus $\bar y \in C^1(\overline\Omega)$. Therefore, we have $\mathbb{f}_{\bar u} \in L^\infty(\Omega)^2 \hookrightarrow L^{p_0}(\Omega)^2$.
    From the definition of the Sobolev spaces of fractional order; see e.g. Definition 6.8.2 in \cite{Kufner}, it then suffices to show that
    \begin{equation*}
        \iint_{\Omega \times \Omega} \frac{|\mathbb{f}_{\bar u}(x) - \mathbb{f}_{\bar u}(\hat x)|^{p_0} }{|x - \hat x|^{2 + p_0 \gamma}} dxd\hat x < \infty,
    \end{equation*}
    or equivalently,
    \begin{equation}
        \label{eq:fu-exponent-baru}
        \iint_{\Omega \times \Omega \cap \{|x - \hat x | < \delta/{L} \}} \frac{|\mathbb{f}_{\bar u}(x) - \mathbb{f}_{\bar u}(\hat x)|^{p_0} }{|x - \hat x|^{2 + p_0 \gamma}} dxd\hat x < \infty.
    \end{equation}
    Here $\delta$ is the constant in \eqref{eq:Omega-123-sets} and $L$ denotes the Lipschitz constant of $\bar y$.
    For any $x, \hat x \in \Omega$ satisfying $|x - \hat x| < \frac{\delta}{L}$, there holds that
    \begin{equation}
        \label{eq:Lipschitz-bary}
        |\bar y(x) - \bar y(\hat x)| \leq L|x-\hat x| < \delta.
    \end{equation}
    Setting now
    \begin{equation}
        \label{eq:Z-baru}
        Z(x,\hat x) := \mathbb{f}_{\bar u}(x) - \mathbb{f}_{\bar u}(\hat x) = \1_{\{ \bar y \neq \bar t \}}(x) a'(\bar y(x)) \nabla \bar y(x) - \1_{\{ \bar y \neq \bar t \}}(\hat x) a'(\bar y(\hat x)) \nabla \bar y(\hat x)
    \end{equation}
    and using the argument in the proof of \cref{lem:Z-decomposition} for the situation $y(x) := \bar y(x)$ and $ \hat y(x) := \bar y(\hat x)$, we can decompose $Z(x, \hat x)$ as follows
    \begin{equation}
        \label{eq:Z-baru-decom}
        Z(x,\hat x) = Z^{(1)}(x,\hat x) + Z^{(2)}(x,\hat x) + Z^{(3)}(x,\hat x) + Z^{(4)}(x,\hat x),
    \end{equation}
    where
    \begin{align}
        |Z^{(1)}(x,\hat x)| + |Z^{(2)}(x,\hat x)| + |Z^{(4)}(x,\hat x)|&\leq C [|\bar y(x) - \bar y(\hat x)| |\nabla \bar y(\hat x) | + | \nabla (\bar y(x) - \bar y(\hat x))| ], \label{eq:Z124-baru} \\
        Z^{(3)}(x,\hat x) &= [a_0'(\bar t) - a_1'(\bar t)][ \chi_2(x,\hat x) - \chi_3(x,\hat x) ] \nabla \bar y(\hat x), \label{eq:Z3-baru}
    \end{align}
    and
    \begin{align*}
        \chi_2(x,\hat x) & = \begin{cases}
            1 & \text{if} \quad \bar y(\hat x) \in (\bar t, \bar t + \delta), \bar y(x) \in (\bar t- \delta, \bar t],\\
            0 & \text{otherwise},
        \end{cases} \\
        \chi_3(x,\hat x) & = \begin{cases}
            1 & \text{if} \quad \bar y(\hat x) \in (\bar t-\delta, \bar t ), \bar y(x) \in [\bar t,  \bar t + \delta), \\
            0 & \text{otherwise}.
        \end{cases}
    \end{align*}
    Here
    we have compared the function $Z(x,\hat x)$ with $Z_{y, \hat y}$ defined in \eqref{eq:Z-func}; compared the estimates for $Z^{(i)}(x,\hat x)$, $i =1,2,4$, with the ones in \eqref{eq:Z-esti}; compared the decomposition \eqref{eq:Z-baru-decom} with \eqref{eq:Z-decomposition}; and compared the functions $\chi_2, \chi_3$ with $\1_{\Omega^2_{y, \hat y}}$, $\1_{\Omega^3_{y, \hat y}}$, respectively, for the sets $\Omega^2_{y, \hat y}, \Omega^3_{y, \hat y}$ given in \eqref{eq:Omega-123-sets}.
    Applying \eqref{eq:Lipschitz-bary} then yields
    \begin{align*}
        \iint_{\Omega \times \Omega \cap \{|x - \hat x | < \delta/{L} \}} \frac{|\bar y(x) - \bar y(\hat x)|^{p_0} }{|x - \hat x|^{2 + p_0 \gamma}} dxd\hat x & \leq \iint_{\Omega \times \Omega \cap \{|x - \hat x | < \delta/{L} \}} L^{p_0} |x - \hat x|^{-2 + p_0(1-\gamma)} dxd\hat x < \infty,
    \end{align*}
    due to the fact that $2 - p_0(1 -\gamma) < 2 = N$; see e.g. page 331 in \cite{Kufner}.
    Moreover, from the Sobolev embedding $W^{1,p_1}(\Omega) \hookrightarrow C^{0, 1- \frac{2}{p_1}}(\overline\Omega)$, there holds
    $\nabla \bar y \in C^{0, 1- \frac{2}{p_1}}(\overline\Omega)^2$. From this, we have
    \begin{align*}
        \iint_{\Omega \times \Omega \cap \{|x - \hat x | < \delta/{L} \}} \frac{|\nabla \bar y(x) -\nabla \bar y(\hat x)|^{p_0} }{|x - \hat x|^{2 + p_0 \gamma}} dxd\hat x & \leq \iint_{\Omega \times \Omega \cap \{|x - \hat x | < \delta/{L} \}} C^{p_0} |x - \hat x|^{-2 + p_0(1 - \frac{2}{p_1}-\gamma)} dxd\hat x < \infty,
    \end{align*}
    in the view of \eqref{eq:gamma-exponent}. By combining the last two integral estimates with \eqref{eq:Z124-baru} and using the fact that $\nabla \bar y \in C(\overline\Omega)^2$, there holds
    \begin{equation}
        \label{eq:Z124-baru-exponent}
        \iint_{\Omega \times \Omega \cap \{|x - \hat x | < \delta/{L} \}} \frac{|Z^{(i)}(x, \hat x)|^{p_0} }{|x - \hat x|^{2 + p_0 \gamma}} dxd\hat x < \infty \quad \text{for } i =1,2,4.
    \end{equation}

    We now estimate $Z^{(3)}(x, \hat x)$ defined in \eqref{eq:Z3-baru}. To this end, we deduce from the definition of $\chi_2(x, \hat x)$ and $\chi_3(x,\hat x)$ that
    \begin{align*}
        |\chi_2(x, \hat x) - \chi_3(x, \hat x)| & \leq
        \begin{cases}
            1 & \text{if} \quad0 < | \bar y(\hat x) - \bar t | \leq |\bar y(\hat x) - \bar y(x) |,\\
            0 & \text{otherwise}
        \end{cases}
        \\
        & \leq \begin{cases}
            1 & \text{if} \quad 0 < | \bar y(\hat x) - \bar t | \leq L|x - \hat x|,\\
            0 & \text{otherwise}.
        \end{cases}
    \end{align*}
    Combining the last estimate with \eqref{eq:Z3-baru} yields
    \begin{align*}
        |Z^{(3)}(x, \hat x)|^{p_0} &\leq C \sigma_0^{p_0} \1_{\{ 0 < |\bar y(\cdot) - \bar t| < L|\cdot - x| \}}(\hat x) |\nabla \bar y (\hat x)|^{p_0} \\
        & \leq C(p_0) \sigma_0^{p_0} \1_{\{ 0 < |\bar y(\cdot) - \bar t| < L|\cdot - x| \}}(\hat x) |\nabla \bar y (\hat x)|,
    \end{align*}
    thanks to the boundedness of $|\nabla \bar y|$ in $C(\overline{\Omega})$.
    On the other hand, we conclude from the definition of $\Sigma(\bar y)$ in \eqref{eq:E-functional} that there exists a constant $r_0 >0$ such that
    \[
        \frac{\sigma_0}{r}\int_{\Omega} \1_{\{ 0 < |\bar y(\cdot) - \bar t| < r \}}(\hat x) |\nabla \bar y (\hat x)| d\hat x  \leq \Sigma(\bar y) + 1 \quad \text{for all } 0 < r < r_0.
    \]
    We then have
    \begin{multline*}
        \iint_{\Omega \times \Omega \cap \{|x - \hat x | < \delta/{L} \}} \frac{|Z^{(3)}(x, \hat x)|^{p_0} }{|x - \hat x|^{2 + p_0 \gamma}} dxd\hat x \\
        \begin{aligned}
            & = \iint_{\Omega \times \Omega \cap \{r_0/L \leq |x - \hat x | < \delta/{L} \}} \frac{|Z^{(3)}(x, \hat x)|^{p_0} }{|x - \hat x|^{2 + p_0 \gamma}} dxd\hat x + \iint_{\Omega \times \Omega \cap \{|x - \hat x | < r_0/{L} \}} \frac{|Z^{(3)}(x, \hat x)|^{p_0} }{|x - \hat x|^{2 + p_0 \gamma}} dxd\hat x \\
            & \leq C \iint_{\Omega \times \Omega \cap \{r_0/L \leq |x - \hat x | < \delta/{L} \}} \frac{1 }{|x - \hat x|^{2 + p_0 \gamma}} dxd\hat x \\
            & \qquad  +  C \sigma_0^{p_0} \iint_{\Omega \times \Omega \cap \{|x - \hat x | < r_0/{L} \}} \frac{1}{|x - \hat x|^{2 + p_0 \gamma}}\1_{\{ 0 < |\bar y(\cdot) - \bar t| < L|\cdot - x| \}}(\hat x) |\nabla \bar y (\hat x)| d\hat x dx \\
            & \leq C(r_0,\delta) + C \sigma_0^{p_0} \int_{\{\xi \in \R^2: |\xi| < \frac{r_0}{L}\} } \frac{1}{|\xi|^{2 + p_0 \gamma}} d\xi  \int_{\Omega} \1_{\{ 0 < |\bar y(\cdot) - \bar t| < L|\xi| \}}(\hat x) |\nabla \bar y (\hat x)| d\hat x \\
            & = C(r_0,\delta) + CL \sigma_0^{p_0-1} \int_{\{\xi \in \R^2: |\xi| < \frac{r_0}{L}\} } \frac{d\xi}{|\xi|^{1 + p_0 \gamma}}  \frac{\sigma_0}{L|\xi|}   \int_{\Omega} \1_{\{ 0 < |\bar y(\cdot) - \bar t| < L|\xi| \}}(\hat x) |\nabla \bar y (\hat x)| d\hat x \\
            & \leq C(r_0,\delta) + CL \sigma_0^{p_0-1} \frac{2\pi}{1- p_0\gamma} \left(\frac{r_0}{L}\right)^{1- p_0 \gamma} (\Sigma(\bar y) + 1)\\
            & < \infty,
        \end{aligned}
    \end{multline*}
    due to \eqref{eq:gamma-exponent}. From this, \eqref{eq:Z124-baru-exponent}, \eqref{eq:Z-baru-decom}, and \eqref{eq:Z-baru}, we have \eqref{eq:fu-exponent-baru}.
\end{proof}

\begin{lemma}
    \label{lem:Sobolev-product}
    Assume that $\bar u \in L^\infty(\Omega)$ satisfies $\Sigma(\bar y) < \infty$.
    Then for any $z \in H^1(\Omega)$, there holds
    \[
        \dive (\mathbb{f}_{\bar u} z) \in W^{-1 + \gamma,2}(\Omega).
    \]
    Moreover,
    \begin{equation}
        \label{eq:Sobolev-product-esti}
        \norm{\dive( \mathbb{f}_{\bar u} z)}_{W^{-1 + \gamma,2}(\Omega)} \leq C \norm{\mathbb{f}_{\bar u} }_{ W^{\gamma,p_0}(\Omega)}  \norm{z}_{H^1(\Omega)}
    \end{equation}
    for some constant $C$ independent of $z$.
\end{lemma}
\begin{proof}
    By \cref{prop:fu-regularity-origin}, one has $\mathbb{f}_{\bar u} \in (W^{\gamma,p_0}(\Omega))^2$. We now apply the multiplication theorem for Sobolev spaces; see, e.g. Theorem 1.4.4.2 in \cite{Grisvard1985} and Theorem 7.4 in \cite{BehzadanHolst2021} for $s_1 = \gamma$, $s_2 =1$, $s=\gamma$,  $p_1 = p_0 >2$, $p_2 =2$, $p=2$, and $n=2$ to obtain
    \[
        \mathbb{f}_{\bar u} z \in (W^{\gamma,2}(\Omega))^2.
    \]
    From this and the continuity of the divergence operator  from $(W^{\gamma,2}(\Omega))^2$ to $W^{-1+\gamma,2}(\Omega)$ for $\gamma \neq \frac{1}{2}$; see, e.g. Theorem 1.4.4.6 in \cite{Grisvard1985}, we have
    $\dive (\mathbb{f}_{\bar u} z) \in W^{-1 + \gamma,2}(\Omega)$. Finally, we also have \eqref{eq:Sobolev-product-esti} from Theorem 7.4 in \cite{BehzadanHolst2021} and Theorem 1.4.4.6 in \cite{Grisvard1985}.
\end{proof}

\begin{proposition}
    \label{prop:adjusted-deri-regularity}
    The following assertions hold:
    \begin{enumerate}[label=(\roman*)]
        \item \label{item:zuv-W1p-esti} Let $p >2$ be fixed but arbitrary.
            Then, for any $u \in L^2(\Omega)$ and $v \in W^{-1,p}(\Omega)$, \eqref{eq:adjusted-deri} admits a unique solution $\tilde{z}_{u,v} \in W^{1,p}_0(\Omega)$.
            Furthermore, there exists a constant $\rho_p >0$ such that
            \begin{equation}
                \label{eq:zuv-W1p-esti}
                \norm{\tilde z_{u,v}}_{W^{1,p}_0(\Omega)} \leq C \norm{v}_{W^{-1,p}(\Omega)} \quad \text{for all } u \in \overline{B}_{L^2(\Omega)}(\bar u, \rho_p)
            \end{equation}
            and for some constant $C$ independent of $u$ and $v$.
        \item \label{item:zuv-Wgamma-esti}
            Assume that $\bar u \in L^\infty(\Omega)$ satisfies $\Sigma(\bar y) < \infty$.
            Then, for any $u \in \mathcal{U}_{ad}$ and any $ v \in W^{-1+\gamma,2}(\Omega)$, the equation \eqref{eq:adjusted-deri} admits a unique solution $\tilde z_{u,v}$ in $W^{1+\gamma,2}(\Omega)$.
            Moreover, there exists a constant $\hat{\rho}$ such that
            \begin{equation} \label{eq:aprior-adjusted-deri}
                \norm{\tilde z_{u,v}}_{W^{1+\gamma,2}(\Omega)} \leq C (1+ \norm{ \mathbb{f}_{\bar u} }_{W^{\gamma,p_0}(\Omega)})\norm{v}_{W^{-1+\gamma,2}(\Omega)}
            \end{equation}
            for all $u \in \overline{B}_{L^2(\Omega)}(\bar u, \hat\rho) \cap \mathcal{U}_{ad}$ and
            for some constant $C$ independent of $u$ and $v$.
    \end{enumerate}
\end{proposition}
\begin{proof}
    \noindent\textit{Ad \ref{item:zuv-W1p-esti}:}
    For any $u \in L^2(\Omega) \hookrightarrow W^{-1,p}(\Omega)$, by \cref{thm:control2state-oper}, we have $y_u = S(u) \in W^{1,p}_0(\Omega)$. From this, and Remark 2.9 in \cite{CasasTroltzsch2009}, we deduce the existence and uniqueness of $\tilde z_{u,v}$ in $H^1_0(\Omega)$ to \eqref{eq:adjusted-deri}.

    We now rewrite \eqref{eq:adjusted-deri} as follows
    \begin{equation*}
        \left\{
            \begin{aligned}
                -\dive [(b + a(y_u)) \nabla \tilde z_{u,v} + \1_{\{ y_u \neq \bar t \}}a'(y_u) \tilde z_{u,v} \nabla y_u ]  &= v - \dive( Z_{y_u, \bar y}\tilde  z_{u,v})&& \text{in } \Omega,\\
                \tilde z_{u,v} &=0 && \text{on } \partial\Omega,
            \end{aligned}
        \right.
    \end{equation*}
    where $Z_{y_u, \bar y}$ is defined as in \eqref{eq:Z-func}.
    In other words, one has
    \begin{equation}
        \label{eq:zuv-deri-form}
        \tilde z_{u,v} = S'(u)(v - \dive( Z_{y_u, \bar y}\tilde  z_{u,v})),
    \end{equation}
    due to \cref{thm:control2state-oper}.
    By \cref{lem:Z-decomposition}, we have that $Z_{y_u, \bar y} \in L^p(\Omega)^2$. Since $\tilde{z}_{u,v} \in H^1_0(\Omega) \hookrightarrow L^{q}(\Omega)$ with $q > \frac{2p}{p-2}$, we then deduce that $Z_{y_u, \bar y}\tilde  z_{u,v} \in L^{r}(\Omega)^2$ with $r = \frac{pq}{p+q} \in (2,p)$.
    There therefore holds
    \[
        v - \dive( Z_{y_u, \bar y}\tilde  z_{u,v}) \in W^{-1,r}(\Omega).
    \]
    Applying \cref{thm:control2state-oper} yields $\tilde{z}_{u,v} \in W^{1,r}_0(\Omega) \hookrightarrow C(\overline\Omega)$. We then have
    $Z_{y_u, \bar y}\tilde  z_{u,v} \in L^{p}(\Omega)^2$ and thus
    \[
        v - \dive( Z_{y_u, \bar y}\tilde  z_{u,v}) \in W^{-1,p}(\Omega).
    \]
    This, together with \cref{thm:control2state-oper}, gives $\tilde{z}_{u,v} \in W^{1,p}_0(\Omega)$.

    We now prove \eqref{eq:zuv-W1p-esti}. To this end, by \cref{thm:control2state-oper} and the compact embedding $L^2(\Omega) \Subset W^{-1,p}(\Omega)$, we first have
    \begin{equation*}
        \sup\{ \norm{S'(u)}_{\Linop(W^{-1,p}(\Omega), W^{1,p}_0(\Omega) )} \mid u \in \overline B_{L^2(\Omega)}(\bar u, 1) \}  \leq C_1,
    \end{equation*}
    which gives for all $u \in \overline B_{L^2(\Omega)}(\bar u, 1)$ that
    \begin{multline}
        \label{eq:zuv-W1p0}
        \norm{\tilde z_{u,v}}_{W^{1,p}_0(\Omega)}  \leq C_{1} \norm{ v - \dive( Z_{y_u, \bar y} \tilde z_{u,v})}_{W^{-1,p}(\Omega)}  \leq C_{1} [\norm{ v }_{W^{-1,p}(\Omega)} + \norm{  Z_{y_u, \bar y} \tilde z_{u,v}}_{L^{p}(\Omega)} ] \\
        \begin{aligned}[t]
            & \leq C_{1} [\norm{ v }_{W^{-1,p}(\Omega)} + \norm{\tilde z_{u,v} }_{L^{2p/(p-2)}(\Omega)} \norm{Z_{y_u, \bar y} }_{L^2(\Omega)} ] \\
            & \leq C_{1} [ \norm{ v }_{W^{-1,p}(\Omega)} + \norm{\tilde z_{u,v} }_{W^{1,p}_0(\Omega)}\norm{Z_{y_u, \bar y} }_{L^2(\Omega)} ],
        \end{aligned}
    \end{multline}
    where we have just used the H\"{o}lder inequality and the embeddings $W^{1,p}_0(\Omega) \hookrightarrow H^1_0(\Omega) \hookrightarrow L^{2p/{p-2}}(\Omega)$.
    From \cref{lem:Z-decomposition} and the fact that $\norm{y_u - \bar y}_{W^{1,p}_0(\Omega)} \to 0$ as $\norm{u - \bar u}_{W^{-1,p}(\Omega)} \leq C \norm{u - \bar u}_{L^2(\Omega)} \to 0$ (see \cref{thm:control2state-oper}), we have
    \[
        \norm{Z_{y_u, \bar y} }_{L^2(\Omega)} \to 0 \quad \text{as} \quad \norm{u - \bar u}_{L^2(\Omega)} \to 0.
    \]
    Then, there exists a  constant $\rho_p \in (0,1]$ such that
    \[
        \norm{Z_{y_u, \bar y} }_{L^2(\Omega)} \leq \frac{1}{2C_1} \quad \text{for all } u \in \overline{B}_{L^2(\Omega)}(\bar u, \rho_p).
    \]
    Combining this with \eqref{eq:zuv-W1p0} yields
    \begin{equation*}
        \norm{\tilde z_{u,v}}_{W^{1,p}_0(\Omega)}  \leq 2C_{1} \norm{ v }_{W^{-1,p}(\Omega)} \quad \text{for all } u \in \overline{B}_{L^2(\Omega)}(\bar u, \rho_p).
    \end{equation*}

    \noindent\textit{Ad \ref{item:zuv-Wgamma-esti}:}
    We first show the $W^{1+\gamma,2}$-regularity of $\tilde z_{u,v}$. To this end, we now rewrite \eqref{eq:adjusted-deri} as follows
    \begin{equation*}
        \left\{
            \begin{aligned}
                -\dive [(b + a(y_u)) \nabla \tilde z_{u,v}  ]  &= v + \dive [ \mathbb{f}_{\bar u}  \tilde z_{u,v}]&& \text{in } \Omega,\\
                \tilde z_{u,v} &=0 && \text{on } \partial\Omega
            \end{aligned}
        \right.
    \end{equation*}
    or, equivalently,
    \begin{equation} \label{eq:adjusted-deri-fu-Delta}
        \left\{
            \begin{aligned}
                -\Delta \tilde z_{u,v}  &=\frac{1}{b + a(y_u)}  (\nabla b + \1_{\{y_u \neq \bar t \}} a'(y_u) \nabla y_u) \cdot \nabla \tilde z_{u,v} + \frac{1}{b + a(y_u)}  [v+\dive ( \mathbb{f}_{\bar u}  \tilde z_{u,v}) ]&& \text{in } \Omega,\\
                \tilde z_{u,v} &=0 && \text{on } \partial\Omega.
            \end{aligned}
        \right.
    \end{equation}
    From \cref{prop:fu-regularity-origin} and \cref{lem:Sobolev-product}, there holds $\dive [ \mathbb{f}_{\bar u}  \tilde z_{u,v}] \in W^{-1 +\gamma,2}(\Omega)$ and thus
    \begin{equation}
        \label{eq:v-fu-Wgamma}
        v + \dive [ \mathbb{f}_{\bar u}  \tilde z_{u,v}]\in W^{-1 +\gamma,2}(\Omega).
    \end{equation}
    Since $\mathcal{U}_{ad}$ is bounded in $L^\infty(\Omega)$ and thus in $L^{p_0}(\Omega)$, we deduce from \cref{thm:control2state-oper} that
    $y_u, \bar y \in W^{2,p_0}(\Omega) \hookrightarrow C^1(\overline{\Omega})$ and then that
    \begin{equation*}
        \label{eq:W2p-estimate}
        \norm{y_u}_{W^{2,p}(\Omega)} + \norm{y_u}_{C^1(\overline{\Omega})} \leq M \quad \text{for all } u \in \mathcal{U}_{ad}
    \end{equation*}
    for some constant $M>0$.
    Here $p_0$ is the constant in \eqref{eq:gamma-exponent}.
    Thanks to \cref{ass:b_func} and \cref{ass:PC1-func}, one has $b + a(y_u), (b + a(y_u))^{-1} \in W^{1,\infty}(\Omega)$ and
    \begin{equation}
        \label{eq:yu-W1infty-esti}
        \norm{b+a(y_u)}_{W^{1,\infty}(\Omega)} + \norm{(b+a(y_u))^{-1}}_{W^{1,\infty}(\Omega)}\leq C_M \quad \text{for all } u \in \mathcal{U}_{ad}.
    \end{equation}
    Combining this with \eqref{eq:v-fu-Wgamma} yields
    \[
        \frac{1}{b+a(y_u)} [v + \dive [ \mathbb{f}_{\bar u}  \tilde z_{u,v}]] \in W^{-1 +\gamma,2}(\Omega),
    \]
    where we have employed the fact that the multiplication bilinear mapping $W^{1, \infty}(\Omega) \times W^{-1 +\gamma,2}(\Omega) \to W^{-1 +\gamma,2}(\Omega)$ is continuous; see, e.g. \cite{BehzadanHolst2021}.
    The second term in the right hand side of \eqref{eq:adjusted-deri-fu-Delta} then belongs to $W^{-1 +\gamma,2}(\Omega)$. Moreover, the first one also belongs to $W^{-1 +\gamma,2}(\Omega)$, since it is in $L^2(\Omega) \hookrightarrow W^{-1 +\gamma,2}(\Omega)$.
    We therefore apply the $W^{1+\gamma,2}$-regularity of solutions to \eqref{eq:adjusted-deri-fu-Delta} to have that
    $\tilde z_{u,v} \in W^{1+\gamma,2}(\Omega)$; see, e.g. Theorem 3 in \cite{Savare1998} and Theorem 4.1 in  \cite{Clain1997}.

    It remains to show the existence of a constant $\hat{\rho} >0$ satisfying \eqref{eq:aprior-adjusted-deri}. To this end, we first see from \eqref{eq:gamma-exponent} that $1 >\frac{2}{p_0}\geq 1 - \gamma$ and $1 < \frac{p_0}{p_0-1} < 2$. From this and the Sobolev embedding theorem; see e.g. Theorem 3.8 in \cite{BehzadanHolst2021}, we deduce that
    $W^{1, \frac{p_0}{p_0-1}}_0(\Omega) \hookrightarrow W^{ \frac{2}{p_0},2}_0(\Omega) \hookrightarrow W^{1- \gamma,2}_0(\Omega)$ and thus
    \begin{equation*}
        \label{eq:embedding-gamma-p0}
        W^{-1+ \gamma,2}(\Omega) \hookrightarrow W^{-1,p_0}(\Omega).
    \end{equation*}
    We now apply assertion \ref{item:zuv-W1p-esti} for $p := p_0$ to derive that  $\tilde{z}_{u,v} \in W^{1,p_0}_0(\Omega)$
    and that
    \begin{equation}
        \label{eq:zuv-Wgamma-esti-p0}
        \norm{\tilde z_{u,v}}_{W^{1,p_0}_0(\Omega)} \leq C \norm{v}_{W^{-1,p_0}(\Omega)} \leq C \norm{v}_{  W^{-1+ \gamma,2}(\Omega) } \quad \text{for all } u \in \overline{B}_{L^2(\Omega)}(\bar u, \rho_{p_0}) \cap \mathcal{U}_{ad}.
    \end{equation}
    Moreover,
    by applying Theorem 3 in \cite{Savare1998} (see, also Theorem 4.1 in  \cite{Clain1997}) to \eqref{eq:adjusted-deri-fu-Delta}, and using estimates \eqref{eq:Sobolev-product-esti} and \eqref{eq:yu-W1infty-esti}, as well as the embedding $W^{1,p_0}_0(\Omega) \hookrightarrow H^1_0(\Omega)$,
    there is a constant $C= C(M)$ such that
    \begin{equation*}
        \label{eq:apriori-zuv}
        \begin{aligned}
            \norm{\tilde z_{u,v}}_{W^{1+\gamma,2}(\Omega)} &\leq C(M) [\norm{v}_{W^{-1+\gamma,2}(\Omega)} + (1+ \norm{ \mathbb{f}_{\bar u} }_{W^{\gamma,p_0}(\Omega)}) \norm{\tilde z_{u,v}}_{H^1_0(\Omega)}]  \\
            &\leq C(M) [\norm{v}_{W^{-1+\gamma,2}(\Omega)} + (1+ \norm{ \mathbb{f}_{\bar u} }_{W^{\gamma,p_0}(\Omega)}) \norm{\tilde z_{u,v}}_{W^{1,p_0}_0(\Omega)}   ]
            \quad \text{for all } u \in \mathcal{U}_{ad}.
        \end{aligned}
    \end{equation*}
    Setting now $\hat{\rho} :=  \rho_{p_0}$ and combining the last inequality with \eqref{eq:zuv-Wgamma-esti-p0} yields \eqref{eq:aprior-adjusted-deri}.
\end{proof}

\bibliographystyle{jnsao}
\bibliography{QuasilinearFEM_II}

\end{document}